\documentclass[11pt]{amsart}
\usepackage{mathtools}
\usepackage{amssymb,latexsym,esint,cite,mathrsfs}
\usepackage{verbatim,wasysym}
\usepackage[left=2.6cm,right=2.6cm,top=2.9cm,bottom=2.9cm]{geometry}
\usepackage[T1]{fontenc}
\usepackage{lmodern}
\usepackage[colorlinks=true,urlcolor=blue, citecolor=red,linkcolor=blue,
linktocpage,pdfpagelabels, bookmarksnumbered,bookmarksopen]{hyperref}
\usepackage{tikz,enumitem,graphicx, subfig, microtype, color}
\usepackage{epic,eepic}
\usepackage[english]{babel}
\usepackage{xcolor}

\allowdisplaybreaks

\renewcommand{\epsilon}{\varepsilon}

\numberwithin{equation}{section}
\newtheorem{theorem}{Theorem}[section]
\newtheorem{proposition}[theorem]{Proposition}
\newtheorem{lemma}[theorem]{Lemma}

\theoremstyle{definition}

\newcommand{\R}{\mathbb{R}}

\begin{document}

\title[Multiple blowing-up solutions for Lane-Emden system ]
{Multiple blowing-up solutions for a slightly critical Lane-Emden system with non-power nonlinearity}

\author[S.\ Deng]{Shengbing Deng}
\author[F.\ Yu]{Fang Yu}

\address[S.\ Deng]{School of Mathematics and Statistics \newline\indent
Southwest University \newline\indent
Chongqing 400715, People's Republic of China.}
\email{shbdeng@swu.edu.cn}

\address[F.\ Yu]{School of Mathematics and Statistics \newline\indent
Southwest University \newline\indent
Chongqing 400715, People's Republic of China.}
\email{fangyumath@163.com}



\begin{abstract}
In this paper, we study the following Lane-Emden system with nearly critical non-power nonlinearity
\begin{eqnarray*}
\left\{ \arraycolsep=1.5pt
   \begin{array}{lll}
-\Delta u =\frac{|v|^{p-1}v}{[\ln(e+|v|)]^\epsilon}\ \   &{\rm in}\ \Omega, \\[2mm]
-\Delta v =\frac{|u|^{q-1}u}{[\ln(e+|u|)]^\epsilon}\ \   &{\rm in}\ \Omega, \\[2mm]
u= v=0 \ \  & {\rm on}\ \partial\Omega,
\end{array}
\right.
\end{eqnarray*}
where $\Omega$ is a bounded smooth domain in $\R^N$, $N\geq 3$, $\epsilon>0$ is a  small parameter, $p$ and $q $ lying on the critical Sobolev hyperbola $\frac{1}{p+1}+\frac{1}{q+1}=\frac{N-2}{N}$.
We construct multiple blowing-up solutions based on  the finite dimensional Lyapunov-Schmidt reduction method as $\epsilon$ goes to zero.
\end{abstract}

\maketitle

\section{Introduction}\label{intro}

Our paper concerns the following Lane-Emden system
\begin{eqnarray}\label{eqa}
\left\{ \arraycolsep=1.5pt
   \begin{array}{lll}
-\Delta u =\frac{|v|^{p-1}v}{[\ln(e+|v|)]^\epsilon}\ \   &{\rm in}\ \Omega, \\[2mm]
-\Delta v =\frac{|u|^{q-1}u}{[\ln(e+|u|)]^\epsilon}\ \   &{\rm in}\ \Omega, \\[2mm]
u,v=0 \ \  & {\rm on}\ \partial\Omega,
\end{array}
\right.
\end{eqnarray}
where $\Omega$ is a bounded smooth domain in   $\R^N$, $N\geq 3$,
$\epsilon>0$ is a  small parameter,
$p$, $q $ lying on the critical
Sobolev hyperbola
\begin{equation}\label{lesem}
\frac{1}{p+1}+\frac{1}{q+1}=\frac{N-2}{N},
\end{equation}
which has been introduced by Mitidieri \cite{mi}.
The  Lane-Emden system   is used in physics to model spatial phenomena in a variety of biological and chemical fields.
The origin of this concept dates back to the papers \cite{cfm,pel}  for
several considerations and motivations behind them.

In the case $\epsilon=0$, $u=v$ and $p=q$, (\ref{eqa}) reduces to
the celebrated Lane-Emden-Fowler problem
\begin{equation}\label{leem}
\begin{cases}
-\Delta u=|u|^{p-1}u \ \ & {\rm in}\ \Omega,\\
u=0 \ \  & {\rm on}\ \partial\Omega,
\end{cases}
\end{equation}
which  plays a central role in the development of methods of nonlinear analysis in the last four decades.
Since the compactness of Sobolev's embedding holds,
there exist at least one positive solution and infinitely many sign-changing solutions to problem (\ref{leem})  provided that
\[
1<p<p_S=
\begin{cases}
 +\infty &  \mathrm{if~}N\leq 2,\\
 \frac{N+2}{N-2} &\mathrm{if~}N\geq 3.
\end{cases}
\]
While for $p\geq \frac{N+2}{N-2}$,
the Pohozaev identity \cite{Poh}  showed that problem (\ref{leem}) has no positive solution
if the domain $\Omega$ is strictly starshaped.
Kazdan and Warner \cite{kaz} established infinitely many radial solutions in an annulus domain.
For the critical case  $p=\frac{N+2}{N-2}$, Bahri and Coron obtained the existence of a positive solution of (\ref{leem}) if the domain $\Omega$
has nontrivial reduced homology with $\mathbb{Z}_2$-coefficient.
If the domain $\Omega$ has  small holes, the size of the hole  effects the number of sign-changing solutions of problem (\ref{leem}), see \cite{ge1,ge2}.

The natural counterpart of  (\ref{leem}) for elliptic systems is the following Lane-Emden system
\begin{eqnarray}\label{pure}
\left\{ \arraycolsep=1.5pt
   \begin{array}{lll}
-\Delta u = |v|^{p-1}v\ \   &{\rm in}\ \Omega, \\[2mm]
-\Delta v = |u|^{q-1}u \ \   &{\rm in}\ \Omega, \\[2mm]
u,v=0 \ \  & {\rm on}\ \partial\Omega,
\end{array}
\right.
\end{eqnarray}
where $p$, $q>1$.
Compared to the single case,
since  the quadratic
part of the functional of system (\ref{pure}) is strongly indefinite,
which leads to the   mountain pass theorem do not work.
However, there are some  ways to get the existence of
solutions, such as  the linking theorem, see \cite{bey,dep,hum3} and references therein.
Mitidieri \cite{mit,mi} pointedx out that  the Sobolev critical exponent $p_S$ should be
played by the so-called Sobolev hyperbola (\ref{lesem}),
which is crucial for the existence of solutions.
By the work  of \cite{mi},
the generalized Pohozaev identity for (\ref{pure}) can be written as
\begin{equation}\label{poh}
\Big(\frac{N}{p+1}-\alpha\Big)\int_{\Omega}|u|^{p+1}dx
+\Big(\frac{N}{q+1}-(N-2-\alpha)\Big)\int_{\Omega}|v|^{q+1}dx
=\int_{\partial\Omega}\frac{\partial u}{\partial\nu}\frac{\partial v}{\partial\nu}d\sigma,
\end{equation}
for every $\alpha>0$.
When $\alpha=\frac{N}{p+1}$ and $(p,q)$ lies on or above the critical hyperbola,  that is
\begin{equation*}\label{lessem}
\frac{1}{p+1}+\frac{1}{q+1}\leq \frac{N-2}{N},
\end{equation*}
The identity (\ref{poh}) gives   the nonexistence of positive solutions on star shaped domains.
Then,  the  works in \cite{blmr,hum} show  that the condition of existence   theory is   ($p,q$) satisfying  the so called  subcriticality, that is
\begin{equation}\label{qssem}
\frac{1}{p+1}+\frac{1}{q+1}>\frac{N-2}{N}.
\end{equation}
It may  hold  that $p<\frac{N+2}{N-2}<q$ for $N\geq3$,
then  the energy functional does not be well defined on $H_0^1(\Omega)\times H_0^1(\Omega)$.
In a variational setting, the  existence of  sign-changing solution  has been studied under a variety of assumptions in a large number of papers,  \cite{clap,cacn} for a single
equation, for systems, see \cite{clapq,clapw,Gue,ck},
in particular, \cite{blm} for H\'{e}non-Lane-Emden system.

System  (\ref{pure}) with critical nonlinearity  has been extensively studied in the past
decade years and many results are known regarding existence, multiplicity and  concentration phenomena.
One direction is the nearly critical hyperbola,
by the moving plane method,
Guerra \cite{Gue} obtained the uniform boundedness of least energy solutions to (\ref{pure}) on convex domains under that $p\in (\frac{2}{N-2}, \frac{N+2}{N-2}]$ and
\begin{equation}\label{qe}
\frac{1}{p+1}+\frac{1}{q_\epsilon+1}=\frac{N-2}{N}+\epsilon,
\end{equation}
for small $\epsilon>0$.
Observe that (\ref{qe}) satisfies the  subcritical condition (\ref{qssem}),
when $\epsilon\rightarrow0$, it   approaches to the critical line (\ref{lesem}).
Then, Choi and  Kim \cite{ckm} removed the convexity assumption,
and they obtained the similar results
by local Pohozaev-type identities and sharp pointwise estimates of the solutions.
Jin and Kim \cite{jin} extended the Coron's result \cite{cor} to (\ref{pure}) with critical hyperbola (\ref{lesem}) provided that $\Omega$ has a sufficiently small hole,
they constructed a  family of positive solutions  concentrating
around the center of the hole   by a perturbation argument.
In the supercritical case, more precisely, $\frac{1}{p+1}+\frac{1}{q+1}\to\frac{N-k-2}{N-k}$,
$\frac{N-k-2}{N-k}<\frac{1}{p+1}+\frac{1}{q+1}<\frac{N-2}{N}$,
where the constant  $k$ ($1\leq k\leq N-3$) is the dimension of    sub-manifolds of  $\partial\Omega$,
Guo et al. \cite{uls} established  multi-bubbling solutions  with layers concentrating along one
or several $k$-dimensional sub-manifolds.

Another direction,
motivated by the classical results of Br\'{e}zis and Nirenberg \cite{bre},
Hulshof et al. \cite{hum} considered the version of a lower order perturbation of  Lane-Emden critical system (\ref{pure}) in a bounded domain for $N\geq 3$ and \cite{abc} states  a new phenomenon for $N=3$.
By  Lyapunov-Schmidt reduction method,
Kim and Pistoia \cite{kpm} also considered the following   Brezis-Nirenberg type problem for $N\geq8$ and $p\in(1,\frac{N-1}{N-2})$
\begin{eqnarray}\label{kitm}
\left\{ \arraycolsep=1.5pt
   \begin{array}{lll}
-\Delta u = |v|^{p-1}v+\epsilon(\alpha u+\beta_1v) \ \   &{\rm in}\ \Omega, \\[2mm]
-\Delta v = |u|^{q-1}u +\epsilon(\beta_2u+\alpha v)\ \   &{\rm in}\ \Omega, \\[2mm]
u,v=0 \ \  & {\rm on}\ \partial\Omega,
\end{array}
\right.
\end{eqnarray}
where $\alpha$, $\beta_1$ and $\beta_2$  are real numbers,
they  first constructed a single bubble  in general bounded domain  and multiple bubbles solution  in the dumbbell-shaped domain.
Meanwhile, they also obtained the  existence of  multiplicity results for a slightly subcritical system
\begin{eqnarray}\label{stemsd}
\left\{ \arraycolsep=1.5pt
   \begin{array}{lll}
-\Delta u = v^{p-\alpha\epsilon}\ \   &{\rm in}\ \Omega, \\[2mm]
-\Delta v = u^{q-\beta\epsilon}\ \   &{\rm in}\ \Omega, \\[2mm]
u, v >0 \ \   &{\rm in}\ \Omega, \\[2mm]
u,v=0 \ \  & {\rm on}\ \partial\Omega.
\end{array}
\right.
\end{eqnarray}
Moreover, due to the local Pohozaev identities, Guo et al. \cite{ghp} obtained the  non-degeneracy of these   blowing-up   solutions,
which plays an important role in the construction
of new solutions.
Under the subcritical  (\ref{qe}) and $\max\{1,\frac{3}{N-2}\}<p<q$ for $N\geq 3$,
 Kim and Moon \cite{mkim}  give  a detailed qualitative and quantitative
description of the asymptotic behavior for
all positive bubbling solutions  in the convex domain.
Many important contributions have been made towards the solution of  system (\ref{stemsd}),
 we refer the interested reader to \cite{ugo} for  Lane-Emden system with Neumann boundary conditions and many others.

The above cases with  power type nonlinearity  have been extensively
studied due to the Lyapunov-Schmidt reduction method.
In the more special case of non-power nonlinearity, the problem becomes more complex and has been the subject of ongoing research in the field.
The seminal work of Castro and Pardo \cite{cp1}
has played a crucial role in the development of this problem,
their objective   is to prove  the   existence of a priori $L^\infty$
bounds for positive solutions of Laplacian
problem,
where the right hand side nonlinear term  have a slightly subcritical growth,
more precisely,
\[
f(u)=\frac{u^{\frac{N+2}{N-2}}}{[\ln(2+u)]^\epsilon}\quad\mbox{with}\ \epsilon>\frac{2}{N-2}.
\]
Then, Mavinga and Pardo \cite{np} has been successfully  obtained the same results to system (\ref{eqa}),
also for parameterized version.
Rigorously considering  priori $L^\infty$
bounds for positive solutions for the case of non-power nonlinearity have been analyzed in larger generality in
\cite{cp2,CUD,rp2} for slightly subcritical problem,
\cite{deng,ddep}  for supercritical case,
 and \cite{rm1,rp2} for studying the asymptotic behavior
of radially symmetric solutions to quasilinear $p$-Laplacian systems.

So far, there are only few results involving the non-power nonlinearity
provided that $\epsilon\rightarrow0$,
\begin{eqnarray}\label{eqad}
\left\{ \arraycolsep=1.5pt
   \begin{array}{lll}
-\Delta u =\frac{|u|^{\frac{4}{N-2}}u}{[\ln(e+|u|)]^{\epsilon}}\ \   &{\rm in}\ \Omega, \\[2mm]
u=0 \ \  & {\rm on}\ \partial\Omega.
\end{array}
\right.
\end{eqnarray}
Clapp et al. \cite{mp} provided the
first construction of   a single bubble solution for  problem (\ref{eqad}) in the bounded domain by Lyapunov-Schmidt reduction method,
where the solution concentrate at the nondegenerate critical of Robin function.
Since then, there are many works looking for solutions with
single and multiple peaks and investigating the location of the asymptotic spikes
as well as their profile as $\epsilon\rightarrow0$.
More specifically,
the investigations conducted in \cite{mp} lead to similar versions
for critical H\'{e}non problem \cite{ZZW}  and \cite{ZZWs} for  Schr\"{o}dinger equation in $\mathbb{R}^N$ for $N\geq 7$.
Another   results of the existence of multiple positive as well as sign changing solutions are allowed for (\ref{eqad})   by
Ben Ayed et al. \cite{bay}.

Motivated by the previous observations and  by the help of finite Lyapunov-Schmidt dimensional reduction procedure,
we intend to construct positive  solutions ($u_\epsilon,v_\epsilon$) to
system (\ref{eqa}),
whose shape around each blow-up point resembles a bubble in the dumbbell-shaped domain $\Omega_\eta$ that we obtain by connecting $l$ disjoint domains $\Omega_1^*\cdots\Omega_l^*$,
with $l-1$ necks of thickness less than a small number $\eta>0$.
More precisely,
let $a_1<b_1<a_2<\cdots<b_{l-1}<a_l<b_l$,
\begin{align*}
\begin{aligned}
& \Omega_i^*\subset \{(x_1,x')\in\mathbb{R}\times\mathbb{R}^{N-1}
:a_i\leq x_1\leq b_i \}
\quad
\text{and}\\
& \Omega_i^*\cap \{(x_1,x')\in\mathbb{R}\times\mathbb{R}^{N-1}:x'=0 \}
\neq\emptyset,
\end{aligned}
\end{align*}
for $i=1,\ldots,k$ and set the $\eta$-neck,
\[
\mathcal{N}_\eta
=\{(x_1,x')\in\mathbb{R}\times\mathbb{R}^{N-1}
:x_1\in(a_1,b_k),|x'|<\eta
\}.
\]
Let $\Omega_0=\cup_{i=1}^l\Omega_i^*$ and $\{\Omega_\eta\}_{\eta>0}$ be a family of smooth (connected) domains such that
\[
\Omega_0\subset\Omega_\eta\subset\Omega_0\cup\mathcal{N}_\eta
\quad\mathrm{and}\quad\Omega_{\eta_1}\subset\Omega_{\eta_2}
\quad\mathrm{for~}\eta_1\leq\eta_2.
\]
The   dumbbell-shaped domain  mentioned here is derived from  \cite{map,danc,dan},
where Musso and Pistoia \cite{map} focused on the multispike solutions for a nearly critical elliptic problem.

Our main results can be stated as followsx.

\begin{theorem}\label{kia}
Let $\Omega$ be a bounded smooth domain in   $\R^N$ with $N\geq 3$, $p\in(1,\frac{N-1}{N-2})$, $(p,q)$ satisfies (\ref{lesem}).
Then there exists a small number $\epsilon_0>0$
such that for any  $\epsilon\in(0,\epsilon_0)$, system (\ref{eqa}) has a solution with
blowing-up at one point in $\Omega$ as  $\epsilon\rightarrow0$.
\end{theorem}

\begin{theorem}\label{kisa}
Assume that $N\geq 3$, $p\in(1,\frac{N-1}{N-2})$, $(p,q)$ satisfies (\ref{lesem}).
Then there exist two small numbers $\eta_0$  and  $\epsilon_0>0$ such that for any  $\epsilon\in(0,\epsilon_0)$ and $\eta\in(0,\eta_0)$, system (\ref{eqa}) with $\Omega=\Omega_\eta$   has $\binom lk$   solutions with
blowing-up at $k$ points   as  $\epsilon\rightarrow0$.
\end{theorem}

Due to
the complexity of the strongly coupling non-power terms  in the sense that $u=0$ if and only if $v=0$,
we can not follow the method in \cite{kpm} to obtain
the desired results,
some new idea and technique computations are needed.

The   paper is organized as follows.
In Section 2,
  we describe the scheme of the proof of Theorems \ref{kia}-\ref{kisa}.
Section 3  provides  the reduction to the finite dimensional problem,
which is done by using the Lyapunov-Schmidt decomposition
at the approximate solutions.
Proposition \ref{leftside} is proved in Section 4.
Finally, there are some estimates in the Appendix.
Below we denote generic constants by $C$,
the values   may change from place to place and  will use big $O$ and small $o$ notations to describe the limit behavior of a certain quantity as
 $\epsilon\rightarrow 0$.

\section{Scheme of the proof}

In this section, we are devoted to find a solution  to  system  (\ref{eqa}).
Recall  that
 \[
 \frac1{p^*}:=\frac p{p+1}-\frac1N=\frac1{q+1}+\frac1N\quad\mathrm{and}\quad
 \frac1{q^*}:=\frac q{q+1}-\frac1N=\frac1{p+1}+\frac1N,
 \]
 so the exponents  $p^*$ and $q^* $ are the H\"{o}lder's conjugates of each other.
 By the Sobolev embedding theorem, it holds
 \[
 \begin{cases}
 \dot{W}^{2,\frac{p+1}{p}}(\mathbb{R}^N)
 \hookrightarrow\dot{W}^{1,p^*}(\mathbb{R}^N)\hookrightarrow L^{q+1}(\mathbb{R}^N),\\
 \dot{W}^{2,\frac{q+1}{q}}(\mathbb{R}^N)
 \hookrightarrow\dot{W}^{1,q^*}(\mathbb{R}^N)\hookrightarrow L^{p+1}(\mathbb{R}^N).
 \end{cases}
 \]
 For any smooth domain  $\Omega$ in $\mathbb{R}^N$, we set
 \[
 X_{p,q}=\Big(\dot{W}^{2,\frac{p+1}{p}}(\Omega)
 \cap\dot{W}_0^{1,p^*}(\Omega)\Big)
 \times
 \Big(\dot{W}^{2,\frac{q+1}{q}}(\Omega)
 \cap\dot{W}_0^{1,q^*}(\Omega)\Big),
 \]
 and
 \begin{equation}\label{uaow}
 X_{p,q,\epsilon}=\left\{(u,v)\in X_{p,q}:u\in L^{q+1- \epsilon}(\Omega),v\in L^{p+1- \epsilon}(\Omega)\right\},
 \end{equation}
equipped with the norm
 \begin{align}\label{eaf}
\|(u,v)\|= &  \|(u,v)\|_{X_{p,q,\epsilon}}\nonumber\\
= & \|\Delta u\|_{L^{\frac{p+1}{p}}(\Omega)}
+\|\Delta v\|_{L^{\frac{q+1}{q}}(\Omega)}
+\|u\|_{L^{p+1-\epsilon}(\Omega)}
+\|v\|_{L^{q+1-\epsilon}(\Omega)},
 \end{align}
and
the orthogonal is taken with respect to the scalar
product in
\begin{equation}\label{inne}
\langle (u,v),(w,z)\rangle
=\int_{\Omega}(\nabla u \nabla w
+ \nabla v \nabla z)dx.
\end{equation}

It will be useful to rewrite system (\ref{eqa})  in a different setting.
Let $\mathcal{I}^*:L^{\frac{p+1}p}(\Omega)\times L^{\frac{q+1}q}(\Omega) \hookrightarrow   X_{p,q,\epsilon}$
be the adjoint operator of the immersion $\mathcal{I}: X_{p,q,\epsilon}\hookrightarrow L^{p+1}(\Omega)\times L^{q+1}(\Omega) $,
that is, for  $(u,v)\in L^{\frac{p+1}p}(\Omega)\times L^{\frac{q+1}q}(\Omega) $,
$(w,z) =\mathcal{I}^*(u,v)$ if and only if
\[
\begin{cases}
-\Delta u=w&\text{in}~ \Omega,\\
-\Delta v=z&\text{in}~ \Omega,\\
u=v=0&\text{on}~ \partial\Omega,
\end{cases}
\quad\Longleftrightarrow\quad
\begin{cases}u(x)=\displaystyle\int_{\Omega}G(x,y)w(y)dy,\\
v(x)=\displaystyle\int_{\Omega}G(x,y)z(y)dy,
\end{cases}
\quad\mbox{for}\ x\in\Omega.
\]
Then,  by the Calder\'{o}n-Zygmund estimate, it holds
\begin{equation}\label{embed}
\|\mathcal{I}^*(u,v)\|\leq C|(u,v)|_{L^{\frac{p+1}p}(\Omega)\times L^{\frac{q+1}q}(\Omega)}
\quad\mbox{for}\ (u,v)\in L^{\frac{p+1}p}(\Omega)\times L^{\frac{q+1}q}(\Omega),
\end{equation}
for some constant $C>0$.
Moreover, we define
\begin{equation}\label{esmbed}
|(u,v)|_{L^{\frac{p+1}p}(\Omega)\times L^{\frac{q+1}q}(\Omega)}
=|u|_{L^{\frac{p+1}p}(\Omega)}
+
|v|_{L^{\frac{q+1}q}(\Omega)}.
\end{equation}
Using these definitions and notations,
system (\ref{eqa})   turns out to be equivalent to
  finding  solutions to the following equation
\[
(w,z) =\mathcal{I}^*[f_\epsilon(u),g_\epsilon(v)],
\]
where
\begin{equation}\label{fg}
f_\epsilon(u)=\frac{|u|^{p-1}u}{[\ln(e+|u|)]^\epsilon}
\quad \mbox{and}\ \
g_\epsilon(v)
=\frac{|v|^{q-1}v}{[\ln(e+|v|)]^\epsilon}.
\end{equation}

We analyze the  Green's function denoting by $G(x,y)$ for Laplace operator in $\Omega$ with zero Dirichlet
boundary condition,
 and to construct an approximate  function for  problem (\ref{eqa}),
let $H(x,y)$ be the  regular part of $G(x,y)$.
Then
\[
G(x,y)=\frac{\gamma_N}{|x-y|^{N-2}}-H(x,y),
\]
for $(x,y)\in\Omega\times\Omega $ such that $x\neq y$,
where $\gamma_{N}=\frac{1}{(N-2)|\mathbb{S}^{N-1}|}>0$ and
$|\mathbb{S}^{N-1}|=2\pi^{N/2}/\Gamma(\frac{N}{2})$   is the unit sphere in $\mathbb{R}^N$ centered at the origin.
 In addition, we introduce a function $\tilde{G}:\Omega\times\Omega\rightarrow\mathbb{R}$ satisfying
 \[
\begin{cases}
-\Delta\widetilde{G}(x,y)=G^p(x,y)\quad &\text{for}~x\in\Omega,\\
\widetilde{G}(x,y)=0 \quad &\text{for}~x\in\partial\Omega,
\end{cases}
\]
and its regular part $\tilde{H}:\Omega\times\Omega\to\mathbb{R}$ by
\[
\tilde{H}(x,y)
=
\frac{\tilde{\gamma}_{N,p}}{|x-y|^{(N-2)p-2}}-\tilde{G}(x,y),
\]
where $\tilde{\gamma}_{N,p}=\frac{\gamma_{N}^{p}}{[(N-2)p-2][N-(N-2)p]}$.

\subsection{The limit system}

By the results of   Corollary I.2 in \cite{line},
if ($p,q$) satisfies critical condition (\ref{lesem}),
the  bubble   defined by
 $(U,V)
 \in\dot{W}^{2,\frac{p+1}{p}}(\mathbb{R}^{N}\times\dot{W}^{2,\frac{q+1}{q}}(\mathbb{R}^{N})$ for $N\geq 3$ solving  the system
 \begin{eqnarray}\label{uni}
\left\{ \arraycolsep=1.5pt
   \begin{array}{lll}
-\Delta U=V^p \quad &\text{in }\mathbb{R}^N,  \\[2mm]
-\Delta V=U^q  \quad &\text{in }\mathbb{R}^N.
\end{array}
\right.
\end{eqnarray}
By the well-known result of  Alvino et. al \cite{aml},  $(U,V)$ is radially symmetric and decreasing in the radial variable after a suitable translation,
and problem (\ref{uni}) has an unique  positive ground state $(U_{1,0}(x),V_{1,0}(x))$
 such that
\begin{equation}\label{max}
 U_{1,0}(0)=\max_{x\in \mathbb{R}}U_{1,0}(x)=1.
\end{equation}
Moreover,  for any $\xi\in\mathbb{R}^{N}$ and $\mu=\mu(\epsilon)\rightarrow0$  as $\epsilon\rightarrow0$, the family of functions
\begin{align}\label{bubble}
 (U_{\mu,\xi}(x),V_{\mu,\xi}(x))
 =  \Big(\mu^{-\frac{N}{q+1}}U_{1,0}(\frac{x-\xi}{\mu}),
\mu^{-\frac{N}{p+1}}V_{1,0} (\frac{x-\xi}{\mu} )\Big),
\end{align}
are  all the positive  solutions of (\ref{uni}), see \cite{hums}.
Since $U_{ \mu,\xi}$ and $V_{ \mu,\xi}$  do not have an explicit expression and we
only have access to their decay at infinity,
which such that the estimates are not at all straightforward.

The sharp asymptotic behavior of $(U_{\mu,\xi}(x),V_{\mu,\xi}(x))$ is  shown in the following.
\begin{lemma}\cite[Theorem 2]{hums}\label{ls}
There exist positive constants $a_{N,p}$ and $b_{N,p}$ depending
only on $N$ and $p$ such that
\begin{equation*}
\begin{cases}
\lim\limits_{r\to\infty}r^{(N-2)p-2}U_{1,0}(r)=a_{N,p},\\
\lim\limits_{r\to\infty}r^{N-2}V_{1,0}(r)=b_{N,p},
\end{cases}
\end{equation*}
where $ U_{1,0}(x)=U_{1,0}(|x|)$, $V_{1,0}(x)=V_{1,0}(|x|)$ and $r=|x|$.
Furthermore,
\begin{equation}\label{lss}
b_{N,p}^p=a_{N,p}[(N-2)p-2][N-(N-2)p].
\end{equation}
\end{lemma}

\begin{lemma} \cite[Corollary 2.6-2.7]{kpm}\label{ls}
 Given any $\sigma\in(0,1)$, for $|x|\geq 1$, it holds that
\[
\left|V_{1,0}(x)-\frac{b_{N,p}}{|x|^{N-2}}\right|=O\left(\frac1{|x|^{N-1}}\right),
\]
and
\[
\left|\nabla V_{1,0}(x)+(N-2)b_{N,p}\frac{x}{|x|^N}\right|
=O\left(\frac{1}{|x|^{N-\sigma}}\right).
\]
 Assume further that $p\in(1,\frac{N-1}{N-2})$, it holds that
 \[
\left|U_{1,0}(x)-\frac{a_{N,p}}{|x|^{(N-2)p-2}}\right|=O\left(\frac{1}{|x|^{(N-2)p-1}}\right),
\]
and
\[
\left|\nabla U_{1,0}(x)+((N-2)p-2)a_{N,p}\frac{x}{|x|^{(N-2)p}}\right|=O\left(\frac{1}{|x|^{\kappa_{0}}}\right),
\]
where $\kappa_{0}=\operatorname*{min}\{N-2,(N-1)p-2\}>(N-2)p-1$.
\end{lemma}

\subsection{The first approximation of the solution}

Hereafter, for the sake of simplicity,
we will write $ (U_i,V_i)$ in place of $(U_{\mu_i,\xi_i},V_{\mu_i,\xi_i})$ for $i=1,\cdots,k$.
Let $ (PU_i,PV_i)$ to be defined as the unique solution to the problem
\begin{equation}\label{oeo}
\begin{cases}
-\Delta PU_i=V_i^p\quad &\text{in}\ \Omega,\\
-\Delta PV_i=U_i^q\quad &\text{in} \ \Omega,\\
PU_i=PV_i=0 \quad &\text{on}\ \partial\Omega,
\end{cases}
\end{equation}
for $i=1,\cdots,k$.
A standard comparison argument based on (\ref{ls}) yields
\begin{lemma}\label{uaoe}
Let $\widehat H:\Omega\times\Omega\to\mathbb{R}$ be a smooth function such that
\[
\begin{cases}
-\Delta\widehat{H}(x,y)=0\quad &\text{in}\ \Omega,\\
\widehat{H}(x,y)=\dfrac{1}{|x-y|^{(N-2)p-2}}\quad &\text{on} \ \partial\Omega,
\end{cases}
\]
 then
 \[
 PU_i(x)=U_i(x)-a_{N,p}\mu_i^{\frac{N_p}{q+1}}\widehat{H}(x,\xi_i)
 +o(\mu^{\frac{N_p}{q+1}}),
 \]
 and
 \[
 PV_i(x)=V_i(x)-\left(\frac{b_{N,p}}{\gamma_N}\right)
 \mu_i^{\frac{N}{q+1}}H(x,\xi_i)
 +o(\mu^{\frac{N}{q+1}}),
 \]
for $i=1,\cdots,k$.
\end{lemma}
 
Our construction takes advantage of the nondegeneracy of
the solutions ($U_{ \mu,\xi},  V_{ \mu,\xi}$ ) in (\ref{bubble}),
whose explicit form
are described   in \cite{fkp},
i.e., the space of solutions to problem
\begin{equation}\label{linear-equ}
\begin{cases}
-\Delta\Psi=pV_{1,0}^{p-1}\Phi \quad & \mbox{in} \ \mathbb{R}^N,\\
-\Delta\Phi=qU_{1,0}^{q-1}\Psi \quad & \mbox{in} \ \mathbb{R}^N,\\
(\Psi,\Phi)\in\dot{W}^{2,\frac{p+1}{p}}(\mathbb{R}^N)
\times\dot{W}^{2,\frac{q+1}{q}}(\mathbb{R}^N),
\end{cases}
\end{equation}
is spanned by
\begin{align}\label{pssi}
(\Psi_{1,0}^0,\Phi_{1,0}^0)
= & \left(x\cdot\nabla U_{1,0}+\frac{NU_{1,0}}{q+1},
x\cdot\nabla V_{1,0}+\frac{NV_{1,0}}{p+1}\right),\nonumber\\
 (\Psi_{1,0}^l,\Phi_{1,0}^l)
= &(\partial_l U_{1,0},\partial_l V_{1,0})\quad  \mbox{for} \ l=1,\cdots,N.
\end{align}
Moreover, we set
\begin{align}\label{psi}
(\Psi_{\mu,\xi}^{0},\Phi_{\mu,\xi}^{0})
= & \Big(\mu^{-\frac{N}{q+1}}\Psi_{1,0}^{0}(\frac{x-\xi}{\mu}),
\mu^{-\frac{N}{p+1}}
\Phi_{1,0}^{0}(\frac{x-\xi}{\mu})\Big),\nonumber\\
(\Psi_{\mu,\xi}^l,\Phi_{\mu,\xi}^l)
= & \Big(\mu^{-\frac{N}{q+1}}\Psi_{1,0}^l(\frac{x-\xi}{\mu}),
\mu^{-\frac{N}{p+1}}
\Phi_{1,0}^l(\frac{x-\xi}{\mu})\Big)
\quad \ \mbox{for} \ l=1,\cdots,N.
\end{align}
Then
\begin{equation}\label{xy}
(\Psi_{\mu,\xi}^{0},\Phi_{\mu,\xi}^{0}) =\mu(\partial_{\mu}U_{\mu,\xi},\partial_{\mu}V_{\mu,\xi}),
\quad
(\Psi_{\mu,\xi}^l,\Phi_{\mu,\xi}^l) =\mu(\partial_{\xi_{l}}U_{\mu,\xi},
\partial_{\xi_{l}}V_{\mu,\xi})
\quad \ \mbox{for} \ l=1,\cdots,N.
\end{equation}
Let $(\Psi_{il},\Phi_{il})
=(\Psi_{\mu_{i},\xi_{i}}^l,\Phi_{\mu_{i},\xi_{i}}^l)$
and    the pair   $(P\Psi_{il},P\Phi_{il}) $ be the unique smooth solution of the system
\begin{equation}\label{ssia}
\begin{cases}
-\Delta P\Psi_{il}=pV_i^{p-1}\Phi_{il}\ \ &\text{in}~\Omega,\\
-\Delta P\Phi_{il}=qU_i^{q-1}\Psi_{il}\ \ &\text{in}~\Omega,\\
P\Psi_{il}=P\Phi_{il}=0\ \ &\text{on}~\partial\Omega,
\end{cases}
\end{equation}
for $i=1,\cdots,k$ and $l=0,\cdots,N$. Then, the following result holds.

\begin{lemma}\cite[Lemma 2.10]{kpm}
If $i=1,\cdots,k$ and $l=0,\cdots,N$, for $x\in\Omega$, it holds that
\[
P\Psi_{il}(x)=
\begin{cases}
\Psi_{il}(x)+a_{N,p}\mu_i^{\frac{N_p}{q+1}}\widehat{H}(x,\xi_i)
+o(\mu^{\frac{N_p}{q+1}})\ \ &  for\ l=0,\\
\Psi_{il}(x)+a_{N,p}\mu_i^{\frac{N_p}{q+1}+1}
\partial_{\xi,l}\widehat{H}(x,\xi_i)
+o(\mu^{\frac{N_p}{q+1}+1})\ \ & for\ l=1,\cdots,N,
\end{cases}
\]
and
 \[
 P\Phi_{il}(x)=
 \begin{cases}
 \Phi_{il}(x)+\left(\frac{b_{N,p}}{\gamma_N}\right)\mu_i^{\frac{N}{q+1}}
 H(x,\xi_i)+o(\mu^{\frac{N}{q+1}})\ \ &for\ l=0,\\
 \Phi_{il}(x)+\left(\frac{b_{N,p}}{\gamma_N}\right)\mu_i^{\frac{N}{q+1}+1}
 \partial_{\xi,l}H(x,\xi_i)+o(\mu^{\frac{N}{q+1}+1})\ \ &for\ l=1,\cdots,N.
 \end{cases}
 \]
Moreover,
\begin{eqnarray}\label{gisewr1}
 |P\Psi_{il}-\Psi_{il}|_{L^{q+1}(\Omega)}=
\left\{ \arraycolsep=1.5pt
   \begin{array}{lll}
 O \Big(\mu_i^{\frac{N_p}{q+1}}\Big)\ \   & for \ l=0, \\[2mm]
 O \Big(\mu_i^{\frac{N_p}{q+1}+1}\Big) \ \  &  for \ l=1,\cdots,N,
\end{array}
\right.
\end{eqnarray}
and
\begin{eqnarray}\label{gisewr2}
 |P\Phi_{il}-\Phi_{il}|_{L^{p+1}(\Omega)}=
\left\{ \arraycolsep=1.5pt
   \begin{array}{lll}
 O \Big(\mu_i^{\frac{N_p}{q+1}}\Big)\ \   &for \ l=0, \\[2mm]
 O \Big(\mu_i^{\frac{N_p}{q+1}+1}\Big) \ \  &for \ l=1,\cdots,N,
\end{array}
\right.
\end{eqnarray}
 for  $i=1,\cdots,k$,
where  $\partial_{\xi,l}\widehat{H}(x,\xi)$ and $\partial_{\xi,l}H(x,\xi)$ stand for the $l$-th components of $\nabla_\xi \widehat{H}(x,\xi)$
and $\nabla_\xi H(x,\xi)$, respectively.
\end{lemma}

\subsection{The second approximation of the solution}

Due to  the error of the  $u$-component  $PU_i$ in Lemma \ref{uaoe}
of the first approximation $(PU_i,PV_i)$
is too big, this fact forces us to find another better   approximate solution (see Subsection 2.3 in \cite{kpm}).
That is, let the function $\mathbf{P}U_{\mathbf{d},\boldsymbol{\xi}}$ defined as the smooth solution of
\begin{equation}\label{oeo1}
\begin{cases}
-\Delta\mathbf{P}U_{\mathbf{d},\boldsymbol\xi}
=\Big(\sum\limits_{i=1}^kPV_i\Big)^p \ \ &\text{in}\ \Omega,\\[2mm]
\mathbf{P}U_{\mathbf{d},\boldsymbol\xi}=0\ \ &\text{on}\ \partial\Omega.
\end{cases}
\end{equation}
We define the function $\tilde{G}_{\mathbf{d},\boldsymbol{\xi}}:\Omega\to\mathbb{R}$ be the solution of
\[
\begin{cases}
-\Delta\widetilde{G}_{\mathbf{d},\boldsymbol{\xi}}(x)
=\left(\sum\limits_{i=1}^kd_i^{\frac{N}{q+1}}G(x,\xi_i)\right)^p \ \
&\text{in}\ \Omega,\\
\widetilde{G}_{\mathbf{d},\boldsymbol{\xi}}=0\ \ &\text{in} \ \partial\Omega,
\end{cases}
\]
and $\tilde{H}_{\mathbf{d},\boldsymbol{\xi}}:\Omega\to\mathbb{R}$ be its regular part given as
\begin{equation}\label{jua}
\widetilde{H}_{\mathbf{d},\boldsymbol{\xi}}(x)
=\sum_{i=1}^{k}d_{i}^{\frac{Np}{q+1}}
\frac{\tilde{\gamma}_{N,p}}{|x-\xi_{i}|^{(N-2)p-2}}
-\widetilde{G}_{\mathbf{d},\boldsymbol{\xi}}(x)\quad\text{for}\ x\in\Omega.
\end{equation}
Then,  for reader's convenience, we repeat the following lemma which has been established in  \cite{kpm}.

\begin{lemma}\cite[Lemma 2.12]{kpm}\label{puo}
For any $x\in \Omega$, we have
\[
\mathbf{P}U_{\mathbf{d},\boldsymbol{\xi}}(x)
=\sum_{i=1}^kU_i(x)-\mu^{\frac{Np}{q+1}}
\Big(\frac{b_{N,p}}{\gamma_N}\Big)^p\widetilde{H}_{\mathbf{d},\boldsymbol{\xi}}(x)
+o(\mu^{\frac{Np}{q+1}}).
\]
\end{lemma}

We fix $k\geq 1$, and write $\mu_{i}=\mu d_{i}$ for a small number $\mu>0$ and $i=1,\cdots,k$.
Given  $\delta_1,\delta_2\in(0,1)$ small enough,
$\mathbf{d}=(d_1,\cdots,d_k)$,
$\boldsymbol{\xi}=(\xi_1,\cdots,\xi_k)$,
let us introduce the configuration space in which the concentration points belong to
\begin{align}\label{uwo}
\Lambda
=  \{(\mathbf{d},\boldsymbol{\xi})
\in(\delta_1,\delta_1^{-1})^k
\times \Omega^k:
 \mbox{dist}(\xi_i,\partial\Omega)\geq \delta_2,\
 \mbox{dist}(\xi_i,\xi_j)\geq\delta_2\},
\end{align}
for $j=1,\cdots,k$  and $i\neq j$.

\subsection{Reformulation of the problem}

We look for a solution to (\ref{eqa}) in a small neighborhood of the second  approximation, more precisely,
 solutions of the form as
\[
(u_\epsilon,v_\epsilon)=
\Big(\mathbf{P}U_{\mathbf{d},\boldsymbol{\xi}}
+\Psi_{\mathbf{d},\boldsymbol{\xi}},
\sum_{i=1}^kPV_i+\Phi_{\mathbf{d},\boldsymbol{\xi}}\Big),
\]
where the rest term $(\Psi_{\mathbf{d},\boldsymbol{\xi}},\Phi_{\mathbf{d},\boldsymbol{\xi}})$ is small.

We introduce the   following kernel and cokernel spaces
\begin{align*}
& E_{\mathbf{d},\boldsymbol{\xi}} =
\mbox{span} \left\{(P\Psi_{il},P\Phi_{il}): i=1,\cdots,k\ \mbox{and} \  l=0,\cdots,N\right\},\\
& E_{\mathbf{d},\boldsymbol{\xi}}^\perp
 =
\left\{(\Psi,\Phi)\in X_{p,q,\epsilon}:  \int_{\Omega}(\nabla \Phi_{il}\nabla\Phi
+\nabla\Psi_{il}\nabla\Psi)dx=0,  i=1,\cdots,k \ \mbox{and} \ l=0,\cdots,N \right\},
\end{align*}
and
the projection operators
$\Pi_{\mathbf{d},\boldsymbol{\xi}}:
E_{\mathbf{d},\boldsymbol{\xi}}\rightarrow E_{\mathbf{d},\boldsymbol{\xi}}^\perp$
are
\[
\Pi_{\mathbf{d},\boldsymbol{\xi}}(\Psi,\Phi)
=\sum_{i=0}^{k}\sum_{l=0}^{N}c_{il}(P\Psi_{il},P\Phi_{il})
\quad\mbox{and}\quad
\Pi_{\mathbf{d},\boldsymbol\xi}^{\perp}
=\mathrm{Id}-\Pi_{\mathbf{d},\boldsymbol\xi}.
\]
Then,  solving (\ref{eqa}) is equivalent to find $ (\mathbf{d},\boldsymbol{\xi})\in\Lambda $ and   functions $(\Psi_{\mathbf{d},\boldsymbol{\xi}},\Phi_{\mathbf{d},\boldsymbol{\xi}})\in E_{\mathbf{d},\boldsymbol{\xi}}^\perp$   such that

(1) the auxiliary equation:
\begin{align}\label{ng2}
\Pi^{\bot}_{\mathbf{d},\boldsymbol{\xi}}
\Bigg\{\Big(\mathbf{P}U_{\mathbf{d},\boldsymbol{\xi}}
+& \Psi_{\mathbf{d},\boldsymbol{\xi}},
\sum_{i=1}^kPV_i+\Phi_{\mathbf{d},\boldsymbol{\xi}}\Big)\nonumber\\
& -\mathcal{I}^*\bigg[\bigg(g_\epsilon (\mathbf{P}U_{\mathbf{d},\boldsymbol\xi}
    +\Psi_{\mathbf{d},\boldsymbol\xi}),
    f_\epsilon\Big(\sum_{i=1}^{k}PV_{i}
    +\Phi_{\mathbf{d},\boldsymbol\xi}\Big)\bigg)\bigg]\Bigg\}=0,
\end{align}

(2) the bifurcation equation:
\begin{align}\label{ng1}
\Pi_{\mathbf{d},\boldsymbol{\xi}}
\Bigg\{\Big(\mathbf{P}U_{\mathbf{d},\boldsymbol{\xi}}
+ & \Psi_{\mathbf{d},\boldsymbol{\xi}},
\sum_{i=1}^kPV_i+\Phi_{\mathbf{d},\boldsymbol{\xi}}\Big)\nonumber\\
& -\mathcal{I}^*\bigg[\bigg(g_\epsilon (\mathbf{P}U_{\mathbf{d},\boldsymbol\xi}
    +\Psi_{\mathbf{d},\boldsymbol\xi}),
    f_\epsilon\Big(\sum_{i=1}^{k}PV_{i}
    +\Phi_{\mathbf{d},\boldsymbol\xi}\Big)\bigg)\bigg]\Bigg\}=0.
\end{align}
First of all we find, for every ($\mathbf{d},\boldsymbol{\xi}$) and for small $\epsilon$, a function $(\Psi_{\mathbf{d},\boldsymbol{\xi}}, \Phi_{\mathbf{d},\boldsymbol{\xi}})\in E_{\mathbf{d},\boldsymbol{\xi}}^\perp $ such that (\ref{ng2})  is
fulfilled.

\begin{proposition}\label{pro21}
There exists $\epsilon_0>0$ such that for any  for   $\epsilon\in(0,\epsilon_0)$
and $(\mathbf{d},\boldsymbol{\xi})\in \Lambda$,
one has the unique solution $(\Psi_{\mathbf{d},\boldsymbol{\xi}}^\epsilon,
\Phi_{\mathbf{d},\boldsymbol{\xi}}^\epsilon)\in E_{\mathbf{d},\boldsymbol{\xi}}^\perp$ to (\ref{ng2}). 
Moreover
\begin{eqnarray}\label{phi}
\|(\Psi_{\mathbf{d},\boldsymbol{\xi}}^{\epsilon},
\Phi_{\mathbf{d},\boldsymbol{\xi}}^{\epsilon}) \|
= O\Big(\mu^{(N-2)p-1}
  +
  \epsilon(\ln|\ln\mu|)
   [\mu^{\frac{Nq}{q+1}}+ \mu^{\frac{Np}{q+1}} ]\Big).
\end{eqnarray}
\end{proposition}

The proof of Proposition \ref{pro21}  is postponed to Section 3.

With help of above  proposition,  there is a unique
$(\Psi_{\mathbf{d},\boldsymbol\xi},
\Phi_{\mathbf{d},\boldsymbol\xi}) \in E^{\bot}_{ \mathbf{d},\boldsymbol{\xi}}$ such that (\ref{ng2}) holds,
which means that there are some constants $c_{il}$ ($i =1,\cdots,k$ and $ l=0,\cdots,N$) such that
\begin{align}\label{ccf}
&\Big(\mathbf{P}U_{\mathbf{d},\boldsymbol\xi}
   +\Psi_{\mathbf{d},\boldsymbol\xi},
   \sum_{i=1}^{k}PV_{i}+\Phi_{\mathbf{d},\boldsymbol\xi}\Big)\nonumber\\
   & -\mathcal{I}^*\bigg[\bigg(g_\epsilon (\mathbf{P}U_{\mathbf{d},\boldsymbol\xi}
    +\Psi_{\mathbf{d},\boldsymbol\xi}),
    f_\epsilon\Big(\sum_{i=1}^{k}PV_{i}
    +\Phi_{\mathbf{d},\boldsymbol\xi}\Big)\bigg)\bigg]
=\sum_{i=1}^k\sum_{l=0}^Nc_{il} (P\Phi_{il},P\Psi_{il}),
\end{align}
it equals to  solving  equation (\ref{ng1}),
that is, the following result is valid,
 whose proof is postponed to Section 4.

We recall the identities
\[
(N-2)p-2=(N-2)(p+1)-N=\frac{N(p+1)}{q+1},
\]
and
let $\widetilde{\tau}(\xi)=\widetilde{H}_{\mathbf{d},\boldsymbol{\xi}}(\xi)$.

\begin{proposition}\label{leftside}
For $\mathbf{d}=(d_1,\cdots,d_k)$
and
$\boldsymbol\xi=(\xi_1,\cdots,\xi_k)$,
the following facts hold.

{\it Part $a$}. If  $(\mathbf{d},\boldsymbol{\xi})$ satisfies
\begin{align*}\label{aosd}
   \bigg\langle \Big(\mathbf{P}U_{\mathbf{d},\boldsymbol\xi}
 &  +\Psi_{\mathbf{d},\boldsymbol\xi},
   \sum_{i=1}^{k}PV_{i}+\Phi_{\mathbf{d},\boldsymbol\xi}\Big)\\  &  -\mathcal{I}^*\bigg[\bigg(g_\epsilon (\mathbf{P}U_{\mathbf{d},\boldsymbol\xi}
    +\Psi_{\mathbf{d},\boldsymbol\xi}),
    f_\epsilon\Big(\sum_{i=1}^{k}PV_{i}
    +\Phi_{\mathbf{d},\boldsymbol\xi}\Big)\bigg)\bigg],
    (P\Phi_{jh},
    P\Psi_{jh})\bigg\rangle
    =(0,0),
\end{align*}
for $j =1,\cdots,k$ and $ h=0,\cdots,N$.
Then $\Big(\mathbf{P}U_{\mathbf{d},\boldsymbol\xi}
   +\Psi_{\mathbf{d},\boldsymbol\xi},
   \sum\limits_{i=1}^{k}PV_{i}+\Phi_{\mathbf{d},\boldsymbol\xi}\Big)$ is a solution of  (\ref{eqa}).

{\it Part $b$}.
There holds
\begin{eqnarray*}
  && \bigg\langle \Big(\mathbf{P}U_{\mathbf{d},\boldsymbol\xi}
   +\Psi_{\mathbf{d},\boldsymbol\xi},
   \sum_{i=1}^{k}PV_{i}+\Phi_{\mathbf{d},\boldsymbol\xi}\Big)\\
& - & \mathcal{I}^*\bigg[\bigg(g_\epsilon (\mathbf{P}U_{\mathbf{d},\boldsymbol\xi}
    +\Psi_{\mathbf{d},\boldsymbol\xi}),
    f_\epsilon\Big(\sum_{i=1}^{k}PV_{i}
    +\Phi_{\mathbf{d},\boldsymbol\xi}\Big)\bigg)\bigg],
    (P\Phi_{jh},
    P\Psi_{jh})\bigg\rangle\\
&  = &
\left\{ \arraycolsep=1.5pt
   \begin{array}{lll}
  -\frac{ k}{N}\frac{\epsilon}{|\ln\mu|} \Big( (p+1)\mathcal{A}_1+ (q+1)\mathcal{\tilde{A}}_1\Big)
  -\frac{1}{N}
  \frac{\epsilon}{|\ln\mu|^2}
  \Big( (p+1)\mathcal{A}_1+ (q+1)\mathcal{\tilde{A}}_1\Big)
  \sum\limits_{i=1}^k |\ln d_i|\\
\ \  +
\mu^{(N-2)p-2}\bigg(
 \Big(\frac{b_{N,p}}{\gamma_N}\Big)^p \mathcal{A}_2
\sum\limits_{i=1}^kd_i^{\frac{N}{q+1}}
\widetilde{H}_{\mathbf{d},\boldsymbol{\xi}}(\xi_i)
 -a_{N,p}\mathcal{A}_4
\sum\limits_{j\neq i}^k\frac{d_{i}^{\frac{2N}{q+1}}
d_{j}^{\frac{N(p-1)}{q+1}}}{|\xi_{i}-\xi_{j}|^{(N-2)p-2}}
\bigg)
\\
\ \ + O\Big(\mu^{(N-2)p-1}
  +  \epsilon (\ln|\ln\mu|)
  [\mu^{\frac{Nq}{q+1}}+ \mu^{\frac{Np}{q+1}}]
 + \sum\limits_{i=1}^k \frac\epsilon{|\ln\mu_i|}
  \Big)\ \   \   {\rm if}\  h=0, \\[2mm]
  \Big(\frac{b_{N,p}}{\gamma_N}\Big)^p
  \mathcal{A}_3
  \mu^{(N-2)p-1}
\sum\limits_{i=1}^kd_i^{\frac{N}{q+1}+1}
\partial_{\xi_{ih}}\tilde{\rho}(\xi_i)
  +
  O\Big(\epsilon(\ln|\ln\mu|)
   [\mu^{\frac{Nq}{q+1}}+  \mu^{\frac{Np}{q+1}}] \Big)
   \ \  \    {\rm if}\  h=1,\cdots,N, \\[2mm]
\end{array}
\right.
\end{eqnarray*}

where $j =1,\cdots,k$, and
\begin{align*}
& \mathcal{A}_1=-\int_{\mathbb{R}^N}
     U^q_{1,0}(y)
     \ln \Big(U_{1,0}(y)\Big)\Psi_{1,0}^0(y)dy >0,\\
& \mathcal{\tilde{A}}_1=-\int_{\mathbb{R}^N}
     V^p_{1,0}(y)
     \ln \Big(V_{1,0}(y)\Big)\Phi_{1,0}^0(y)dy >0\\
     &  \mathcal{A}_2=q\int_{\mathbb{R}^N}
 U^{q-1}_{1,0}(y)
\Psi_{1,0}^0(y)dy,\quad
     \mathcal{A}_3= \int_{\mathbb{R}^N}
 U^{q}_{1,0}(y)dy,
\quad
 \mbox{and}\quad
\mathcal{A}_4
=\frac{1}{q}\int_{\mathbb{R}^N}  U^{q-1}_{1,0}(y)dy.
\end{align*}
\end{proposition}

From Propositions \ref{pro21} and \ref{leftside},
we   view that  $\Big(\mathbf{P}U_{\mathbf{d},\boldsymbol\xi}
   +\Psi_{\mathbf{d},\boldsymbol\xi},
   \sum\limits_{i=1}^{k}PV_{i}+\Phi_{\mathbf{d},\boldsymbol\xi}\Big) $
is the solution to system (\ref{eqa}) if there  are $\mathbf{d}_\epsilon>0$ and $\boldsymbol\xi_\epsilon\in \Omega^k$ such that $c_{il}$ ($i=1,\cdots,k$ and $ l=0,\cdots,N$) are zero when $\epsilon$ small enough.

We are now ready to prove our main results.

\noindent{\it \textbf{Proof of   Theorem  \ref{kia}}}.
By Proposition \ref{leftside}, we have
\begin{align*}
G_0(d,\xi)
= &
-\frac{ k}{N}\frac{\epsilon}{|\ln\mu|} \Big( (p+1)\mathcal{A}_1+ (q+1)\mathcal{\tilde{A}}_1\Big)
  -\frac{1}{N}\frac{\epsilon}{|\ln \mu|^2}|\ln d| \Big( (p+1)\mathcal{A}_1+ (q+1)
  \mathcal{\tilde{A}}_1\Big)
 \\
& +
\Big(\frac{b_{N,p}}{\gamma_N}\Big)^p \mathcal{A}_2
 \mu^{(N-2)p-2}d^{\frac{N}{q+1}}\widetilde{\tau}(\xi)\\
& + O\bigg(\mu^{(N-2)p-1}
  + \epsilon(\ln|\ln\mu|)
  [ \mu^{\frac{Nq}{q+1}}
  +  \mu^{\frac{Np}{q+1}}]
   \frac{\epsilon}{|\ln\mu|}
  \bigg)\\
 = &- C_0\frac{\epsilon}{|\ln\mu|}
 -C_1 \frac{\epsilon}{|\ln \mu|^2}|\ln d|
 +
  C_2
\mu^{(N-2)p-2}d^{\frac{N}{q+1}}\widetilde{\tau}(\xi)
+(h.o.t.),
\end{align*}
and for $h=1,\cdots,N$,
\begin{align*}
 G_h(d,\xi)
= &
  \Big(\frac{b_{N,p}}{\gamma_N}\Big)^p
  \mathcal{A}_3
  \mu^{(N-2)p-1}
\sum\limits_{i=1}^kd_i^{\frac{N}{q+1}+1}
\partial_{\xi_{ih}}\tilde{\rho}(\xi_i)
 + O\Big(\mu^{(N-2)p-1}
  +
  \epsilon (\ln|\ln\mu|)
  [\mu^{\frac{Nq}{q+1}}+ \mu^{\frac{Np}{q+1}}]
   \Big) \\
  = &   C_3
  \mu^{(N-2)p-1}
\sum\limits_{i=1}^kd_i^{\frac{N}{q+1}+1}
\partial_{\xi_{ih}}\tilde{\rho}(\xi_i)
+(h.o.t.).
\end{align*}
In function $G_0$,
it will turn out that a convenient choice for $\mu$ gives their size of the
order
\begin{equation}\label{awao}
\frac{\epsilon}{|\ln\mu|^2}=\mu^{(N-2)p-2}
\Longrightarrow
\epsilon= \mu^{(N-2)p-2}|\ln\mu|^2.
\end{equation}
Then
\[
G_0(d,\xi)
= - C_0 \mu^{(N-2)p-2}|\ln\mu|
-\mu^{(N-2)p-2}
\underbrace{\Big(C_1|\ln d|- C_2
d^{\frac{N}{q+1}}
\widetilde{\tau}(\xi)
\Big)}_{\tilde{G}_0(d,\xi)}+(h.o.t.).
\]
Let $\xi_0\in\Omega$  be a strict minimum point  of  function $\widetilde{\rho}$,
from (\ref{uwo}), we choose $\delta_1$, $\delta_2$
 small enough,
it follows that  the
function $\tilde{G}_0$ have a strict minimum point in $\mbox{int}(\Lambda)$,
which means that  $G_0$ has a minimum point in  $\mbox{int}(\Lambda)$ as $\epsilon$ goes to zero.
\qed

\noindent{\it \textbf{Proof of   Theorem  \ref{kisa}}.}
From (\ref{jua}),
let $
\widetilde{H}^\eta_{\mathbf{d},\boldsymbol{\xi}}$,
$\tilde{\rho}^\eta$
be the function introduced  for the dumbbell-shaped domain $\Omega_\eta$.
The functions
$G_{h\eta}$
and $G^\epsilon_{h\eta}$ for $h=0,\cdots,N$
 are related to the disconnected domain $\Omega_0=\cup_{i=1}^l\Omega_i^*$.
By Proposition \ref{leftside}, for $h=0$, we write
\begin{equation*}
G^\epsilon_{0\eta}(\mathbf{d},\boldsymbol{\xi})
=    C_3
  - C_4
  \frac{\epsilon}{|\ln\mu|^2}
  \sum\limits_{i=1}^k |\ln d_i|
   +\mu^{(N-2)p-2}
   G_{0\eta}(\mathbf{d},\boldsymbol{\xi})
+ (h.o.t.),
\end{equation*}
with
\begin{equation*}
G_{0\eta}(\mathbf{d},\boldsymbol{\xi})
  = \bigg(
 \Big(\frac{b_{N,p}}{\gamma_N}\Big)^p \mathcal{A}_2
\sum\limits_{i=1}^kd_i^{\frac{N}{q+1}}
\widetilde{H}^\eta_{\mathbf{d},\boldsymbol{\xi}}(\xi_i)
 -a_{N,p}\mathcal{A}_4
\sum\limits_{j\neq i}^k\frac{d_{i}^{\frac{2N}{q+1}}
d_{j}^{\frac{N(p-1)}{q+1}}}{|\xi_{i}-\xi_{j}|^{(N-2)p-2}}
\bigg),
\end{equation*}
and for $h=1,\cdots,N$,
\begin{equation*}
G^\epsilon_{h\eta}(\mathbf{d},\boldsymbol{\xi})
  =C_5
  \mu^{(N-2)p-1}
\sum\limits_{i=1}^kd_i^{\frac{N}{q+1}+1}
\partial_{\xi_{ih}}\tilde{\rho}^\eta(\xi_i)
+ (h.o.t.),
\end{equation*}
We obtain the same relation as (\ref{awao}).
Let $\Lambda_0$ be the configuration space  $\Lambda$  defined in (\ref{uwo}) related to $\Omega_0$.
By Lemmas 6.1-6.2 in  \cite{kpm}, we obtain
\begin{align*}
G^\epsilon_{0\eta}(\mathbf{d},\boldsymbol{\xi}) \rightarrow
\hat{G}_0(\mathbf{d},\boldsymbol{\xi})
=
 \Big(\frac{b_{N,p}}{\gamma_N}\Big)^p \mathcal{A}_2
\sum\limits_{i=1}^k d_i^{\frac{N(p+1)}{q+1}}
\widetilde{\tau}_{\Omega_i^*}(\xi_i),
\end{align*}
uniformly on $\Lambda_0$   as  $\epsilon\rightarrow0$.
Thus,
\begin{align*}
& G^\epsilon_{0\eta}(\mathbf{d},\boldsymbol{\xi}) \rightarrow
\tilde{G}_0(\mathbf{d},\boldsymbol{\xi})
=
-\mu^{(N-2)p-2}
\bigg(C_4\sum\limits_{i=1}^k |\ln d_i|
- C_6
\sum\limits_{i=1}^k d_i^{\frac{2N(p+1)}{q+1}}
\widetilde{\tau}_{\Omega_i^*}(\xi_i)
\bigg),\\
& G^\epsilon_{h\eta}(\mathbf{d},\boldsymbol{\xi}) \rightarrow
\tilde{G}_h(\mathbf{d},\boldsymbol{\xi})
=C_5
  \mu^{(N-2)p-1}
\sum\limits_{i=1}^kd_i^{\frac{N}{q+1}+1}
\partial_{\xi_{ih}}\tilde{\rho}^\eta(\xi_i),
\quad\mbox{for}\ h=1,\cdots,N,
\end{align*}
uniformly on $\Lambda_0$  as $ \epsilon\rightarrow0$.

The functions $\tilde{G}_0$ and $\tilde{G}_h$ have a strict minimum point $(\mathbf{d}_0,\boldsymbol{\xi}_0) \in(0,\infty)^k\times (\Omega_1^*\times\cdots \times\Omega_k^*)$.
It follows that  the functions  $G^\epsilon_{0\eta}$ and $G^\epsilon_{h\eta}$ also have a strict minimum point $(\mathbf{d}_\eta,\boldsymbol{\xi}_\eta) \in (0,\infty)^k\times (\Omega_1^*\times\cdots \times\Omega_k^*)$ provided that $\eta$ is small enough.
Thus,  we obtain the existence of  a minimum
point provided that $\epsilon$ is small enough.
Moreover, we deduce that
the right hand side of (\ref{ccf}) is zero, i.e.,
\[
\sum_{i=1}^k\sum_{l=0}^Nc_{il}
\Big\langle (P\Phi_{il},P\Psi_{il}),
(P\Phi_{jh},P\Psi_{jh})\Big\rangle=0,
\]
for $j=1,\cdots,k$ and $h=0,\cdots,N$,
and by Lemma \ref{inerpro}, we conclude that $c_{il} $  are zero.
 We finish the proof of this theorem.
\qed

\section{The  finite dimensional reduction}

In this section, we outline the main steps of the so-called finite-dimensional reduction.
That is, we need to  prove Proposition \ref{pro21}.
First, we define a linear operator $L_{\mathbf{d},\boldsymbol\xi}:E_{\mathbf{d},\boldsymbol{\xi}}^\perp \to E_{\mathbf{d},\boldsymbol{\xi}}^\perp $ by
\begin{equation}\label{lineare}
\begin{aligned}
L_{\mathbf{d},\boldsymbol{\xi}}(\Psi,\Phi)& =(\Psi,\Phi)-\Pi_{\mathbf{d},\boldsymbol{\xi}}^{\perp}
\mathcal{I}^{*}
\Bigg[\bigg(g^{'}_\epsilon (\mathbf{P}U_{\mathbf{d},\boldsymbol\xi}
    +\Psi_{\mathbf{d},\boldsymbol\xi})\Psi,
    f^{'}_\epsilon\Big(\sum_{i=1}^{k}PV_{i}
    +\Phi_{\mathbf{d},\boldsymbol\xi}\Big)\Phi\bigg)\bigg].
\end{aligned}
\end{equation}

Arguing as in \cite{kpm}, one can prove the following  invertibility of the operator $L_{\mathbf{d},\boldsymbol{\xi}}$  on $E_{\mathbf{d},\boldsymbol{\xi}}^\perp$.
\begin{lemma}\label{inver}
Reduce the value of $\epsilon_{0}>0$ if necessary.
Then there is a universal constant $C>0$ such that for each $\epsilon\in(0,\epsilon_{0})$ and $(\mathbf{d},\boldsymbol\xi)\in\Lambda $, the operator $L_{\mathbf{d},\boldsymbol{\xi}}$ satisfies
\begin{equation}\label{lal}
\|L_{\mathbf{d},\boldsymbol{\xi}}(\Psi,\Phi)\|\geq C\|(\Psi,\Phi)\|\quad for\:every\:(\Psi,\Phi)\in E_{\mathbf{d},\boldsymbol{\xi}}^\perp.
\end{equation}
\end{lemma}

\begin{proof}
We prove it by contradiction.
For $m\in \mathbb{N}_+$,
assume that there are sequences of parameters    $ \epsilon_m\rightarrow0$,
$\boldsymbol{\xi}_m=(\xi_{1m},\cdots,\xi_{km})\in\Omega^k$  and
$\mathbf{d}_m =(d_{1m},\cdots,d_{km})\in \R^k_+$
with $\boldsymbol{\xi}_m\rightarrow \boldsymbol{\xi}_\infty\in\Omega$,
and $d_{im}\rightarrow \mathbf{d}_\infty>0$,
$i=1,\cdots,k$,
$(\Psi_m,\Phi_m)$, $(H_{1m},H_{2m})\in E_{\mathbf{d}_m,\boldsymbol{\xi}_m}^\perp$
such that
\begin{equation}\label{assu}
 \|L_{\mathbf{d}_m,\boldsymbol{\xi}_m}(\Psi_m,\Phi_m)\|
 =(H_{1m},H_{2m}),
\end{equation}
and
\begin{equation}\label{aye}
  \|(\Psi_m,\Phi_m)\|=1 \quad\mbox{and}\quad  \|(H_{1m},H_{2m})\|\rightarrow0.
\end{equation}
Let $\mathbf{P}U_m
=\mathbf{P}U_{\mathbf{d}_m,\boldsymbol\xi_m}$.
From (\ref{lineare}) and (\ref{assu}),
there is a sequence  $\{c_{il,m}\}$ ($i=1,\cdots,k$ and $
l=0,\cdots,N$) of coefficients such that
\begin{equation}\label{linear}
\begin{aligned}
(\Psi_m,\Phi_m)
- & \mathcal{I}^{*}
\bigg[\bigg(g^{'}_{\epsilon_m} (\mathbf{P}U_m)
\Psi_m,
f^{'}_{\epsilon_m}\Big(\sum_{i=1}^{k}PV_{im}\Big)
\Phi_m\bigg)\bigg]\\
= & (H_{1m},H_{2m})
+\sum_{i=1}^k \sum_{l=0}^N c_{il,m} (P\Psi^l_{im},P\Phi^l_{im}).
\end{aligned}
\end{equation}
The proof of this result consists of three steps.

\textbf{Step 1.}
We prove that
\begin{equation}\label{lfimi}
\sum\limits_{i=1}^k\sum\limits_{l=0}^N c_{il,m} \rightarrow0, \quad\mbox{as}\ m\rightarrow \infty.
\end{equation}
We multiply (\ref{linear})
by  $(P\Phi^h_{ jm},P\Psi^h_{ jm})$ for $j=1,\cdots,k$ and $h=1,\cdots,N$,
and integrating over  $\Omega$, using (\ref{ssia}) and (\ref{aye}),  we have
\begin{align}\label{aea}
&  \int_\Omega\sum\limits_{i=1}^k
\sum\limits_{l=0}^Nc_{il,m} \Big(f^{'}_0(V_{im})\Phi^l_{im}P\Phi^h_{ jm}
+g^{'}_0(U_{im}) \Psi^l_{im}P\Psi^h_{ jm}\Big)dx\nonumber\\
= &  \int_\Omega
\bigg[f^{'}_{\epsilon_m}\Big(\sum_{i=1}^{k}PV_{im}\Big)
\Phi_mP\Phi^h_{ jm}
+g^{'}_{\epsilon_m} (\mathbf{P}U_{im})
\Psi_mP\Psi^h_{ jm}
\bigg]dx.
\end{align}
From Lemma \ref{inerpro},
the left side in (\ref{aea}) is
\begin{align}\label{aea3}
& \sum\limits_{i=1}^k
\sum\limits_{l=0}^Nc_{il,m}
\int_\Omega\Big(f^{'}_0(V_{im})\Phi^l_{im}P\Phi^h_{ jm}
+g^{'}_0(U_{im}) \Psi^l_{im}P\Psi^h_{ jm}\Big)dx\nonumber\\
= &  \sum\limits_{i=1}^k\sum\limits_{l=0}^N c_{il,m}
\Big[\int_{\mathbb{R}^N} \Big(f_0^{'}(V_{1,0})(\Phi^l_{1,0})^2
+g^{'}_0(U_{1,0})  (\Psi^l_{1,0})^2\Big)dy
+o(1)\Big].
\end{align}
Now, from  (\ref{aegs}), (\ref{aesgs}),  (\ref{subu2}),  (\ref{subu3}),
(\ref{gisewr1}), (\ref{pssi}), (\ref{psi}),
Lemmas \ref{yy}-\ref{ysy} and \ref{ls},
one has
\begin{align*}
 & \int_\Omega
\bigg[f^{'}_{\epsilon_m}\Big(\sum_{i=1}^{k}PV_{im}\Big)
\Phi_mP\Phi^h_{ jm}
+g^{'}_{\epsilon_m} (\mathbf{P}U_m)
\Psi_mP\Psi^h_{ jm}
\bigg]dx
\nonumber\\
  = & \int_\Omega \bigg[f_{\epsilon_m}^{'}\Big(\sum_{i=1}^{k}PV_{im}\Big)
  -f_0^{'}\Big(\sum_{i=1}^{k}PV_{im}\Big)\bigg]\Phi_m
  \Big(P\Phi^h_{jm}-\Phi^h_{jm}\Big)dx\nonumber\\
 &  + \int_\Omega\bigg[ f_{\epsilon_m}^{'}\Big(\sum_{i=1}^{k}PV_{im}\Big)
 -f_0^{'}\Big(\sum_{i=1}^{k}PV_{im}\Big)\bigg]
  \Phi_m \Phi^h_{jm}dx\nonumber\\
  & + \int_\Omega\bigg[f_0^{'}
  \Big(\sum_{i=1}^{k}PV_{im}\Big)
  -\sum\limits_{i=1}^kf_0^{'}(PV_{im})\bigg]   \Phi_m P\Phi^h_{jm}dx
     + \sum\limits_{i=1}^k
    \int_\Omega\Big[f_0^{'}(PV_{im})
    -f_0^{'}(V_{im})\Big]   \Phi_m P\Phi^h_{jm}dx\nonumber\\
   & + \int_\Omega
   \Big[g^{'}_{\epsilon_m} (\mathbf{P}U_m)
   -\sum_{i=1}^{k}g^{'}_0 (U_{im})\Big]
\Psi_mP\Psi^h_{ jm}dx
+ \sum_{i=1}^{k}\int_\Omega
   g^{'}_0 (U_{im})
\Psi_mP\Psi^h_{ jm}dx\nonumber\\
  \leq  &   \bigg|f_{\epsilon_m}^{'}\Big(\sum_{i=1}^{k}PV_{im}\Big)
  -f_0^{'}\Big(\sum_{i=1}^{k}PV_{im}\Big)\bigg|_{L^{\frac{p+1}{p-1}}(\Omega)}
  |\Phi_m|_{L^{p+1}(\Omega)}
  \Big|P\Phi^h_{jm}-\Phi^h_{jm}\Big|_{L^{p+1}(\Omega)}\nonumber\\
  & +  \bigg|f_{\epsilon_m}^{'}\Big(\sum_{i=1}^{k}PV_{im}\Big)-f_0^{'}\Big(\sum_{i=1}^{k}PV_{im}\Big) \bigg|_{L^{\frac{p+1}{p-1}}(\Omega)}|\Phi_m|_{L^{p+1}(\Omega)}
  | \Phi^h_{jm}|_{L^{p+1}(\Omega)}\nonumber\\
  & + \bigg|f_0^{'}\Big(\sum_{i=1}^{k}PV_{im}\Big)
  -\sum\limits_{i=1}^kf_0^{'}(PV_{im})\bigg|_{L^{\frac{p+1}{p-1}}(\Omega)}
   |P\Phi^h_{jm}|_{L^{p+1}(\Omega)}
   |\Phi_m|_{L^{p+1}(\Omega)}\nonumber\\
  & + \sum\limits_{i=1}^k\Big|f_0^{'}(PV_{im})-f_0^{'}(V_{im})\Big|_{L^{\frac{p+1}{p-1}}(\Omega)}
   |P\Phi^h_{jm}|_{L^{p+1}(\Omega)}
   |\Phi_m|_{L^{p+1}(\Omega)}\nonumber\\
  & + \Big|g^{'}_{\epsilon_m} (\mathbf{P}U_m)
   -\sum_{i=1}^{k}g^{'}_0 (U_{im})\Big|_{L^{\frac{q+1}{q-1}}(\Omega)}
|\Psi_m|_{L^{q+1}(\Omega)}|P\Psi^h_{ jm}|_{L^{q+1}(\Omega)}\nonumber\\
&+ \sum_{i=1}^{k}|
   g^{'}_0 (U_{im})|_{L^{\frac{q+1}{q-1}}(\Omega)}
|\Psi_m|_{L^{q+1}(\Omega)}|P\Psi^h_{ jm}|_{L^{q+1}(\Omega)}\nonumber\\
= & O \Big(\epsilon_m(\ln|\ln\mu_m|)
   [\mu_m^{\frac{Np}{p+1}}+  \mu_m^{\frac{Np}{q+1}}]
   +
   \mu_m^{(N-2)p}\Big).
  \end{align*}
In conclusion, from which and  (\ref{aea3}), we obtain (\ref{lfimi}).

\textbf{Step 2.}
We set  a smooth cut-off function $\chi:\mathbb{R}^N\rightarrow[0,1]$ as
\begin{align}\label{eqda}
\chi(x) =
\left\{ \arraycolsep=1.5pt
   \begin{array}{lll}
1 \ \   & {\rm in}\ B(\xi_{h\infty},\varrho), \\[2mm]
0 \ \   & {\rm in}\ \Omega\setminus B(\xi_{h\infty},2\varrho) ,
\end{array}
\right.
\quad
|\nabla\chi(x)|\leq \frac{2}{\varrho}
\quad {\rm and } \quad
|\nabla^2\chi(x)|\leq \frac{4}{ \varrho^2},
\end{align}
for  any $ h=1,\cdots,k$.
We define
\begin{equation}\label{sarj}
(\tilde{\Phi}_m(y),\tilde{\Psi}_m(y))
=\Big(\mu_{lm}^{-\frac{N}{q+1}}
(\chi\Phi_m)(\mu_{lm}y+\xi_{lm}),
\mu_{lm}^{-\frac{N}{p+1}}
(\chi\Psi_m)(\mu_{lm}y+\xi_{lm})\Big)\quad\mbox{for}\ y\in \mathbb{R}^N.
\end{equation}
It follows that
\begin{eqnarray}\label{ulae}
\left\{ \arraycolsep=1.5pt
   \begin{array}{lll}
\Delta \tilde{\Phi}_m(y)
= \mu_{lm}^{\frac{pN}{p+1}}
\bigg[\Big(\chi\Delta \Phi_m
+2\nabla\chi\nabla\Phi_m
+\Phi_m\Delta\chi
\Big)(\mu_{lm}y+\xi_{lm})\quad\mbox{for}\ y\in \mathbb{R}^N, \\[2mm]
\Delta \tilde{\Psi}_m(y)
= \mu_{lm}^{\frac{qN}{q+1}}
\bigg[\Big(\chi\Delta \Psi_m
+2\nabla\chi\nabla\Psi_m
+\Psi_m\Delta\chi
\Big)(\mu_{lm}y+\xi_{lm})\quad\mbox{for}\ y\in \mathbb{R}^N. \\[2mm]
\end{array}
\right.
\end{eqnarray}
Then
\begin{equation}\label{ayfe}
\|\Delta\tilde{\Phi}_m(y)\|_{L^{\frac{p+1}{p}}(\Omega)}
+
\|\Delta\tilde{\Psi}_m(y)\|_{L^{\frac{q+1}{q}}(\Omega)}
\leq C.
\end{equation}

 Next, we prove that
\begin{equation}\label{uua}
(\tilde{\Phi}_m,\tilde{\Psi}_m) \rightarrow (0,0)
 \quad
 \mbox{weakly\ in}\ \dot{W}^{2,\frac{p+1}{p}}(\mathbb{R}^N)
\times
 \dot{W}^{2,\frac{q+1}{q}}(\mathbb{R}^N).
 \end{equation}
From (\ref{ayfe}), we obtain
\[
(\tilde{\Phi}_m,\tilde{\Psi}_m) \rightarrow (\tilde{\Phi}_\infty,\tilde{\Psi}_\infty)
\quad \mbox{weakly\ in}\ \dot{W}^{2,\frac{p+1}{p}}(\mathbb{R}^N)
\times
 \dot{W}^{2,\frac{q+1}{q}}(\mathbb{R}^N).
\]
For any $(\Theta_1,\Theta_2)\in C_c^\infty(\mathbb{R}^N)\times C_c^\infty(\mathbb{R}^N)$,
by (\ref{linear}) and (\ref{ulae}), it holds
\begin{align} \label{linearn}
 & \int_\Omega \Big[ \nabla \tilde{\Psi}_m \nabla\Theta_2
+\nabla\tilde{\Phi}_m \nabla\Theta_1\Big]dx \nonumber\\
= & \int_\Omega\Bigg[
\mu_{lm}^{\frac{pN}{p+1}}
 \Big(\chi\Delta \Psi_m
+2\nabla\chi\nabla\Psi_m
+\Psi_m\Delta\chi
\Big)\Theta_2
 +
\mu_{lm}^{\frac{qN}{q+1}}
 \Big(\chi\Delta \Phi_m
+2\nabla\chi\nabla\Phi_m
+\Phi_m\Delta\chi
\Big)\Theta_1 \bigg]dx \nonumber\\
 & +
\int_\Omega\Big[ g^{'}_{\epsilon_m} (\mathbf{P}U_m)
\Psi_m \Theta_2
+
f^{'}_{\epsilon_m}\Big(\sum_{i=1}^{k}PV_{im}\Big)
\Phi_m\Theta_1\Big]dx.
\end{align}
Using (\ref{eqda}), we get
\begin{align*}
 \int_\Omega\Bigg[
\mu_{lm}^{\frac{pN}{p+1}}
 \Big(\chi\Delta \Psi_m
+2\nabla\chi\nabla\Psi_m
+\Psi_m\Delta\chi
\Big)\Theta_2
 +
\mu_{lm}^{\frac{qN}{q+1}}
 \Big(\chi\Delta \Phi_m
+2\nabla\chi\nabla\Phi_m
+\Phi_m\Delta\chi
\Big)\Theta_1 \bigg]dx=o(1).
\end{align*}
On the other hand,
since
\begin{align*}
&  g^{'}_{\epsilon_m} (\mathbf{P}U_m\Big(\mu_{lm}y+\xi_{lm})\Big)
  +
   f^{'}_{\epsilon_m}
  \Big(\sum\limits_{i=1}^{k}PV_{im}
 (\mu_{lm}y+\xi_{lm})\Big)\nonumber\\
  = &  g_\epsilon^{'}\Big( U_{ im}(\mu_{lm}y+\xi_{lm})
  +
  \sum\limits_{i=1, j\neq i}^kU_{jm}(\mu_{lm}y+\xi_{lm})\Big)\nonumber\\
 & +
  f_\epsilon^{'}\Big( PV_{ im}(\mu_{lm}y+\xi_{lm})
  +
  \sum\limits_{i=1, j\neq i}^kPV_{jm}(\mu_{lm}y+\xi_{lm})\Big) \nonumber\\
  = &  g_\epsilon^{'}\Big((\mu_m)^{-\frac{N}{q+1}} U_{1,0} (y)++o(1)\Big)
  +
  f_\epsilon^{'}\Big((\mu_m)^{-\frac{N}{p+1}} V_{1,0} (y)++o(1)\Big).
\end{align*}
Let $\mbox{supp} \Theta
=\min\{\mbox{supp}\Theta_1,\mbox{supp}\Theta_2\}$,
then, by Lebesgue's dominated convergence theorem,
we deduce
\begin{align}\label{eqsdfsa}
& \lim\limits_{m\rightarrow \infty}
 \bigg[\mu_{lm}^{\frac{pN}{p+1}}
 \int_{\mbox{supp}\Theta_1}
 f^{'}_{\epsilon_m}\Big(\sum\limits_{i=1}^{k}PV_{im}
 (\mu_{lm}y+\xi_{lm})\Big)
\Phi_m(y)\Theta_1(y)dy\nonumber\\
& +
\lim\limits_{m\rightarrow \infty}
 \mu_{lm}^{\frac{qN}{q+1}}
 \int_{\mbox{supp}\Theta_2}
 g^{'}_{\epsilon_m} \Big(\mathbf{P}U_m(\mu_{lm}y+\xi_{lm})\Big)
\Psi_m (y)\Theta_2(y)dy\nonumber\\
= &
\int_{\mbox{supp}\Theta}
\Big[f^{'}_0(V_{1,0})\tilde{\Phi}_\infty\Theta_1
+ g^{'}_0(U_{1,0})\tilde{\Psi}_\infty\Theta_2 \Big]dy,
\end{align}
for each $\Theta\in C_c^\infty(\mathbb{R}^N)$.
From (\ref{linearn})-(\ref{eqsdfsa}), we conclude that $(\tilde{\Psi}_\infty,\tilde{\Phi}_\infty)$
  is a weak solution of
 \begin{equation}\label{linebar-equ}
\begin{cases}
-\Delta\tilde{\Psi}_\infty
=f_0^{'}(V_{1,0})\tilde{\Phi}_\infty \quad & \mbox{in} \ \mathbb{R}^N,\\
-\Delta\tilde{\Phi}_\infty
=g_0^{'}(U_{1,0})\tilde{\Psi}_\infty \quad & \mbox{in} \ \mathbb{R}^N,\\
(\tilde{\Psi}_\infty,\tilde{\Phi}_\infty)\in \dot{W}^{2,\frac{p+1}{p}}(\mathbb{R}^N)
\times
 \dot{W}^{2,\frac{q+1}{q}}(\mathbb{R}^N),
\end{cases}
\end{equation}
 and satisfies the following  orthogonality condition
\begin{align}
 &\int_{\mathbb{R}^N} \left(f^{'}_0(V_{1,0})\Phi_{1,0}^l\tilde{\Phi}_\infty
+g^{'}_0(U_{1,0})\Psi_{1,0}^l\tilde{\Psi}_\infty\right)dx\nonumber\\
= &  \lim\limits_{m\rightarrow \infty}
\int_{\mathbb{R}^N} \left(f^{'}_0(V_{1,0})\Phi_{1,0}^l\tilde{\Phi}_m
+g^{'}_0(U_{1,0})\Psi_{1,0}^l\tilde{\Psi}_m\right)dx\nonumber \\
= &  \lim\limits_{m\rightarrow \infty}
\int_{\mathbb{R}^N} \left(f^{'}_0(V_{1,0})\Phi_{1,0}^l\tilde{\Phi}_m
+g^{'}_0(U_{1,0})\Psi_{1,0}^l\tilde{\Psi}_m\right)dx \nonumber\\
= & \lim\limits_{m\rightarrow \infty}
\int_{B(0,\frac{3\varrho}{\mu_{lm}})\setminus
B(0,\frac{2\varrho}{\mu_{lm}})} \bigg(f^{'}_0(V_{1,0})\Phi_{1,0}^l\cdot
(\mu_{lm})^{\frac{N}{p+1}}\{(\chi-1)\Phi_m\}
(\mu_{lm}y+\xi_{lm})\nonumber\\
& \quad+ g^{'}_0(U_{1,0})\Psi_{1,0}^l\cdot
(\mu_{lm})^{\frac{N}{p+1}}\{(\chi-1)\Psi_m\}
(\mu_{lm}y+\xi_{lm})\bigg)dx \nonumber\\
= & \lim\limits_{m\rightarrow \infty}
O\Big((\mu_{m})^{[(N-2)p-2]\frac{p}{p+1}}
+(\mu_{m})^{\frac{Npq}{q+1}}\Big)
=0,
\end{align}
for $  l=0,\cdots,N$.
Further,
since $(\Phi_{1,0}^l,\Psi_{1,0}^l)\in E_{\mathbf{d},\boldsymbol{\xi}}^\perp$,
we deduce that
$
(\tilde{\Psi}_\infty,\tilde{\Phi}_\infty)=(0,0).
$
We obtain (\ref{uua}).

\textbf{Step 3.}
Let us prove that a contradiction arises.
First,  using Lemmas \ref{yy} and  \ref{sumbu2}, we obtain
\begin{align}\label{eqsda}
&\bigg|f^{'}_{\epsilon_m}\Big(\sum\limits_{i=1}^{k}PV_{im}
  \Big)\Phi_m\bigg|_{L^{\frac{p+1}{p}}(\Omega)}\nonumber\\
\leq & C \bigg|\Big[f^{'}_{\epsilon_m}
\Big(\sum\limits_{i=1}^{k}PV_{im}
  \Big)
  -f^{'}_0\Big(\sum\limits_{i=1}^{k}PV_{im}
  \Big)
  \Big]\Phi_m\bigg|_{L^{\frac{p+1}{p}}(\Omega)}\nonumber\\
 & +
  C \bigg|\Big[f^{'}_0\Big(\sum\limits_{i=1}^{k}PV_{im}
  \Big)
  -\sum\limits_{i=1}^{k}f^{'}_0(PV_{im})\Big]
  \Phi_m\bigg|_{L^{\frac{p+1}{p}}(\Omega)}
   + C  \sum\limits_{i=1}^{k}\Big|f^{'}_0(PV_{im})
  \Phi_m\Big|_{L^{\frac{p+1}{p}}(\Omega)}\nonumber\\
  \leq & C\bigg|f^{'}_{\epsilon_m}\Big(\sum_{i=1}^{k}PV_{im}\Big) -f^{'}_0\Big(\sum_{i=1}^{k}PV_{im}\Big) \bigg|_{L^{\frac{p+1}{p-1}}(\Omega)}
  |\Phi_m|_{L^{p+1}(\Omega)}\nonumber\\
& + C\bigg|f^{'}_0\Big(\sum_{i=1}^{k}PV_{im}\Big)
 -\sum_{i=1}^{k}f^{'}_0(PV_{im})\Big) \bigg|_{L^{\frac{p+1}{p-1}}(\Omega)}
 |\Phi_m|_{L^{p+1}(\Omega)}
 \nonumber\\
& + C  \sum\limits_{i=1}^{k}|f^{'}_0(PV_{im})
   |_{L^{\frac{p+1}{p-1}}(\Omega)}
   |\Phi_m|_{L^{p+1}(\Omega)}\nonumber\\
  \leq &
  O \Big(\mu_m^{(N-2)p}+\epsilon_m
  \mu_m^{\frac{Np}{p+1}}
   (\ln|\ln\mu_m|)\Big)
   \rightarrow0 \quad \mbox{as}\ m\rightarrow\infty.
\end{align}
Using the estimate  (\ref{pfs1}), a direct computation   yields that
\begin{align}\label{uaek}
&|g^{'}_{\epsilon_m} (\mathbf{P}U_m)
\Psi_m|_{L^{\frac{q+1}{q}}(\Omega)}\nonumber\\
 = &  |g^{'}_{\epsilon_m} (\mathbf{P}U_m)
\Big|_{L^{\frac{q+1}{q}}(\Omega)}
 |\Psi_m|_{L^{q+1}(\Omega)}\nonumber\\
 = &\Big|g^{'}_{\epsilon_m}(\mathbf{P}U_m)
 -  g^{'}_0(\mathbf{P}U_m) \Big|_{L^{\frac{q+1}{q-1}}(\Omega)}
 |\Psi_m|_{L^{q+1}(\Omega)}\nonumber\\
& +\Big|g^{'}_0(\mathbf{P}U_m)
 -  \sum_{i=1}^{k}g^{'}_0(U_{im})
 \Big|_{L^{\frac{q+1}{q-1}}(\Omega)}
 |\Psi_m|_{L^{q+1}(\Omega)}
 + \sum_{i=1}^{k}|g^{'}_0(U_{im})
 |_{L^{\frac{q+1}{q-1}}(\Omega)}
 |\Psi_m|_{L^{q+1}(\Omega)}\nonumber\\
  = & O \Big(\epsilon_m (\ln|\ln\mu_m|)+\mu_m^{\frac{Np}{q+1}}\Big)
  \rightarrow0 \quad \mbox{as}\ m\rightarrow\infty.
\end{align}

In fact, (\ref{linear})  can be write as
\[
\begin{cases}
-\Delta \Psi_m=
f^{'}_{\epsilon_m}\Big(\sum\limits_{i=1}^{k}PV_{im}\Big)
\Phi_m
-\Delta H_{1m}+\sum\limits_{i=1}^k
\sum\limits_{l=0}^Nc_{il,m} f^{'}_0(V_{im})\Phi^l_{im}
\ \ &\text{in}~\Omega,\\
-\Delta \Phi_m=  g^{'}_{\epsilon_m} (\mathbf{P}U_m)
\Psi_m
-\Delta H_{2m}+\sum\limits_{i=1}^k\sum\limits_{l=0}^Nc_{il,m} g^{'}_0(U_{im}) \Psi^l_{im}
\ \ &\text{in}~\Omega,\\
\Psi_m=\Phi_m=0\ \ &\text{on}~\partial\Omega.
\end{cases}
\]
Then, by (\ref{aye}),  (\ref{lfimi}), (\ref{eqsda}) and (\ref{uaek}), it holds
\begin{align*}
  1= & \|(\Phi_m,\Psi_m)\|  \\
  \leq &  C \Bigg[\bigg|
  f^{'}_{\epsilon_m}\Big(\sum\limits_{i=1}^{k}PV_{im}
  \Big)\Phi_m\bigg|_{L^{\frac{p+1}{p}}(\Omega)}
+\Big|g^{'}_{\epsilon_m} (\mathbf{P}U_m)\Psi_m\Big|_{L^{\frac{q+1}{q}}(\Omega)}
+\|(H_{1m},H_{2m})\|  \\
& +\sum\limits_{i=1}^k\sum\limits_{l=0}^Nc_{il,m}
\bigg(\Big|f^{'}_0(V_{im})
\Phi^l_{im}\Big|_{L^{\frac{p+1}{p}}(\Omega)}
+\Big|g^{'}_0(U_{im}) \Psi^l_{im}\Big|_{L^{\frac{q+1}{q}}(\Omega)}
\bigg)\Bigg]\\
\leq &  C \Bigg[\bigg|
  f^{'}_{\epsilon_m}\Big(\sum\limits_{i=1}^{k}PV_{im}
  \Big)\Phi_m\bigg|_{L^{\frac{p+1}{p}}(\Omega)}
+\Big|g^{'}_{\epsilon_m} (\mathbf{P}U_m)\Psi_m\Big|_{L^{\frac{q+1}{q}}(\Omega)}
+\|(H_{1m},H_{2m})\|  \\
 &  +\sum\limits_{i=1}^k\sum\limits_{l=0}^N c_{il,m}
\bigg(\Big| f_0^{'}(V_{1,0})\Phi^l_{1,0}
\Big|_{L^{\frac{p+1}{p}}(\Omega)}
+\Big|g^{'}_0(U_{1,0})  \Psi^l_{1,0}\Big|_{L^{\frac{q+1}{q}}(\Omega)}
+o(1)\bigg)\Bigg]\\
\leq & C\Bigg( \bigg|
  f^{'}_{\epsilon_m}\Big(\sum\limits_{i=1}^{k}PV_{im}
  \Big)\Phi_m\bigg|_{L^{\frac{p+1}{p}}(\Omega)}
+\Big|g^{'}_{\epsilon_m} (\mathbf{P}U_m)
\Psi_m\Big|_{L^{\frac{q+1}{q}}(\Omega)}\Bigg)\rightarrow0 \quad \mbox{as}\ m\rightarrow\infty.
\end{align*}
This is a contradiction.  Finally, we get the desired result.
\end{proof}

By using the invertibility of the operator  $L_{ \mathbf{d},\boldsymbol\xi}$,
we are in position to solve equation (\ref{ng2}).

\noindent{\it \textbf{Proof of Proposition \ref{pro21}}}:
First of all, we point out that $(\Phi,\Psi )$
solves equation  (\ref{ng2}) if and only if $(\Phi,\Psi )$ is   a fixed point of the map  $T_{ \mathbf{d},\boldsymbol\xi}: E^{\bot}_{ \mathbf{d},\boldsymbol\xi}\rightarrow E^{\bot}_{ \mathbf{d},\boldsymbol\xi}$ defined by
\begin{align*}
  T_{ \mathbf{d},\boldsymbol\xi}
  (\Phi,\Psi )
  = & L_{ \mathbf{d},\boldsymbol\xi}^{-1} \Pi^{\bot}_{ \mathbf{d},\boldsymbol\xi}\mathcal{I}^*
  \Bigg\{
     - \int_\Omega \bigg[f_\epsilon\Big(\sum_{i=1}^{k}PV_{i}
 +\Phi \Big)
 - f_\epsilon\Big(\sum_{i=1}^{k}PV_{i}\Big)
 -f^{'}_\epsilon\Big(\sum_{i=1}^{k}PV_{i}\Big)
 \Phi \bigg] dx \\
 & - \int_\Omega \bigg[f_\epsilon\Big(\sum_{i=1}^{k}PV_{i}\Big)
 -f_0\Big(\sum_{i=1}^{k}PV_{i}\Big)\bigg] dx\\
  & - \int_\Omega \bigg[
 f_0\Big(\sum_{i=1}^{k}PV_{i}\Big)
 -\sum_{i=1}^{k}f_0(PV_{i})\bigg]dx
  - \sum_{i=1}^{k}\int_\Omega \Big[
  f_0(PV_{i})-f_0(V_{i})\Big] dx \\
  & - \int_\Omega \bigg[f^{'}_\epsilon\Big(\sum_{i=1}^{k}PV_{i}\Big)
 -f^{'}_0\Big(\sum_{i=1}^{k}PV_{i}\Big)\bigg]
 \Phi  dx\\
  & - \int_\Omega \bigg[
 f^{'}_0\Big(\sum_{i=1}^{k}PV_{i}\Big)
 -\sum_{i=1}^{k}f^{'}_0(PV_{i})\bigg]
 \Phi dx
  - \sum_{i=1}^{k}\int_\Omega \Big[
  f^{'}_0(PV_{i})-f^{'}_0(V_{i})\Big]
  \Phi dx \\
&  - \int_\Omega \Big[g_\epsilon
(\mathbf{P}U_{\mathbf{d},\boldsymbol\xi}
    +\Psi )
 - g_\epsilon(\mathbf{P}U_{\mathbf{d},\boldsymbol\xi})
 -g^{'}_\epsilon
 (\mathbf{P}U_{\mathbf{d},\boldsymbol\xi})
 \Psi  \Big] dx \\
  &
  - \int_\Omega \Big[ g_\epsilon
  (\mathbf{P}U_{\mathbf{d},\boldsymbol\xi})
 -g_0(\mathbf{P}U_{\mathbf{d},\boldsymbol\xi}) \Big] dx
   - \int_\Omega \Big[ g_0(\mathbf{P}U_{\mathbf{d},\boldsymbol\xi})
 - \sum_{i=1}^{k}g_0(U_i) \Big] dx\\
&  - \int_\Omega \Big[ g^{'}_\epsilon
  (\mathbf{P}U_{\mathbf{d},\boldsymbol\xi})
 -  \sum_{i=1}^{k}g^{'}_0(U_i) \Big] \Psi dx\Bigg\}.
\end{align*}
Let
\[
\tilde{B }
  =\{(\Phi ,
  \Psi ) \in E^{\bot}_{ \mathbf{d},\boldsymbol\xi}:
\|(\Phi ,
  \Psi )\|\leq C^*R_\epsilon\},
\]
where $R_\epsilon = \epsilon
   (\ln|\ln\mu|)[\mu^{\frac{Np}{p+1}}+\mu^{\frac{Np}{q+1}}]
   +\mu^{\frac{Np}{q+1}}+\mu^{\frac{N(p-1)}{p+1}}$.
   We will show that
$
T_{ \mathbf{d},\boldsymbol\xi}:
\tilde{B }\rightarrow \tilde{B }
$
is a contraction mapping.

From  Lemma \ref{inver},  (\ref{embed}) and (\ref{esmbed}), we have
\begin{align*}
  \|T_{\mathbf{d},\boldsymbol\xi}
  (\Phi ,
  \Psi )\|
  \leq &  C \left| f_\epsilon\Big(\sum_{i=1}^{k}PV_{i}
 +\Phi \Big)
 - f_\epsilon\Big(\sum_{i=1}^{k}PV_{i}\Big)
 -f^{'}_\epsilon\Big(\sum_{i=1}^{k}PV_{i}\Big)
 \Phi    \right|_{L^{\frac{p+1}{p}}(\Omega)} \\
 & + C \bigg|f_\epsilon\Big(\sum_{i=1}^{k}PV_{i}\Big)
 -f_0\Big(\sum_{i=1}^{k}PV_{i}\Big)
 \bigg|_{L^{\frac{p+1}{p}}(\Omega)}\\
  & +\bigg|
 f_0\Big(\sum_{i=1}^{k}PV_{i}\Big)
 -\sum_{i=1}^{k}
 f_0(PV_{i})\bigg|_{L^{\frac{p+1}{p}}(\Omega)}
 + \sum_{i=1}^{k} \Big|
  f_0(PV_{i})-f_0(V_{i}) \Big|_{L^{\frac{p+1}{p}}(\Omega)} \\
  & +  \bigg|\Big[f^{'}_\epsilon
  \Big(\sum_{i=1}^{k}PV_{i}\Big)
 -f^{'}_0\Big(\sum_{i=1}^{k}PV_{i}\Big)\Big]
 \Phi  \bigg|_{L^{\frac{p+1}{p}}(\Omega)}\\
  & +  \bigg|\Big[
 f^{'}_0\Big(\sum_{i=1}^{k}PV_{i}\Big)
 -\sum_{i=1}^{k}f^{'}_0(PV_{i})\Big]
 \Phi \Big|_{L^{\frac{p+1}{p}}(\Omega)}
  - \sum_{i=1}^{k}
  \Big| [f^{'}_0(PV_{i})-f^{'}_0(V_{i})]
  \Phi
   \Big|_{L^{\frac{p+1}{p}}(\Omega)} \\
&  +  \Big|g_\epsilon
(\mathbf{P}U_{\mathbf{d},\boldsymbol\xi}
    +\Psi )
 -g_\epsilon
 (\mathbf{P}U_{\mathbf{d},\boldsymbol\xi})
 -g^{'}_\epsilon
 (\mathbf{P}U_{\mathbf{d},\boldsymbol\xi})
 \Psi  \Big|_{L^{\frac{q+1}{q}}(\Omega)}  \\
  &
  +  \Big| g_\epsilon
  (\mathbf{P}U_{\mathbf{d},\boldsymbol\xi})
 -g_0(\mathbf{P}U_{\mathbf{d},\boldsymbol\xi}) \Big|_{L^{\frac{q+1}{q}}(\Omega)}
  + \Big|g_0(\mathbf{P}U_{\mathbf{d},\boldsymbol\xi})
  -  \sum_{i=1}^{k}g_0(U_i) \bigg|_{L^{\frac{q+1}{q}}(\Omega)}\\
  & + \bigg|\Big[ g^{'}_\epsilon
  (\mathbf{P}U_{\mathbf{d},\boldsymbol\xi})
  -  \sum_{i=1}^{k}g^{'}_0(U_i) \Big] \Psi
 \bigg|_{L^{\frac{q+1}{q}}(\Omega)} \\
 = &   H_1+\cdots+H_{11}.
\end{align*}
We will estimate $H_1$-$H_{12}$ respectively.

\emph{Estimate of  $ H_1$}:
From  the mean value theorem, we choose $t=t(x)\in [0,1]$, then
\begin{align}\label{dkoa}
  H_1 = & \bigg|f_\epsilon\Big(\sum_{i=1}^{k}PV_{i}
 +\Phi \Big)
 -f_\epsilon\Big(\sum_{i=1}^{k}PV_{i}
\Big) -f^{'}_\epsilon\Big(\sum_{i=1}^{k}PV_{i}\Big) \Phi \bigg|_{L^{\frac{p+1}{p}}(\Omega)}\nonumber\\
   =  &  \bigg|\bigg[f^{'}_\epsilon\Big(\sum_{i=1}^{k}PV_{i}
 +t\Phi \Big)-f^{'}_\epsilon\Big(\sum_{i=1}^{k}PV_{i}\Big) \bigg]\Phi \bigg|_{L^{\frac{p+1}{p}}(\Omega)}.
\end{align}
When $n\leq 6$,  Lemma \ref{zsj}   follows that
\begin{align*}
  H_1 \leq
  & C\Big(  \Big||\Phi |^p\Big|_{L^{\frac{p+1}{p}}(\Omega)}
  +\Big|\Big(\sum_{i=1}^{k}PV_{i}\Big)^{p-2}
  \Phi ^2\Big|_{L^{\frac{p+1}{p}}(\Omega)} \Big)\\
   \leq  &  C \Big( |\Phi |^{p-2}_{L^{p+1}{(\Omega)}}
   +\Big|\sum_{i=1}^{k}PV_{i}\Big|_{L^{p+1}{(\Omega)}}^{p-2} \Big)|\Phi |^2_{L^{p+1}{(\Omega)}}
   =   C ( \|\Phi \|^{p-2}+1)\|\Phi \|^2.
\end{align*}
When $n>6$, there holds
\begin{align*}
  H_1 \leq & C\Big( \Big||\Phi |^p\Big|_{L^{\frac{p+1}{p}}(\Omega)}
  +\epsilon\Big|\Big(\sum_{i=1}^{k}PV_{i}\Big)^{p-1}
  \Phi  \Big|_{L^{\frac{p+1}{p}}(\Omega)} \Big)\\
   = & C \Big[ |\Phi |^p_{L^{p+1}{(\Omega)}}
   +\epsilon\bigg(\int_\Omega
   \Big[ \Big(\sum_{i=1}^{k}PV_{i}\Big)^{p-1}
   |\Phi |
   \Big]^{\frac{p+1}{p}} dx\bigg)
   ^{^{\frac{p}{p+1}}}\Big]\nonumber\\
   \leq  & C \Big( |\Phi |^p
   _{L^{p+1}{(\Omega)}}
   +\epsilon\Big|\sum_{i=1}^{k}PV_{i}\Big|^{p-1}_{L^{p+1}{(\Omega)}} |\Phi |_{L^{p+1}
   {(\Omega)}} \Big)\\
   \leq  &  C \Big( |\Phi |^{p-1}_{L^{p+1}{(\Omega)}}
   +\epsilon\sum_{i=1}^{k}|PV_{i}|^{p-1}_{L^{p+1}{(\Omega)}}  \Big)|\Phi |_{L^{p+1}{(\Omega)}}\\
   \leq & C ( \|\Phi \|^{p-1} +\epsilon  )\|\Phi \|.
\end{align*}
Sum up  these estimates, we have
\begin{eqnarray}\label{dt}
H_1 \leq
\left\{ \arraycolsep=1.5pt
  \begin{array}{lll}
C  ( \|\Phi \|^{p-2}+1 )\|\Phi \|^2 \ \   &{\rm if}\  3\leq N\leq 6, \\[2mm]
C  ( \|\Phi \|^{p-1} +\epsilon  )\|\Phi \|\ \   &{\rm if}\  N>6. \\[2mm]
\end{array}
\right.
\end{eqnarray}

\emph{Estimate of  $ H_2$}:
From (\ref{onran}), we get
\begin{align*}
 H_2= \bigg|f_\epsilon\Big(\sum_{i=1}^{k}PV_{i}\Big)
 -f_0\Big(\sum_{i=1}^{k}PV_{i}\Big)
 \bigg|_{L^{\frac{p+1}{p}}(\Omega)}
 = O \Big(\epsilon\mu^{\frac{Np}{p+1}}
   (\ln|\ln\mu|)\Big).
\end{align*}

\emph{Estimate of  $ H_3$}:
The    Lemma \ref{ysy} allows us to deduce
\[
\bigg|
 f_0\Big(\sum_{i=1}^{k}PV_{i}\Big)
 -\sum_{i=1}^{k}
 f_0(PV_{i})\bigg|_{L^{\frac{p+1}{p}}(\Omega)}
 =O (\mu^{(N-2)p}).
\]

\emph{Estimate of  $ H_4$}:
By (\ref{sumbu3}), there holds
 \begin{align*}\label{uio}
   H_4=   \Big|
  f_0(PV_{i})-f_0(V_{i}) \Big|_{L^{\frac{p+1}{p}}(\Omega)}
  =O(\mu^{\frac{N}{q+1}}).
 \end{align*}
 
\emph{Estimate of  $ H_5$}:
Indeed, using   (\ref{fepli2}), we get
\begin{align*}
H_5= &  \bigg|\Big[f^{'}_\epsilon
  \Big(\sum_{i=1}^{k}PV_{i}\Big)
 -f^{'}_0\Big(\sum_{i=1}^{k}PV_{i}\Big)\Big]
 \Phi  \bigg|_{L^{\frac{p+1}{p}}(\Omega)}\\
 \leq &   \Big|f_\epsilon^{'}\Big(\sum_{i=1}^{k}PV_{i}\Big)
  - f_0^{'}\Big(\sum_{i=1}^{k}PV_{i}\Big)\Big|
  _{L^{\frac{p+1}{p-1}}(\Omega)}
  |\Phi  |_{L^{p+1}{(\Omega)}}
  = O \Big(\epsilon\mu^{\frac{Np}{p+1}}
   (\ln|\ln\mu|)\Big).
\end{align*}

\emph{Estimate of  $ H_6$}:
In view of  Lemma \ref{sumbu2},  we conclude
\begin{align*}
  H_6= & \bigg|\Big[
 f^{'}_0\Big(\sum_{i=1}^{k}PV_{i}\Big)
 -\sum_{i=1}^{k}f^{'}_0(PV_{i})\Big]
 \Phi \bigg| \\
 \leq &   \Big|f^{'}_0\Big(\sum_{i=1}^{k}PV_{i}\Big)
 -\sum_{i=1}^{k}f^{'}_0(PV_{i})\Big|
  _{L^{\frac{p+1}{p-1}}(\Omega)}
  |\Phi  |_{L^{p+1}{(\Omega)}}
  = O \Big(\mu^{(N-2)p}
   (\ln|\ln\mu|)\Big).
\end{align*}

\emph{Estimate of  $ H_7$}:
From (\ref{sumfbu2}),
one has
\begin{align*}
 H_7= &  \Big| [f^{'}_0(PV_{i})-f^{'}_0(V_{i})]
  \Phi
   \Big|_{L^{\frac{p+1}{p}}(\Omega)}\\
 \leq &   \Big|f^{'}_0(PV_{i})-f^{'}_0(V_{i})\Big|
  _{L^{\frac{p+1}{p-1}}(\Omega)}
  |\Phi  |_{L^{p+1}{(\Omega)}}
  = O \Big(\mu^{(N-2)p}
   (\ln|\ln\mu|)\Big).
\end{align*}

\emph{Estimate of  $ H_8$}:
By similar calculation of $ H_1$, we have
\begin{eqnarray}\label{dt}
 H_8&=& \Big|g_\epsilon
(\mathbf{P}U_{\mathbf{d},\boldsymbol\xi}
    +\Psi )
 -g_\epsilon
 (\mathbf{P}U_{\mathbf{d},\boldsymbol\xi})
 -g^{'}_\epsilon
 (\mathbf{P}U_{\mathbf{d},\boldsymbol\xi})
 \Psi  \Big|_{L^{\frac{q+1}{q}}(\Omega)}\nonumber\\
&\leq &
\left\{ \arraycolsep=1.5pt
  \begin{array}{lll}
C  ( \|\Psi \|^{p-2}+1 )\|\Psi \|^2 \ \   &{\rm if}\  3\leq N\leq 6, \\[2mm]
C  ( \|\Psi \|^{p-1} +\epsilon  )\|\Psi \|\ \   &{\rm if}\  N>6. \\[2mm]
\end{array}
\right.
\end{eqnarray}

\emph{Estimate of  $ H_9$}:
From (\ref{hdua}), it holds
\begin{align*}
 H_9=  \Big| g_\epsilon
  (\mathbf{P}U_{\mathbf{d},\boldsymbol\xi})
 -g_0(\mathbf{P}U_{\mathbf{d},\boldsymbol\xi}) \Big|_{L^{\frac{q+1}{q}}(\Omega)}
 =O \Big(\epsilon\mu^{\frac{Np}{q+1}}
   (\ln|\ln\mu|)\Big).
\end{align*}

\emph{Estimate of  $ H_{10}$}:
From Lemma \ref{pfS1}, we have
\begin{align*}
 H_{10}= \Big|g_0(\mathbf{P}U_{\mathbf{d},\boldsymbol\xi})
 -\sum_{i=1}^{k}g_0(U_i)
 \Big|_{L^{\frac{q+1}{q}}(\Omega)}
  = O (\mu^{\frac{Np}{q+1}}).
\end{align*}

\emph{Estimate of  $ H_{11}$}:
From (\ref{pfs1}), we get
\begin{align*}
H_{11}=  &   \bigg|\Big[ g^{'}_\epsilon
  (\mathbf{P}U_{\mathbf{d},\boldsymbol\xi})
 -  \sum_{i=1}^{k}g^{'}_0(U_i) \Big] \Psi
 \bigg|_{L^{\frac{q+1}{q}}(\Omega)}\\
 \leq &   \Big|g^{'}_\epsilon(\mathbf{P}U_{\mathbf{d},\boldsymbol\xi})
 -  \sum_{i=1}^{k}g^{'}_0(U_i) \Big|_{L^{\frac{q+1}{q-1}}(\Omega)}
  |\Psi  |_{L^{q+1}{(\Omega)}}
  =O \Big(\epsilon \mu^{\frac{Np}{q+1}}
   (\ln|\ln\mu|)\Big).
\end{align*}
From $H_1$-$H_{11}$,
there are constants $C^*>0$ and $ \mu_0>0$ such that  for each $ \mu\in(0, \mu_0)$, we obtain
\[
\|T_{\mathbf{d},\boldsymbol\xi}
  (\Phi ,
  \Psi )\|
  \leq C^*R_\epsilon
\quad
 \mbox{for\ every}\
 (\Phi,\Psi )\in \tilde{B }.
\]
Finally, we prove that  $T_{ \mathbf{d},\boldsymbol\xi}$ is a contraction map.
If $(\Phi_1,\Psi_1)$,
$(\Phi_2,\Psi_2)\in \tilde{B }$,
by the similar computations to $H_1$-$H_{11}$,
there exists a constant $L^*\in (0,1)$ such that
\[  \|T_{ \mathbf{d},\boldsymbol\xi}(\Phi_2,\Psi_2)-T_{ \mathbf{d},\boldsymbol\xi}(\Phi_1,\Psi_1) \|
   \leq L^* \|\Phi_2-\Phi_1 \|.
\]
It follows that
$T_{ \mathbf{d},\boldsymbol\xi}$  is a contraction mapping from $\tilde{B }$ to $ \tilde{B }$,
then,
it has a unique fixed point  $(\Phi,\Psi )\in \tilde{B }$.
This concludes the proof.
\qed

\section{Proof of  Proposition \ref{leftside}}

This section is devoted to prove Proposition \ref{leftside}.

\noindent{\emph{\textbf{Proof of  Part $a$}}}.
Since this procedure is carried out in a standard way, we omit the proofs and refer to
Proposition 4.7 in \cite{kpm} for technical details.

\noindent{\emph{\textbf{Proof of  Part $b$}}}.
For $j=1,\cdots,k$ and $h=0,\cdots,N$, by multiplying with $ (P\Phi_{jh},
    P\Psi_{jh})$ on both sides of (\ref{ccf}),
 using   (\ref{oeo}) and (\ref{oeo1}), there holds
\begin{align*}
  & \bigg\langle \Big(\mathbf{P}U_{\mathbf{d},\boldsymbol\xi}
   +\Psi_{\mathbf{d},\boldsymbol\xi},
   \sum_{i=1}^{k}PV_{i}+\Phi_{\mathbf{d},\boldsymbol\xi}\Big)
-\mathcal{I}^*\bigg[\bigg(g_\epsilon (\mathbf{P}U_{\mathbf{d},\boldsymbol\xi}
    +\Psi_{\mathbf{d},\boldsymbol\xi}),
    f_\epsilon\Big(\sum_{i=1}^{k}PV_{i}
    +\Phi_{\mathbf{d},\boldsymbol\xi}\Big)\bigg)\bigg],
    (P\Phi_{jh},
    P\Psi_{jh})\bigg\rangle\\
  = &  \Big\langle \Big(\mathbf{P}U_{\mathbf{d},\boldsymbol\xi},
   \sum_{i=1}^{k}PV_{i}\Big),
   (P\Phi_{jh},
    P\Psi_{jh}) \Big\rangle
      -\int_\Omega
      \bigg[
       g_\epsilon(\mathbf{P}U_{\mathbf{d},\boldsymbol\xi}
    +\Psi_{\mathbf{d},\boldsymbol\xi})P\Psi_{jh}
    +
    f_\epsilon\Big(\sum_{i=1}^{k}PV_{i}
    +\Phi_{\mathbf{d},\boldsymbol\xi}\Big)P\Phi_{jh}
    \bigg]dx \\
  = & \int_\Omega f_0(\sum_{i=1}^{k}PV_{i})P\Phi_{jh} dx
-  \int_\Omega f_\epsilon\Big(\sum_{i=1}^{k}PV_{i}
 +\Phi_{\mathbf{d},\boldsymbol\xi}\Big)P\Phi_{jh}dx \\
  & + \sum_{i=1}^k \int_\Omega g_0(U_i)P\Psi_{jh}dx
    - \int_\Omega g_\epsilon(\mathbf{P}U_{\mathbf{d},\boldsymbol\xi}
    +\Psi_{\mathbf{d},\boldsymbol\xi})P\Psi_{jh}dx\\
  = &\mathcal{ M}_1+\mathcal{M}_2
  -   \sum_{i=1}^{k}\int_\Omega
  \Big[f^{'}_0( V_{i})\Phi_{\mathbf{d},\boldsymbol\xi} P\Phi_{jh}
  +g^{'}_0(U_i) \Psi_{\mathbf{d},\boldsymbol\xi} P\Psi_{jh} \Big] dx.
\end{align*}
Where
\begin{align*}
\mathcal{M}_1
  = &   - \int_\Omega \bigg[f_\epsilon\Big(\sum_{i=1}^{k}PV_{i}
 +\Phi_{\mathbf{d},\boldsymbol\xi}\Big)
 - f_\epsilon\Big(\sum_{i=1}^{k}PV_{i}\Big)
 -f^{'}_\epsilon\Big(\sum_{i=1}^{k}PV_{i}\Big)
 \Phi_{\mathbf{d},\boldsymbol\xi} \bigg]P\Phi_{jh} dx \\
 & - \int_\Omega \bigg[f_\epsilon\Big(\sum_{i=1}^{k}PV_{i}\Big)
 -f_0\Big(\sum_{i=1}^{k}PV_{i}\Big)\bigg]P\Phi_{jh} dx
   - \int_\Omega \bigg[f^{'}_\epsilon\Big(\sum_{i=1}^{k}PV_{i}\Big)
 -f^{'}_0\Big(\sum_{i=1}^{k}PV_{i}\Big)\bigg]\Phi_{\mathbf{d},\boldsymbol\xi} P\Phi_{jh} dx \\
 & - \int_\Omega \bigg[
 f^{'}_0\Big(\sum_{i=1}^{k}PV_{i}\Big)
 -\sum_{i=1}^{k}f^{'}_0(PV_{i})\bigg]\Phi_{\mathbf{d},\boldsymbol\xi} P\Phi_{jh} dx
  - \sum_{i=1}^{k}\int_\Omega \Big[
  f^{'}_0(PV_{i})-f^{'}_0(V_{i})\Big]\Phi_{\mathbf{d},\boldsymbol\xi} P\Phi_{jh} dx \\
    = & P_1+\cdots+P_5,
\end{align*}
and
\begin{align*}
\mathcal{M}_2
  = &
   - \int_\Omega \Big[g_\epsilon(\mathbf{P}U_{\mathbf{d},\boldsymbol\xi}
    +\Psi_{\mathbf{d},\boldsymbol\xi})
 - g_\epsilon(\mathbf{P}U_{\mathbf{d},\boldsymbol\xi})
 -g^{'}_\epsilon(\mathbf{P}U_{\mathbf{d},\boldsymbol\xi})
 \Psi_{\mathbf{d},\boldsymbol\xi} \Big]P\Psi_{jh} dx \\
  & - \int_\Omega \Big[ g_\epsilon(\mathbf{P}U_{\mathbf{d},\boldsymbol\xi})
 -g_0(\mathbf{P}U_{\mathbf{d},\boldsymbol\xi}) \Big](P\Psi_{jh}-\Psi_{jh}) dx
  - \int_\Omega \Big[ g_\epsilon(\mathbf{P}U_{\mathbf{d},\boldsymbol\xi})
 -g_0(\mathbf{P}U_{\mathbf{d},\boldsymbol\xi}) \Big]\Psi_{jh} dx \\
 & - \int_\Omega \Big[ g_0(\mathbf{P}U_{\mathbf{d},\boldsymbol\xi})
 -  \sum_{i=1}^{k}g_0(U_i) \Big](P\Psi_{jh}-\Psi_{jh})dx
  - \int_\Omega \Big[ g_0(\mathbf{P}U_{\mathbf{d},\boldsymbol\xi})
 -  \sum_{i=1}^{k}g_0(U_i) \Big]\Psi_{jh} dx \\
 & - \int_\Omega \Big[ g^{'}_\epsilon(\mathbf{P}U_{\mathbf{d},\boldsymbol\xi})
 -  \sum_{i=1}^{k}g^{'}_0(U_i) \Big]\Psi_{\mathbf{d},\boldsymbol\xi} P\Psi_{jh} dx
 -  \sum_{i=1}^{k} \int_\Omega g_0(U_i)(P\Psi_{jh}-\Psi_{jh})dx\\
    = & Q_1+\cdots+Q_7.
\end{align*}
We next estimate each term as follows.

\emph{Estimate of $ P_1$}:
From (\ref{dkoa}) and (\ref{dt}), we obtain
\begin{eqnarray*}
  P_1 & = &
  \int_\Omega \Big[f_\epsilon\Big(\sum_{i=1}^{k}PV_{i}
 +\Phi_{\mathbf{d},\boldsymbol\xi}\Big)
 -f_\epsilon\Big(\sum_{i=1}^{k}PV_{i}\Big) -f^{'}_\epsilon\Big(\sum_{i=1}^{k}PV_{i}\Big) \Phi_{\mathbf{d},\boldsymbol\xi}  \Big]P\Phi_{jh}dx\\
  & = & O\bigg(\bigg|f_\epsilon\Big(\sum_{i=1}^{k}PV_{i}
 +\Phi_{\mathbf{d},\boldsymbol\xi}\Big)-f_\epsilon\Big(\sum_{i=1}^{k}PV_{i}\Big) -f^{'}_\epsilon\Big(\sum_{i=1}^{k}PV_{i}\Big) \Phi_{\mathbf{d},\boldsymbol\xi} \bigg|_{L^{\frac{p+1}{p}}(\Omega)}
   |P\Phi_{jh} |_{L^{p+1}(\Omega)}\bigg)\\
 & \leq &\left\{ \arraycolsep=1.5pt
   \begin{array}{lll}
  O(1+\|\Phi_{\mathbf{d},\boldsymbol\xi}\|^{p-2})\|\Phi_{\mathbf{d},\boldsymbol\xi}\|^2\ \   &{\rm if}\ 3\leq N\leq6, \\[2mm]
  O(\epsilon+\|\Phi_{\mathbf{d},\boldsymbol\xi}\|^{p-1})
  \|\Phi_{\mathbf{d},\boldsymbol\xi}\|\ \   &{\rm if}\   N>6, \\[2mm]
 \end{array}
 \right.\\
 &  = &
  O\Big(\mu^{(N-2)p-1}(\ln|\ln\mu|)\Big)
  =O(\mu^{(N-2)p-2}).
 \end{eqnarray*}

\emph{Estimate of $P_2$}:
By the results of Lemma \ref{pf1R},  we have
\begin{eqnarray*}
P_2&  = &
- \int_\Omega \bigg[f_\epsilon\Big(\sum_{i=1}^{k}PV_{i}\Big)
 -f_0\Big(\sum_{i=1}^{k}PV_{i}\Big)\bigg]P\Phi_{jh} dx \nonumber\\
   &  = &
\left\{ \arraycolsep=1.5pt
   \begin{array}{lll}
  -\frac{p+1}{N} \mathcal{\tilde{A}}_1
  \sum\limits_{i=1}^k \frac\epsilon{|\ln\mu_i|}
  + O\Big(\sum\limits_{i=1}^k \frac\epsilon{|\ln\mu_i|}
  +
  \epsilon\mu^{\frac{Nq}{q+1}}
   (\ln|\ln\mu|)\Big) \ \   \ & {\rm if}\  h=0, \\[2mm]
  O \Big(\epsilon\mu^{\frac{Nq}{q+1}}
   (\ln|\ln\mu|)\Big)\ \  \  & {\rm if}\  h=1,\cdots,N. \\[2mm]
\end{array}
\right.
\end{eqnarray*}

\emph{Estimate of $P_3$}:
The calculations  (\ref{gisewr2}), (\ref{phi}),  (\ref{subu2})  and  (\ref{fepli2}) assert that
 \begin{align*}
  P_3  = &
  \int_\Omega\bigg|\Big [f^{'}_\epsilon\Big(\sum_{i=1}^{k}PV_{i}\Big)
  -f^{'}_0\Big(\sum_{i=1}^{k}PV_{i}\Big)\Big] \Phi_{\mathbf{d},\boldsymbol\xi} P\Phi_{jh}\bigg|dx\\
   = & O\bigg( \bigg|f^{'}_\epsilon\Big(\sum_{i=1}^{k}PV_{i}\Big) -f^{'}_0\Big(\sum_{i=1}^{k}PV_{i}\Big) \bigg|_{L^{\frac{p+1}{p-1}}(\Omega)}
   |\Phi_{\mathbf{d},\boldsymbol\xi}|_{L^{p+1}(\Omega)}
   |P\Phi_{jh}|_{L^{p+1}(\Omega)}\bigg)\\
  = &
  O \Big(\epsilon\mu^{\frac{Np}{p+1}}
   (\ln|\ln\mu|)\|\Phi_{\mathbf{d},\boldsymbol\xi}\|\Big)
   =O(\mu^{(N-2)p-2}).
 \end{align*}

\emph{Estimate of $P_4$}:
 Using the
estimates (\ref{gisewr2}), (\ref{phi}), (\ref{subu2}) and (\ref{sumbu2}) we have
\begin{align*}
  P_4  = &
  \int_\Omega\bigg|\Big [f^{'}_0\Big(\sum_{i=1}^{k}PV_{i}\Big)
 -\sum_{i=1}^{k}f^{'}_0(PV_{i})\Big)\Big] \Phi_{\mathbf{d},\boldsymbol\xi} P\Phi_{jh} \bigg|dx\\
   = & O\bigg( \bigg|f^{'}_0\Big(\sum_{i=1}^{k}PV_{i}\Big)
 -\sum_{i=1}^{k}f^{'}_0(PV_{i})\Big) \bigg|_{L^{\frac{p+1}{p-1}}(\Omega)}
   |\Phi_{\mathbf{d},\boldsymbol\xi} |_{L^{p+1}(\Omega)}
   |P\Phi_{jh}|_{L^{p+1}(\Omega)}\bigg)\\
  = &
  O \Big(\mu^{(N-2)p-1}\|\Phi_{\mathbf{d},\boldsymbol\xi}\|\Big)
=  O(\mu^{(N-2)p-2}).
 \end{align*}

 \emph{Estimate of $P_5$}:
 We have
\begin{align}\label{aegs}
  P_5= &
  \int_\Omega\bigg|[
  f^{'}_0(PV_{i})-f^{'}_0(V_{i})]\Phi_{\mathbf{d},\boldsymbol\xi} P\Phi_{jh} \bigg|dx\nonumber\\
   = & O\bigg( \Big|[f^{'}_0(PV_{i})-f^{'}_0(V_{i})] P\Phi_{jh}
   \Big|_{L^{\frac{p+1}{p}}(\Omega)}
   |\Phi_{\mathbf{d},\boldsymbol\xi}|_{L^{p+1}(\Omega)}\bigg).
 \end{align}
 Let $B_i=B(\xi_i,r)=\{x\in\mathbb{R}^N:|x-\xi_i|<r\}$ for  $r>0$.
By Lemmas \ref{ls}-\ref{uaoe}, (\ref{bubble}),
the mean value theorem,  there exists $t=t(x)\in [0,1]$ such that
\begin{align*}
  &  \Big(\int_\Omega\Big|[f^{'}_0(PV_{i})-f^{'}_0(V_{i})] P\Phi_{jh}
   \Big|^{\frac{p+1}{p}}dx\Big)^{\frac{p}{p+1}}\\
   =  &  \Big(\int_\Omega
   \Big|[ (PV_{i})^{p-1}- V_{i}^{p-1}] P\Phi_{jh}
   \Big|^{\frac{p+1}{p}}dx\Big)^{\frac{p}{p+1}}\\
  \leq  & C\Big( \int_{B_i}
   \Big|\Big(V_i+t (PV_{i}-V_{i})\Big)^{p-2} (PV_{i}-V_{i}) P\Phi_{jh}
   \Big|^{\frac{p+1}{p}}dx\Big)^{\frac{p}{p+1}}
   +O (\mu^{(N-2)p-2})\\
   \leq  & C\Big( \int_{B_i}
   \Big|V_i^{p-2} (PV_{i}-V_{i}) \Phi_{jh}
   \Big|^{\frac{p+1}{p}}dx\Big)^{\frac{p}{p+1}}
   +O (\mu^{(N-2)p-2})\\
   \leq  & C\bigg( \int_{B_i}
   \Big|\mu_i^{-\frac{N}{p+1}(p-2)}V^{p-2}_{1,0} (\frac{x-\xi_i}{\mu_i} )
   \Big(-(\frac{b_{N,p}}{\gamma_N})
 \mu_i^{\frac{N}{q+1}}H(x,\xi_i)
 +o(\mu_i^{\frac{N}{q+1}}) \Big)
  \mu_j^{-\frac{N}{p+1}}V_{1,0} (\frac{x-\xi_j}{\mu_j} )
   \Big|^{\frac{p+1}{p}}dx\bigg)^{\frac{p}{p+1}}\\
   &+O (\mu^{(N-2)p-2})\\
   \leq  &  O(\mu^{N-2+\frac{N}{p+1}}).
 \end{align*}
 Therefore,
 \begin{equation}\label{aesgs}
  P_5 =
  O \Big(\mu^{(N-2)p-1}\|\Phi_{\mathbf{d},\boldsymbol\xi}\|\Big)
  =O(\mu^{(N-2)p-2}).
 \end{equation}

\emph{Estimate of $Q_1$}:
We argue exactly as in the proof of  $ P_1$, then
\[
Q_1=
\int_\Omega \bigg|\Big[g_\epsilon(\mathbf{P}U_{\mathbf{d},\boldsymbol\xi}
    +\Psi_{\mathbf{d},\boldsymbol\xi})
 - g_\epsilon(\mathbf{P}U_{\mathbf{d},\boldsymbol\xi})
 -g^{'}_\epsilon(\mathbf{P}U_{\mathbf{d},\boldsymbol\xi}
 \Psi_{\mathbf{d},\boldsymbol\xi} \Big]P\Psi_{jh} \bigg|dx
 =O(\mu^{(N-2)p-2}).
\]

\emph{Estimate of $Q_2$}:
The proof is similar to (\ref{mrfa}) and by (\ref{onran}), we obtain
\begin{align}\label{hdua}
Q_2= & \int_\Omega \Big[ g_\epsilon(\mathbf{P}U_{\mathbf{d},\boldsymbol\xi})
 -g_0(\mathbf{P}U_{\mathbf{d},\boldsymbol\xi}) \Big](P\Psi_{jh}-\Psi_{jh}) dx\nonumber\\
 \leq  &  C\Big| g_\epsilon(\mathbf{P}U_{\mathbf{d},\boldsymbol\xi})
 -g_0(\mathbf{P}U_{\mathbf{d},\boldsymbol\xi})\Big|_{L^{\frac{q+1}{q}}(\Omega)}
   |P\Psi_{jh}-\Psi_{jh}|_{L^{q+1}(\Omega)}
 =O \Big(\epsilon\mu^{\frac{Np}{q+1}}
   (\ln|\ln\mu|)\Big).
\end{align}

\emph{Estimate of $Q_3$}:
From Lemma \ref{pf1}, we have
\begin{eqnarray*}
Q_3 &  = &
- \int_\Omega \Big[ g_\epsilon(\mathbf{P}U_{\mathbf{d},\boldsymbol\xi})
 -g_0(\mathbf{P}U_{\mathbf{d},\boldsymbol\xi}) \Big]\Psi_{jh} dx \nonumber\\
   &  = &
\left\{ \arraycolsep=1.5pt
   \begin{array}{lll}
  -\frac{q+1}{N} \mathcal{A}_1
  \sum\limits_{i=1}^k \frac\epsilon{|\ln\mu_i|}
  + O\Big(\sum\limits_{i=1}^k \frac\epsilon{|\ln\mu_i|}
  +\epsilon\mu^{\frac{Np}{q+1}} \ln|\ln\mu|\Big)\ \   \ & {\rm if}\  h=0, \\[2mm]
  O\Big(\epsilon\mu^{\frac{Np}{q+1}} \ln|\ln\mu|\Big)\ \  \  & {\rm if}\  h=1,\cdots,N. \\[2mm]
\end{array}
\right.
\end{eqnarray*}

\emph{Estimate of $Q_4$}:
From  Lemma  \ref{pfS1} and (\ref{gisewr1}), one has
\begin{align*}
Q_4= & \int_\Omega \Big[ g_0(\mathbf{P}U_{\mathbf{d},\boldsymbol\xi})
 -  \sum_{i=1}^{k}g_0(U_i) \Big](P\Psi_{jh}-\Psi_{jh})dx\\
 \leq  &  C\Big|g_0(\mathbf{P}U_{\mathbf{d},\boldsymbol\xi})
 -  \sum_{i=1}^{k}g_0(U_i)\Big|_{L^{\frac{q+1}{q}}(\Omega)}
   |P\Psi_{jh}-\Psi_{jh}|_{L^{q+1}(\Omega)}
   =O\Big(\epsilon\mu^{\frac{Np}{q+1}}\Big).
\end{align*}

\emph{Estimate of $Q_5$}:
We show the main result in the following,
and the proof is given in   Lemma \ref{poi},
\begin{eqnarray*}
Q_5 &  = &
- \int_\Omega \Big[ g_0(\mathbf{P}U_{\mathbf{d},\boldsymbol\xi})
 -  \sum_{i=1}^{k}g_0(U_i) \Big]\Psi_{jh}dx
 \nonumber\\
   &  = &
\left\{ \arraycolsep=1.5pt
   \begin{array}{lll}
  \Big(\frac{b_{N,p}}{\gamma_N}\Big)^p \mathcal{A}_2
\mu^{\frac{N(p+1)}{q+1}}
\sum\limits_{i=1}^kd_i^{\frac{N}{q+1}}
\widetilde{H}_{\mathbf{d},\boldsymbol{\xi}}(\xi_i)\\
\quad
-a_{N,p}\mathcal{A}_4
\mu^{(N-2)p-2}
\sum\limits_{j\neq i}^k\frac{d_{i}^{\frac{2N}{q+1}}
d_{j}^{\frac{N(p-1)}{q+1}}}{|\xi_{i}-\xi_{j}|^{(N-2)p-2}}
  + O(\mu^{(N-2)p-1})
  \ \   \  {\rm if}\  h=0, \\[2mm]
 \frac{1}{2}
  \Big(\frac{b_{N,p}}{\gamma_N}\Big)^p
  \mathcal{A}_3
  \mu^{\frac{N(p+1)}{q+1}+1}
\sum\limits_{i=1}^kd_i^{\frac{N}{q+1}+1}
\partial_{\xi_{ih}}\tilde{\rho}(\xi_i)
+ O(\mu^{(N-2)p-1})
  \ \  \   {\rm if}\  h=1,\cdots,N. \\[2mm]
\end{array}
\right.
\end{eqnarray*}

\emph{Estimate of $Q_6$}:
From (\ref{gisewr1}), (\ref{phi}), (\ref{subu3}) and (\ref{pfs1}), we have
\begin{align*}
Q_6 = &
  \int_\Omega\bigg| \Big[ g^{'}_\epsilon(\mathbf{P}U_{\mathbf{d},\boldsymbol\xi})
 -  \sum_{i=1}^{k}g^{'}_0(U_i) \Big]\Psi_{\mathbf{d},\boldsymbol\xi} P\Psi_{jh}  \bigg|dx\\
   = & O\bigg( \Big|g^{'}_\epsilon(\mathbf{P}U_{\mathbf{d},\boldsymbol\xi})
 -  \sum_{i=1}^{k}g^{'}_0(U_i) \Big|_{L^{\frac{q+1}{q-1}}(\Omega)}
   |P\Psi_{jh}|_{L^{q+1}(\Omega)}
   |\Psi_{\mathbf{d},\boldsymbol\xi}|_{L^{q+1}(\Omega)}\bigg)\\
  = &  O \Big(\epsilon \mu^{\frac{Np}{q+1}}
 (\ln|\ln\mu|)\Big).
 \end{align*}

 \emph{Estimate of $Q_7$}:
In virtu of  (\ref{gisewr1}), (\ref{bubble}) and Lemma \ref{ls},  one has
\begin{align*}
Q_7 =
 \sum_{i=1}^{k} \int_\Omega g_0(U_i)(P\Psi_{jh}-\Psi_{jh})dx
 \leq     C\sum_{i=1}^{k}\Big|  g_0(U_i)\Big|_{L^{\frac{q+1}{q}}(\Omega)}
   |P\Psi_{jh}-\Psi_{jh}|_{L^{q+1}(\Omega)}
   =O\Big(\epsilon\mu^{\frac{Np}{q+1}}\Big).
\end{align*}
Together above estimates, we get the results.
\qed

\appendix

\section{ Some estimates}

In this appendix,
we collect some estimates,
which  play
an important role in the proof of our main results.
We first give a lemma concerning a precise estimate related to
$\Phi_{jh}$ and $\Psi_{jh}$ for
$ h=0,\cdots,N$
and   $j=1,\cdots,k$.

\begin{lemma}\label{uwal}
There holds
\begin{eqnarray}\label{subu2}
 |\Phi_{jh}|_{L^{p+1}(\Omega)}=
\left\{ \arraycolsep=1.5pt
   \begin{array}{lll}
 O \Big( \mu^{N-2-\frac{N}{p+1}}\Big)\ \   &{\rm if }\  h=0, \\[2mm]
  O \Big( \mu^{N-1-\frac{N}{p+1}}\Big)
\ \  & {\rm if}\   h=1,\cdots,N,
\end{array}
\right.
\end{eqnarray}
and
\begin{eqnarray}\label{subu3}
|\Psi_{jh}|_{L^{q+1}(\Omega)}=
\left\{ \arraycolsep=1.5pt
   \begin{array}{lll}
 O \Big( \mu^{N-2-\frac{N}{q+1}}\Big)\ \   &{\rm if }\  h=0, \\[2mm]
  O \Big( \mu^{N-1-\frac{N}{q+1}}\Big)
\ \  & {\rm if}\   h=1,\cdots,N,
\end{array}
\right.
\end{eqnarray}
for  $j=1,\cdots,k$.
\end{lemma}

\begin{lemma} \label{inerpro}
For $i$, $j=1,\cdots,k$  and  $h$, $l=0,\cdots,N$,  it holds
\begin{eqnarray*}
 \Big\langle (P\Phi_{il},P\Psi_{il}),
(P\Phi_{jh},P\Psi_{jh})\Big\rangle
=
 \left\{ \arraycolsep=1.5pt
   \begin{array}{lll}
 C_h(1+o(1))\ \   &{\rm if }  \ j=i \ {\rm and}\ l=h, \\[2mm]
 o(1)\ \  & {\rm   else}.
\end{array}
\right.
\end{eqnarray*}

\end{lemma}

\begin{proof}
By (\ref{inne}), we have
\begin{align}\label{ssia}
\Big\langle (P\Phi_{il},P\Psi_{il}),
(P\Phi_{jh},P\Psi_{jh})\Big\rangle
=  & \int_\Omega\Big[pV_i^{p-1}\Phi_{il}(P\Phi_{jh}-\Phi_{jh})
+qU_i^{q-1}\Psi_{il}(P\Psi_{jh}-\Psi_{jh})\Big]dx\nonumber\\
& +
\int_\Omega\Big[pV_i^{p-1}\Phi_{il}\Phi_{jh}
+qU_i^{q-1}\Psi_{il}\Psi_{jh}\Big]dx.
\end{align}
First, by (\ref{bubble}), (\ref{gisewr1}), (\ref{gisewr2}) and Lemma \ref{uwal}, we get
\begin{align*}
& \int_\Omega\Big[pV_i^{p-1}\Phi_{il}(P\Phi_{jh}-\Phi_{jh})
+qU_i^{q-1}\Psi_{il}(P\Psi_{jh}-\Psi_{jh})\Big]dx\nonumber\\
\leq &  p |V_i^{p-1}|_{L^{\frac{p+1}{p-1}}(\Omega)}
 |\Phi_{il}|_{L^{p+1}(\Omega)}
 |P\Phi_{jh}-\Phi_{jh}|_{L^{p+1}(\Omega)}
  +
 q |U_i^{q-1}|_{L^{\frac{q+1}{q-1}}(\Omega)}
 |\Psi_{il}|_{L^{q+1}(\Omega)}
 |P\Psi_{jh}-\Psi_{jh}|_{L^{q+1}(\Omega)}\nonumber\\
\leq &
 C\Big(\mu^{N-\frac{N(p-1)}{p+1}
 +N-2-\frac{N}{p+1}+\frac{N_p}{q+1}}
 +
  \mu^{N-\frac{N(q-1)}{q+1}+N-2-\frac{N}{q+1}
  +\frac{N_p}{q+1}}\Big)
  \leq C.
\end{align*}
Next, we estimate the second term in (\ref{ssia}).
By (\ref{bubble}), (\ref{psi}), (\ref{pssi}) and Lemma \ref{ls}, we obtain
\begin{eqnarray*}
&&
\int_\Omega\Big[pV_i^{p-1}(x)\Phi_{il}(x)\Phi_{jh}(x)
+qU_i^{q-1}(x)\Psi_{il}(x)\Psi_{jh}(x)\Big]dx\\
& \leq & C \mu_i^{-\frac{N(p-1)}{p+1} -\frac{N}{p+1}}
 \mu_j^{ -\frac{N}{p+1}}
 \int_{\mathbb{R}^N} V_{1,0}^{p-1}(y)
 \Phi^l_{1,0}(y)\Phi^h_{1,0}
 (\frac{\mu_iy+\xi_i-\xi_j}{\mu_j})dy\\
\quad&+   & C\mu_i^{-\frac{N(q-1)}{q+1} -\frac{N}{q+1}}
 \mu_j^{ -\frac{N}{q+1}}
 \int_{\mathbb{R}^N}U_{1,0}^{q-1}(y)\Psi^l_{1,0}(y)
 \Psi^h_{1,0}(\frac{\mu_iy+\xi_i-\xi_j}{\mu_j})dy\\
 & \leq &
 \left\{ \arraycolsep=1.5pt
   \begin{array}{lll}
 C_h(1+o(1))\ \   &{\rm if }  \ j=i \ {\rm and}\ l=h, \\[2mm]
 o(1)\ \  & {\rm  else}.
\end{array}
\right.
\end{eqnarray*}
Combining all the estimates, we obtain the results.
\end{proof}

\begin{lemma}\label{Ckls}
We have the following estimates,
\begin{equation}\label{sumbu3}
 |  f_0( PV_i)- f_0(V_i) |_{L^{\frac{p+1}{p}}(\Omega)}
   =O(\mu^{\frac{N}{q+1}}).
\end{equation}
\begin{eqnarray}\label{sumfbu2}
 |  f_0^{'}( PV_i)- f^{'}_0(V_i)|_{L^{\frac{p+1}{p-1}}(\Omega)}
    =
 \left\{ \arraycolsep=1.5pt
   \begin{array}{lll}
 O(\mu^{\frac{N}{q+1}})\ \   &{\rm if }\  N\leq 6, \\[2mm]
 O(\mu^{N-2})\ \  & {\rm if}\   N>6,
\end{array}
\right.
\end{eqnarray}

\begin{eqnarray}\label{sumfdbou1}
  |  f_0( PV_i)- f_0(V_i)
     -  f^{'}_0(V_i) (PV_i-V_i) |_{L^{\frac{p+1}{p}}(\Omega)}
 =
\left\{ \arraycolsep=1.5pt
   \begin{array}{lll}
 O(\mu^{(N-2)+\frac{N}{q+1}})\ \   &{\rm if }\  N\leq 6, \\[2mm]
 O(\mu^{(N-2)p})\ \  & {\rm if}\   N>6.
\end{array}
\right.
\end{eqnarray}
\end{lemma}

\begin{proof}
For any $u>0$, $v\in\R$, there is a fact that
\begin{eqnarray}\label{fuin1}
\Big||u+v|^t-u^t\Big|  \leq
\left\{ \arraycolsep=1.5pt
   \begin{array}{lll}
 c(t)\min\{u^{t-1}|v|,|v|^t\}\ \   &{\rm if }\  0<t<1, \\[2mm]
 c(t)(u^{t-1}|v|+|v|^t)\ \  & {\rm if}\   t\geq1.
\end{array}
\right.
\end{eqnarray}
Since $t=p >1$ for $N\geq 3$,
it follows that
\begin{align*}
& \Big(\int_\Omega |  f_0( PV_i)- f_0(V_i) |^{\frac{p+1}{p}}dx\Big)^{\frac{p}{p+1}} \\
  \leq & C
 \mu_i^{\frac{N}{q+1}}\Big(\int_\Omega |V^p_i(x) H(x,\xi_i)|^{\frac{p+1}{p}}dx\Big)^{\frac{p}{p+1}}
 +o(\mu^{\frac{N}{q+1}})\\
  = & C\mu_i^{\frac{N}{q+1}}H(\xi_i,\xi_i)
    \Big(\int_{\frac{\Omega-\xi_i}{\mu_i}} |V_{1,0}^p(y)|^{\frac{p+1}{p}}dy
    \Big)^{\frac{p}{p+1}}
    +o(\mu^{\frac{N}{q+1}})
  \leq   C\mu_i^{\frac{N}{q+1}}.
\end{align*}
We obtain the estimate (\ref{sumbu3}).
Now, we prove (\ref{sumfbu2}).
Since $t=p-1\geq 1$ for $N\leq 6$,
using (\ref{fuin1}),  (\ref{uaoe}) and (\ref{bubble}), we have
\begin{align*}
 & \Big(\int_\Omega |  f^{'}_0( PV_i)- f^{'}_0(V_i) |^{\frac{p+1}{p-1}}dx\Big)^{\frac{p-1}{p+1}} \\
 \leq & C
 \mu_i^{\frac{N}{q+1}}\Big(\int_\Omega |V^{p-1}_i(x) H(x,\xi_i)|^{\frac{p+1}{p-1}}dx\Big)^{\frac{p-1}{p+1}}
 +o(\mu^{\frac{N}{q+1}})\\
 = & C\mu_i^{\frac{N}{q+1} -\frac{N(p-1)}{p+1}+
\frac{N(p-1)}{p+1}}
H(\xi_i,\xi_i)
    \Big(\int_{\frac{\Omega-\xi_i}{\mu_i}} |V_{1,0}^{p-1}(y)|^{\frac{p+1}{p-1}}dy
    \Big)^{\frac{p-1}{p+1}}
    +o(\mu^{\frac{N}{q+1}})
  \leq   C\mu_i^{\frac{N}{q+1}}.
\end{align*}
When $0<t=p-1<1$ for $N>6$, we get
\begin{align*}
& \bigg|\Big| V_i-\left(\frac{b_{N,p}}{\gamma_N}\right)
 \mu_i^{\frac{N}{q+1}}H(x,\xi_i)
 +o(\mu^{\frac{N}{q+1}})\Big|^{p-1}
    -|V_i^{p-1}\bigg|\\
\leq & C \min\bigg\{V_i^{p-2}\Big| -\left(\frac{b_{N,p}}{\gamma_N}\right)
 \mu_i^{\frac{N}{q+1}}H(x,\xi_i)
 +o(\mu^{\frac{N}{q+1}})\Big|,
 \Big| -\left(\frac{b_{N,p}}{\gamma_N}\right)
 \mu_i^{\frac{N}{q+1}}H(x,\xi_i)
 +o(\mu^{\frac{N}{q+1}})\Big|^{p-1}\bigg\}.
\end{align*}
A computation
leads to
\begin{align*}
&  \bigg(\int_\Omega \bigg|V_i^{p-2}\Big[ -\left(\frac{b_{N,p}}{\gamma_N}\right)
 \mu_i^{\frac{N}{q+1}}H(x,\xi_i)
 +o(\mu^{\frac{N}{q+1}})\Big] \bigg|^{\frac{p+1}{p-1}}dx\bigg)^{\frac{p-1}{p+1}}\\
 \leq & C
 \mu_i^{\frac{N}{q+1}}\Big(\int_\Omega \Big|V^{p-2}_i(x) H(x,\xi_i)\Big|^{\frac{p+1}{p-1}}dx\Big)^{\frac{p-1}{p+1}}
 +o(\mu^{\frac{N}{q+1}})\\
 = & C\mu_i^{\frac{N}{q+1}-\frac{N(p-2)}{p+1}
 +\frac{N(p-1)}{p+1}}
 H(\xi_i,\xi_i)
    \Big(\int_{\frac{\Omega-\xi_i}{\mu_i}} |V_{1,0}^{p-2}(y)|^{\frac{p+1}{p-1}}dy
    \Big)^{\frac{p-1}{p+1}}
    +o(\mu^{\frac{N}{q+1}})
 \leq   C\mu^{N-2}.
\end{align*}
On the other hand,
\begin{align*}
&  \bigg(\int_\Omega \bigg|\Big[ -\left(\frac{b_{N,p}}{\gamma_N}\right)
 \mu_i^{\frac{N}{q+1}}H(x,\xi_i)
 +o(\mu^{\frac{N}{q+1}})\Big]^{p-1} \bigg|^{\frac{p+1}{p-1}}dx\bigg)^{\frac{p-1}{p+1}}\\
 \leq & C
 \mu_i^{\frac{N(p-1)}{q+1}}\Big(\int_\Omega | H(x,\xi_i)|^{p+1}dx\Big)^{\frac{p-1}{p+1}}
 +o(\mu^{\frac{N(p-1)}{q+1}})\\
 = & C\mu_i^{\frac{N(p-1)}{q+1}
 +\frac{N(p-1)}{p+1}}
    \Big(\int_{\frac{\Omega-\xi_i}{\mu_i}} |H(\mu_i y+\xi_i,\xi_i)|^{p+1}dy
    \Big)^{\frac{p-1}{p+1}}
    +o(\mu^{\frac{N}{q+1}})
 \leq  C\mu^{(N-2)(p-1)}.
\end{align*}
By collecting  the previous estimates,
we get the estimate (\ref{sumfbu2}).

Finally, there is a fact that
\begin{eqnarray*}\label{fuin2}
\Big||u+v|^t(u+v)-u^{t+1}-(1+t)u^tv\Big|  \leq
\left\{ \arraycolsep=1.5pt
   \begin{array}{lll}
 C(t)\min\{u^{t-1}v^2,|v|^{t+1}\}\ \   &{\rm if }\  0<t<1, \\[2mm]
 C(t)(u^{t-1}v^2+|v|^{t+1})\ \  & {\rm if}\   t\geq1,
\end{array}
\right.
\end{eqnarray*}
which shows
\begin{eqnarray*}\label{fuin2}
&& \Big| f_0( PV_i)- f_0(V_i)
     -  f^{'}_0(V_i) (PV_i-V_i) \Big|  \\
 & \leq&
\left\{ \arraycolsep=1.5pt
   \begin{array}{lll}
 C\min\bigg\{V_i^{p-2}\Big[-\left(\frac{b_{N,p}}{\gamma_N}\right)
 \mu_i^{\frac{N}{q+1}}H(x,\xi_i)
 +o(\mu^{\frac{N}{q+1}})\Big]^2,\\
\quad  \Big|-\left(\frac{b_{N,p}}{\gamma_N}\right)
 \mu_i^{\frac{N}{q+1}}H(x,\xi_i)
 +o(\mu^{\frac{N}{q+1}})\Big|^p\bigg\}\ \   &{\rm if }\  0<p-1<1, \\[2mm]
 C \bigg(V_i^{p-2}\Big[-\left(\frac{b_{N,p}}{\gamma_N}\right)
 \mu_i^{\frac{N}{q+1}}H(x,\xi_i)
 +o(\mu^{\frac{N}{q+1}})\Big]^2
 \\
\quad
 +\Big| -\left(\frac{b_{N,p}}{\gamma_N}\right)
 \mu_i^{\frac{N}{q+1}}H(x,\xi_i)
 +o(\mu^{\frac{N}{q+1}}) \Big|^p\bigg)\ \  & {\rm if}\   p-1\geq1.
\end{array}
\right.
\end{eqnarray*}
Since  $p-1\geq 1$ for $N\leq 6$,
we have
\begin{align*}
&\bigg |\int_\Omega  \bigg| V_i^{p-2}\Big[-\left(\frac{b_{N,p}}{\gamma_N}\right)
 \mu_i^{\frac{N}{q+1}}H(x,\xi_i)
+o(\mu^{\frac{N}{q+1}})\Big]^2
+\Big| -\left(\frac{b_{N,p}}{\gamma_N}\right)
 \mu_i^{\frac{N}{q+1}}H(x,\xi_i)
 +o(\mu^{\frac{N}{q+1}}) \Big|^p
 \bigg| ^{\frac{p+1}{p}}dx\bigg|^{\frac{p}{p+1}}\\
\leq  & C \bigg| \int_{\frac{\Omega-\xi_i}{\mu_i}} \bigg|\mu_i^{\frac{2N}{q+1}-\frac{N(p-2)}{p+1}}
V_{1,0}^{p-2}(y)
\Big[H(\mu_i y+\xi_i,\xi_i)\Big ]^2
+\mu_i^{\frac{Np}{q+1}}H(\mu_i y+\xi_i,\xi_i)\Big]
\bigg|^{\frac{p+1}{p}}dy\bigg|^{\frac{p}{p+1}}
\\
\leq  & C \mu^{(N-2)-\frac{Np}{p+1}}.
\end{align*}
On the other hand,  $0<t=p-1<1$ for $N>6$, there  holds
\begin{align*}
  &\bigg |\int_\Omega  \bigg| V_i^{p-2}\Big[-\left(\frac{b_{N,p}}{\gamma_N}\right)
 \mu_i^{\frac{N}{q+1}}H(x,\xi_i)
 +o(\mu^{\frac{N}{q+1}})\Big]^2\bigg| ^{\frac{p+1}{p}}dx\bigg|^{\frac{p}{p+1}}\\
\leq  & C \mu_i^{\frac{2N}{q+1}-\frac{N(p-2)}{p+1}}
 \mu_i^{\frac{N(p-1)}{p+1}}
\bigg|\int_{\frac{\Omega-\xi_i}{\mu_i}}  \Big |
V_{1,0}^{p-2}(y)H(\mu_i y+\xi_i,\xi_i)\Big |^{\frac{p+1}{p}}dy\bigg|^{\frac{p}{p+1}}
\leq    C \mu^{(N-2)+\frac{N}{q+1}},
\end{align*}
and
\begin{align*}
\bigg |\int_\Omega  \bigg| \Big[ -\left(\frac{b_{N,p}}{\gamma_N}\right)
 \mu_i^{\frac{N}{q+1}}H(x,\xi_i)
 +o(\mu^{\frac{N}{q+1}}) \Big]^p
 \bigg| ^{\frac{p+1}{p}}dx\bigg|^{\frac{p}{p+1}}
\leq   C \mu_i^{\frac{Np}{q+1}+\frac{Np}{p+1}}
\leq   C \mu^{(N-2)p}.
\end{align*}
The estimate  (\ref{sumfdbou1}) holds.
\end{proof}

\begin{lemma}\label{pfS1}
There holds
\[
\Big|g_0(\mathbf{P}U_{\mathbf{d},\boldsymbol\xi})
 -  \sum_{i=1}^{k}g_0(U_i)\Big|_{L^{\frac{q+1}{q}}(\Omega)}
 =O(\mu^{\frac{Np}{q+1}} ),
\]
and
\[
\Big|g^{'}_0(\mathbf{P}U_{\mathbf{d},\boldsymbol\xi})
 -  \sum_{i=1}^{k}g^{'}_0(U_i)\Big|_{L^{\frac{q+1}{q-1}}(\Omega)}
 =O(\mu^{\frac{Np}{q+1}} ).
\]
\end{lemma}
\begin{proof}
We prove the first estimate,
the another one is obtained in a similar way. Observe that,
\begin{align}\label{susel}
&\int_\Omega\Big|g_0(\mathbf{P}U_{\mathbf{d},\boldsymbol\xi})
 -  \sum_{i=1}^{k}g_0(U_i)\Big|^{\frac{q+1}{q}}dx \nonumber\\
 \leq &  \int_\Omega\Big|g_0(\mathbf{P}U_{\mathbf{d},\boldsymbol\xi})
 -g_0(\sum_{i=1}^{k}U_i)\Big|^{\frac{q+1}{q}}dx
 +\int_\Omega\Big|g_0( \sum_{i=1}^{k}U_i)
 -  \sum_{i=1}^{k}g_0(U_i)\Big|^{\frac{q+1}{q}}dx.
\end{align}
First, it follows from Lemma \ref{puo},
\begin{align}\label{usel}
 &  \int_\Omega\Big|g_0(\mathbf{P}U_{\mathbf{d},\boldsymbol\xi})
 -g_0( \sum_{i=1}^{k}U_i)\Big|^{\frac{q+1}{q}}dx\nonumber\\
 \leq & \int_\Omega\Big| (\mathbf{P}U_{\mathbf{d},\boldsymbol\xi})^q
 -  (\sum_{i=1}^{k}U_i)^q\Big|^{\frac{q+1}{q}}dx\nonumber\\
 \leq & \int_\Omega\bigg| \bigg(\sum_{i=1}^kU_i-\mu^{\frac{Np}{q+1}}
\Big(\frac{b_{N,p}}{\gamma_N}\Big)^p\widetilde{H}_{\mathbf{d},\boldsymbol{\xi}}(x)
+o(\mu^{\frac{Np}{q+1}})\bigg)^q
 -  (\sum_{i=1}^{k}U_i)^q\bigg|^{\frac{q+1}{q}}dx
  =   O(\mu^{\frac{Np}{q+1}\frac{q+1}{q}} ).
\end{align}
By Taylor  expansion, (\ref{ls}) and (\ref{bubble}), we get
\begin{align}\label{ajow}
& \int_\Omega\Big|g_0( \sum_{i=1}^{k}U_i)
 -   \sum_{i=1}^{k}g_0(U_i)\Big|^{\frac{q+1}{q}}dx \nonumber \\
 =  &\int_\Omega\Big| ( \sum_{i=1}^{k}U_i)^q
 -  \sum_{i=1}^{k} U_i^q\Big|^{\frac{q+1}{q}}dx \nonumber\\
 \leq & \sum_{j\neq i}^{k}\int_{B_i}\Big|U_i^{q-1}U_j\Big|^{\frac{q+1}{q}}dx
 +O(\mu^{\frac{Np}{q+1}\frac{q+1}{q}} )\nonumber\\
  \leq & \sum_{j\neq i}^{k}\int_{B_i}\bigg|\mu_i^{-\frac{N(q-1)}{q+1}}
  U^{q-1}_{1,0}(\frac{x-\xi_i}{\mu_i})
  \mu_j^{-\frac{N}{q+1}}U_{1,0}(\frac{x-\xi_j}{\mu_j})\bigg|^{\frac{q+1}{q}}dx
 +O(\mu^{\frac{Np}{q+1}\frac{q+1}{q}} )\nonumber\\
 \leq & \sum_{j\neq i}^{k}\mu_i^{N-\frac{N(q-1)}{q+1}\frac{q+1}{q}}\mu_j^{-\frac{N}{q+1}\frac{q+1}{q}}
 \int_{\frac{B_i-\xi_i}{\mu_i}}\bigg|
  U^{q-1}_{1,0}(y)
  U_{1,0}(\frac{\mu_iy+\xi_i-\xi_j}{\mu_j})\bigg|^{\frac{q+1}{q}}dy
 +O(\mu^{\frac{Np}{q+1}\frac{q+1}{q}} )\nonumber\\
  = & O(\mu^{\frac{Np}{q+1}\frac{q+1}{q}} ).
 \end{align}
Finally, inserting (\ref{usel}) and (\ref{ajow}) into (\ref{susel}),
 we obtain the second result and finish the proof of this lemma.
\end{proof}

\begin{lemma}\cite{mp}\label{zsj}
Let $\theta>0$ and $s>1$,
if $\epsilon>0$ small enough, for any  $u, v\in\mathbb{R}$,
  it holds that
  
$ (i)$
$ | f_\epsilon( u)-f_0(u)|
   \leq \epsilon |u|^s\ln\ln(e+|u |)$.
   
$ (ii)$
$|f^{'}_\epsilon(u)|\leq C|u |^{s-1}$.

$ (iii)$
$|f^{'}_\epsilon(u)-f^{'}_0(u) |
\leq \epsilon |u  |^{s-1}
\Big(s\ln\ln(e+|u | )+\frac{1}{\ln (e+ |u |)} \Big)$.

$ (iv)$
\begin{eqnarray*}
\Big|f^{'}_\epsilon(u+v)-f^{'}_\epsilon(u)\Big|\leq
\left\{ \arraycolsep=1.5pt
   \begin{array}{lll}
 C (|u  |^{s-2}+|v|^{s-2} )|v|\ \   &{\rm if }\  N\leq6, \\[2mm]
 C (|v|^{s-1}+\epsilon |u|^{s-1})\ \  & {\rm if}\  N>6,
\end{array}
\right.
\end{eqnarray*}

$ (v)$
$\ln\ln  (e+\mu^{-\theta} u )=\ln\ln( \mu^{-\theta} )
+\ln\Big(1+\frac{ \ln(e^{1-\theta|\ln\mu|}+u)}
{\theta|\ln \mu|}\Big)$.

$ (vi)$
$\lim\limits_{ \mu\rightarrow0}\bigg(|\ln\mu|\ln\Big(1+\frac{ \ln(e^{1-\theta|\ln \mu|}+u)}
{\theta|\ln\mu|}\Big)\bigg)=\frac{1}{\theta}\ln u$,
where $C$ is a   positive constant.
\end{lemma}

The next lemma concerns the relation of the non-power nonlinearity and power type.
\begin{lemma}\label{yy}
It holds true that
\begin{equation}\label{fepli2}
\bigg|f^{'}_\epsilon\Big(\sum_{i=1}^{k}PV_{i}\Big) -f^{'}_0\Big(\sum_{i=1}^{k}PV_{i}\Big) \bigg|_{L^{\frac{p+1}{p-1}}(\Omega)}
=O \Big(\epsilon\mu^{\frac{Np}{p+1}}
   (\ln|\ln\mu|)\Big),
\end{equation}
and
\begin{equation}\label{pfs1}
\Big|g^{'}_\epsilon(\mathbf{P}U_{\mathbf{d},\boldsymbol\xi})
 -  \sum_{i=1}^{k}g^{'}_0(U_i) \Big|_{L^{\frac{q+1}{q-1}}(\Omega)}
 = O \Big(\epsilon \mu^{\frac{Np}{q+1}}
   (\ln|\ln\mu|)\Big).
\end{equation}
\end{lemma}

\begin{proof}
We have
\begin{align*}
&\Big|g^{'}_\epsilon(\mathbf{P}U_{\mathbf{d},\boldsymbol\xi})
 -  \sum_{i=1}^{k}g^{'}_0(U_i) \Big|_{L^{\frac{q+1}{q-1}}(\Omega)}\\
 = & \Big|g^{'}_\epsilon(\mathbf{P}U_{\mathbf{d},\boldsymbol\xi})
 -  g^{'}_0(\mathbf{P}U_{\mathbf{d},\boldsymbol\xi}) \Big|_{L^{\frac{q+1}{q-1}}(\Omega)}
 +\Big|g^{'}_0(\mathbf{P}U_{\mathbf{d},\boldsymbol\xi})
 -  \sum_{i=1}^{k}g^{'}_0(U_i)\Big|_{L^{\frac{q+1}{q-1}}(\Omega)}
  = I+II.
\end{align*}
The proof of (\ref{fepli2}) and $I$ are  essentially same as  (\ref{onqan}), it holds
\[
I=\Big|g^{'}_\epsilon(\mathbf{P}U_{\mathbf{d},\boldsymbol\xi})
 -  g^{'}_0(\mathbf{P}U_{\mathbf{d},\boldsymbol\xi}) \Big|_{L^{\frac{q+1}{q-1}}(\Omega)}
 =O \Big(\epsilon
   (\ln|\ln\mu|)\Big).
\]
Using   the result of Lemma \ref{pfS1},  we obtain
\begin{align*}
  II= \Big|g^{'}_0(\mathbf{P}U_{\mathbf{d},\boldsymbol\xi})
 -  \sum_{i=1}^{k}g^{'}_0(U_i)\Big|_{L^{\frac{q+1}{q-1}}(\Omega)}
  =  O(\mu^{\frac{Np}{q+1}} ).
\end{align*}
Combining $I$ and $II$,  the lemma follows immediately.
\end{proof}

\begin{lemma}\label{ysy}
There holds
\begin{equation*}\label{su8mbu2}
\bigg|f_0\Big(\sum_{i=1}^{k}PV_{i}\Big)
 -\sum_{i=1}^{k}f_0(PV_{i})\Big) \bigg|_{L^{\frac{p+1}{p}}(\Omega)}
=O (\mu^{(N-2)p}).
\end{equation*}
and
\begin{equation*}\label{sumbu2}
\bigg|f^{'}_0\Big(\sum_{i=1}^{k}PV_{i}\Big)
 -\sum_{i=1}^{k}f^{'}_0(PV_{i})\Big) \bigg|_{L^{\frac{p+1}{p-1}}(\Omega)}
=O (\mu^{(N-2)p}).
\end{equation*}
\end{lemma}

\begin{proof}
We will estimate equality, and the first one is similar, we omit the details here.
Observe that
\begin{align*}
 & \int_\Omega\Big|  f^{'}_0\Big(\sum_{i=1}^{k}PV_{i}\Big)-\sum\limits_{i=1}^k f^{'}_0(PV_i) \Big|^{\frac{p+1}{p-1}}dx\\
  = & \int_{\Omega\setminus \cup_{i=1 }^k B_i}
  \Big|\Big(\sum_{i=1}^{k}PV_{i}\Big)^{p-1}
  - \sum\limits_{i=1}^k(PV_i)^{p-1} \Big|^{\frac{p+1}{p-1}}dx
  +\sum\limits_{i=1}^k\int_{B_i}\Big|\Big(\sum_{i=1}^{k}PV_{i}\Big)^{p-1}
  - \sum\limits_{i=1}^k(PV_i)^{p-1} \Big|^{\frac{p+1}{p-1}}dx.
\end{align*}
The first term follows directly from the facts (\ref{bubble}) that
\begin{align*}
    \int_{\Omega\setminus \cup_{i=1 }^k B_i}
    \Big|\Big(\sum_{i=1}^{k}PV_{i}\Big)^{p-1}
  - \sum\limits_{i=1}^k(PV_i)^{p-1} \Big|^{\frac{p+1}{p-1}}dx
 \leq  \sum\limits_{i=1}^k\int_{\Omega\setminus \cup_{i=1 }^k B_i}
 V_i^{(p-1)\frac{p+1}{p-1}}dx
 =o(\mu_j^{-N}).
\end{align*}
For any $i$, by the mean value theorem,  there exists
$t = t(x)\in[0, 1] $ such that
\begin{align*}
 & \int_{B_i}\Big|\Big(\sum_{i=1}^{k}PV_{i}\Big)^{p-1}
  - \sum\limits_{i=1}^k(PV_i)^{p-1} \Big|^{\frac{p+1}{p-1}}dx\\
  = & \int_{B_i}\Big|\Big( PV_i+\sum\limits_{j\neq i}^k
  PV_j\Big)^{p-1}
  - (PV_i)^{p-1}-\sum\limits_{j\neq i}^k(PV_j)^{p-1} \Big|^{\frac{p+1}{p-1}}dx  \\
 \leq &  C\int_{B_i}\Big|\Big( PV_i+t\sum\limits_{j\neq i}^k PV_j\Big)^{p-2}\sum\limits_{j\neq i}^k PV_j\Big|^{\frac{p+1}{p-1}}dx
 +C\sum\limits_{j\neq i}^k\int_{B_i} |PV_j|^{(p-1)\frac{p+1}{p-1} } dx  \\
 \leq &  C\int_{B_i}\Big|( PV_i)^{p-2}\sum\limits_{j\neq i}^k PV_j\Big|^{\frac{p+1}{p-1}}dx
 +C\sum\limits_{j\neq i}^k\int_{B_i} |PV_j|^{p+1} dx  \\
 \leq &  C\sum\limits_{j\neq i}^k\int_{B_i}\Big|V_i^{p-2} V_j\Big|^{\frac{p+1}{p-1}}dx
 +C\sum\limits_{j\neq i}^k\int_{B_i} |V_j|^{p+1} dx.
\end{align*}
Since $j\neq i$, using the notation introduced in (\ref{bubble}),  then
\begin{align*}
 \int_{B_i} |V_j|^{p+1} dx
  \leq   C \int_{B_i}
\bigg( \mu_j^{-\frac{N}{p+1}}V_{1,0} (\frac{x-\xi_j}{\mu_j} )\bigg)^{p+1}dx
 =o(\mu_j^{-N}).
\end{align*}
In view of  Lemma \ref{ls} and (\ref{bubble}),   let $x-\xi_i=\mu_i y$, then
\begin{align*}
   \int_{B_i}\Big|V_i^{p-2} V_j\Big|^{\frac{p+1}{p-1}}dx
   \leq & C \int_{B_i}
\bigg| \bigg(\mu_i^{-\frac{N}{p+1}}V_{1,0} (\frac{x-\xi_i}{\mu_i} )\bigg)^{p-2}
\mu_j^{-\frac{N}{p+1}}V_{1,0} (\frac{x-\xi_j}{\mu_j} )\bigg|^{\frac{p+1}{p-1}}
dx  \\
    =  &  C\mu_i^{N-\frac{(p-2)N}{p+1}\frac{p+1}{p-1}} \mu_j^{\frac{-N}{p+1}\frac{p+1}{p-1}}
    \int_{\frac{B_i-\xi_i}{\mu_i}}\Big|
 V^{p-2}_{1,0}(y)
 V_{1,0} (\frac{\mu_i y+\xi_i-\xi_j}{\mu_j} )\Big|^{\frac{p+1}{p-1}}dy\\
  \leq & C\mu_i^{N-\frac{(p-2)N}{p-1}}
   \mu_j^{-\frac{N}{p-1}}\mu^{(N-2)(p-2)}
    \leq   C\mu^{(N-2)p-1}.
\end{align*}
Thus,  Lemma \ref{sumbu2} holds.
\end{proof}

\begin{lemma}\label{pf1}
For  $j=1,\cdots,k$, there holds
\begin{eqnarray*}
Q_3 &  = &
- \int_\Omega \Big[ g_\epsilon(\mathbf{P}U_{\mathbf{d},\boldsymbol\xi})
 -g_0(\mathbf{P}U_{\mathbf{d},\boldsymbol\xi}) \Big]\Psi_{jh}dx \nonumber\\
   &  = &
\left\{ \arraycolsep=1.5pt
   \begin{array}{lll}
  -\frac{q+1}{N} \mathcal{A}_1
  \sum\limits_{i=1}^k \frac\epsilon{|\ln\mu_i|}
  + O\Big(\sum\limits_{i=1}^k \frac\epsilon{|\ln\mu_i|}
  +\epsilon\mu^{\frac{Np}{q+1}} \ln|\ln\mu|\Big)\ \   \ & {\rm if}\  h=0, \\[2mm]
  O\Big(\epsilon\mu^{\frac{Np}{q+1}} \ln|\ln\mu|\Big)\ \  \  & {\rm if}\  h=1,\cdots,N, \\[2mm]
\end{array}
\right.
\end{eqnarray*}
where $\mathcal{A}_1$  is given  in  Proposition \ref{leftside}.
\end{lemma}

\begin{proof}
By Taylor  expansion with respect to $\epsilon$, we have
\begin{align}\label{mrfa}
Q_3 = &
- \int_\Omega \Big[ g_\epsilon(\mathbf{P}U_{\mathbf{d},\boldsymbol\xi})
 -g_0(\mathbf{P}U_{\mathbf{d},\boldsymbol\xi}) \Big]\Psi_{jh}dx \nonumber\\
= & \int_\Omega \Big[ g_0(\mathbf{P}U_{\mathbf{d},\boldsymbol\xi})
 -g_\epsilon(\mathbf{P}U_{\mathbf{d},\boldsymbol\xi}) \Big]\Psi_{jh}dx \nonumber\\
 = & \epsilon\int_\Omega  (\mathbf{P}U_{\mathbf{d},\boldsymbol\xi})^q\ln\ln(e+\mathbf{P}U_{\mathbf{d},\boldsymbol\xi})\Psi_{jh}dx
  -\epsilon^2\int_\Omega (\mathbf{P}U_{\mathbf{d},\boldsymbol\xi})^q
   \Big(\ln\ln(e+\mathbf{P}U_{\mathbf{d},\boldsymbol\xi})\Big)^2
  \Psi_{jh}dx.
\end{align}
We split the second   integral and  from Lemma \ref{zsj},
\begin{align}\label{ma1}
   \int_\Omega \bigg| (\mathbf{P}U_{\mathbf{d},\boldsymbol\xi})^q
   \Big(\ln\ln(e+\mathbf{P}U_{\mathbf{d},\boldsymbol\xi})\Big)^2
  \Psi_{jh}\bigg|dx
  \leq & \int_\Omega \Big|\Big(\sum\limits_{i=1}^k U_i\Big)^q\Big[\ln\ln\Big(e+\sum\limits_{i=1}^k U_i\Big)\Big]^2
  \Psi_{jh}\Big|dx\nonumber\\
 = & \sum\limits_{i=1}^k\int_{B_i} \bigg|\Big(\sum\limits_{i=1}^k U_i\Big)^q
 \Big[\ln\ln\Big(e+\sum\limits_{i=1}^k U_i\Big)\Big]^2
  \Psi_{jh}\bigg|dx
  + o(\mu^{\frac{Np}{q+1}}).
\end{align}
We now estimate the integral over   $B_i$,
it is convenient to use the change of variables $\mu_iy=x-\xi_i$,
for $h=0$, by (\ref{pssi}), (\ref{psi}), (\ref{bubble}) and Lemma \ref{ls},  then
\begin{align*}
   &  \int_{B_i}\bigg|\Big(  U_i+\sum\limits_{j\neq i}^k U_j\Big)^q
  \Big[\ln\ln\Big(e+  U_i+\sum\limits_{j\neq i}^k U_j\Big)\Big]^2
  \Psi^0_j\bigg|dx\nonumber\\
   = &   \mu_j^{-\frac{N}{q+1}}\int_{\frac{B_i-\xi_i}{\mu_i}}
  \bigg|\mu_i^{-\frac{N}{q+1}}U_{1,0}(y)
  + \sum\limits_{j\neq i}^k\mu_j^{-\frac{N}{q+1}}
  U_{1,0}(\frac{\mu_iy+\xi_i-\xi_j}{\mu_j})
   \bigg|^q\nonumber\\
  & \times  \bigg|\ln\ln\Big[e+\mu_i^{-\frac{N}{q+1}}U_{1,0}(y)
  + \sum\limits_{j\neq i}^k\mu_j^{-\frac{N}{q+1}}
  U_{1,0}(\frac{\mu_iy+\xi_i-\xi_j}{\mu_j})\Big]\bigg|^2 \bigg|\Psi_{1,0}^0(\frac{\mu_iy+\xi_i-\xi_j}{\mu_j})
  \bigg|dy\nonumber\\
 = &  \mu_i^{-\frac{Nq}{q+1}+N}\mu_j^{-\frac{N}{q+1}}
 \int_{\frac{B_i-\xi_i}{\mu_i}} U_{1,0}^q(y)
 \Big[\ln\ln\Big(e+\mu_i^{-\frac{N}{q+1}}U_{1,0}(y)
  + o(1)\Big)\Big]^2 \nonumber\\
 & \times
  \Big[\frac{\mu_iy+\xi_i-\xi_j}{\mu_j}
  \cdot\nabla U_{1,0}(\frac{\mu_iy+\xi_i-\xi_j}{\mu_j})
   +\frac{N }{q+1} U_{1,0}(\frac{\mu_iy+\xi_i-\xi_j}{\mu_j})\Big]dy \nonumber\\
 =    &  O\Big(( \ln|\ln\mu|)^2\Big).
\end{align*}
Using  the similar calculations as above,  for $h=1,\cdots,N$, we have
\begin{align*}
    \int_{B_i}\bigg|\Big(  U_i+\sum\limits_{j\neq i}^k U_j\Big)^q
  \Big[\ln\ln\Big(e+  U_i+\sum\limits_{j\neq i}^k U_j\Big)\Big]^2
  \Psi_{jh}\bigg|dx
   = O\Big(( \ln|\ln\mu|)^2\Big).
\end{align*}
Thus the second term in (\ref{mrfa}) becomes
\begin{align*}
\int_\Omega \Big|(\mathbf{P}U_{\mathbf{d},\boldsymbol\xi})^q
   \Big(\ln\ln(e+\mathbf{P}U_{\mathbf{d},\boldsymbol\xi})\Big)^2
  \Psi_{jh}\Big|dx
 = O\Big( (\ln|\ln\mu| )^2\Big).
\end{align*}
It follows that
\begin{align}\label{ma}
 Q_3 = &
- \int_\Omega \Big[ g_\epsilon(\mathbf{P}U_{\mathbf{d},\boldsymbol\xi})
 -g_0(\mathbf{P}U_{\mathbf{d},\boldsymbol\xi}) \Big]\Psi_{jh}dx \nonumber\\
   = & \epsilon\int_\Omega (\mathbf{P}U_{\mathbf{d},\boldsymbol\xi})^q
   \Big(\ln\ln(e+\mathbf{P}U_{\mathbf{d},\boldsymbol\xi})\Big)\Psi_{jh}dx
   +O\Big( (\epsilon\ln|\ln\epsilon| )^2\Big).
\end{align}
So it suffices to compute
\begin{align}\label{iam}
  &   \int_\Omega (\mathbf{P}U_{\mathbf{d},\boldsymbol\xi})^q\Big(\ln\ln(e+\mathbf{P}U_{\mathbf{d},\boldsymbol\xi})\Big)
      \Psi_{jh}dx\nonumber\\
  = & \int_\Omega \Big(\sum\limits_{i=1}^kU_i\Big)^q
  \Big[\ln\ln\Big(e+\sum\limits_{i=1}^k  U_i \Big)\Big]\Psi_{jh}dx\nonumber\\
    &   -\bigg[
    \int_\Omega  \Big(\sum\limits_{i=1}^k U_i\Big)^q  \Big[\ln\ln\Big(e+ \sum\limits_{i=1}^k U_i\Big) \Big]
    -  (\mathbf{P}U_{\mathbf{d},\boldsymbol\xi})^q  \Big(\ln\ln(e+ \mathbf{P}U_{\mathbf{d},\boldsymbol\xi})\Big)\bigg]
         \Psi_{jh}dx.
\end{align}
Let us set $h(u)=u^q\ln\ln(e+u)$,  by the mean value theorem,  one has
\[
0\leq h(u)-h(v)\leq Cu^{q-1}\Big(\ln\ln(e+u)+1\Big)(u-v)\quad\mbox{for}\ 0\leq v\leq u.
\]
Then,  the second term in (\ref{iam})  takes the form
\begin{align*}
  &  \int_\Omega  \Big(\sum\limits_{i=1}^k U_i\Big)^{q-1} \bigg(\ln\ln\Big(e+ \sum\limits_{i=1}^k U_i  \Big)+1\bigg)
       \bigg(\sum\limits_{i=1}^k U_i- \mathbf{P}U_{\mathbf{d},\boldsymbol\xi}\bigg) \Psi_{jh}dx\\
  = & \sum\limits_{i=1}^k \int_{B_i}
     \Big(\sum\limits_{i=1}^k U_i\Big)^{q-1}
     \bigg(\ln\ln\Big(e+ \sum\limits_{i=1}^k U_i  \Big)+1\bigg)
       (\sum\limits_{i=1}^kU_i -  \mathbf{P}U_{\mathbf{d},\boldsymbol\xi}) \Psi_{jh}dx
       + o(1).
\end{align*}
Let us consider now the region  $B_i$, if $h=0$, in the variable  setting $\mu_iy=x-\xi_i$,
by (\ref{psi}),  Lemmas \ref{puo}  and \ref{zsj},
we have
\begin{eqnarray*}
  &&   \int_{B_i}\bigg|
     \Big( U_i(x)+  \sum_{j\neq i}^kU_j(x)\Big)^{q-1}
     \bigg(\ln\ln\Big(e+  U_i(x)+  \sum_{j\neq i}^kU_j(x)\Big)+1\bigg)\\
     &\times&
     \Big(\sum\limits_{i=1}^kU_i(x) -  \mathbf{P}U_{\mathbf{d},\boldsymbol\xi}(x)\Big) \Psi^0_j (x)\bigg|dx \\
 & = & \mu^{\frac{Np}{q+1}}
\Big(\frac{b_{N,p}}{\gamma_N}\Big)^{p-1}
   \int_{B_i}\bigg|    U_i^{q-1}(x)
     \Big[\ln\ln \Big(e+  U_i(x) \Big)+1\Big]
      \widetilde{H}_{\mathbf{d},\boldsymbol{\xi}}(x)
       \Psi^0_j (x)\bigg|dx
        + o(1)\\
 & = &\Big(\frac{b_{N,p}}{\gamma_N}\Big)^{p-1}\mu^{\frac{Np}{q+1}}
\mu_j^{-\frac{N}{q+1}}\mu_i^N
  \int_{\frac{B_i-\xi_i}{\mu_i}}
  \bigg|\mu_i^{-\frac{N(q-1)}{q+1}}U^{q-1}_{1,0}(y)
  \Big(\ln\ln \Big (e+ \mu_i^{-\frac{N}{q+1}}U_{1,0}(y))+1\Big) \\
 & \times& \Big( \widetilde{H}_{\mathbf{d},\boldsymbol{\xi}}(\xi_i+\mu_i y)+o(1)\Big)
 \Psi_{1,0}^0(\frac{\mu_iy+\xi_i-\xi_j}{\mu_j}) \bigg|dy
 +o(1)\\
 & \leq & \left\{ \arraycolsep=1.5pt
   \begin{array}{lll}
     \mu^{\frac{Np}{q+1}}
\Big(\frac{b_{N,p}}{\gamma_N}\Big)^{p-1}
\widetilde{H}_{\mathbf{d},\boldsymbol{\xi}}(\xi_i)
\ln|\ln\mu_i|+o(1)
\ \   {\rm if}\ j=i, \\[2mm]
\Big(\frac{b_{N,p}}{\gamma_N}\Big)^{p-1}\mu^{\frac{Np}{q+1}}
\mu_j^{-\frac{N}{q+1}}\mu_i^{N-\frac{N(q-1)}{q+1}}
\widetilde{H}_{\mathbf{d},\boldsymbol{\xi}}(\xi_i)+o(1)
\ \   {\rm if}\ j\neq i, \\[2mm]
  \end{array}
\right.\\
& = &
  O\Big(\mu^{\frac{Np}{q+1}} \ln|\ln\mu|\Big).
\end{eqnarray*}
Notice that
 by the same argument, for $h=1,\cdots,N$, we have
\begin{align*}
   \int_{B_i}\bigg|
     \Big( U_i+ \sum_{j\neq i}^kU_j \Big)^{q-1}
     \bigg(\ln\ln\Big(e+ U_i+ \sum_{j\neq i}^kU_j\Big )+1\bigg)
     (\sum\limits_{i=1}^kU_i -  \mathbf{P}U_{\mathbf{d},\boldsymbol\xi})\Psi_{jh}\bigg|dx
 =      O\Big(\mu^{\frac{Np}{q+1}} \ln|\ln\mu|\Big).
\end{align*}
Therefore,  we rewrite the second term  in (\ref{iam})  as
\begin{align*}
 &   \int_\Omega \bigg| \Big(\sum\limits_{i=1}^k U_i\Big)^q\bigg(\ln\ln\Big(e+ \sum\limits_{i=1}^k U_i\Big)  \bigg)
          \Psi_{jh}\bigg|dx
 - \int_\Omega  \Big|(\mathbf{P}U_{\mathbf{d},\boldsymbol\xi})^q \ln\ln(e+ \mathbf{P}U_{\mathbf{d},\boldsymbol\xi})
         \Psi_{jh}\Big| dx\\
   = &  O\Big(\mu^{\frac{Np}{q+1}} \ln|\ln\mu|\Big).
\end{align*}

Now, we split the first term in (\ref{iam})  into two parts, and
   if $j=i$, from Lemma \ref{zsj}, (\ref{psi}), (\ref{pssi}),
let  $x-\xi_i=\mu_iy$,
then
\begin{align}\label{maa}
  &   \int_\Omega \bigg| \Big(\sum\limits_{i=1}^kU_i(x) \Big)^q\Big[\ln\ln \Big(e+\sum\limits_{i=1}^k
        U_i(x) \Big)\Big]\Psi^h_i (x) \bigg|dx\nonumber\\
   = & \sum\limits_{i=1}^k \int_{B_i}
    \bigg|\Big(U_i(x)+ \sum\limits_{j\neq i}^k U_j(x) \Big)^q
    \Big[\ln\ln\Big(e+   U_i(x)+\sum\limits_{j\neq i}^k U_j(x)  \Big)\Big]\Psi^h_i(x) \bigg|dx
   +  o(1)\nonumber\\
  = & \sum\limits_{i=1}^k
  \mu_i^{N-\frac{N}{q+1}}\mu_i^{-\frac{Nq}{q+1}}
  \int_{\frac{B_i-\xi_i}{\mu_i}}U^q_{1,0}(y)\nonumber\\
 &\times\bigg|\ln\ln\bigg(e+ \mu_i^{-\frac{N}{q+1}}U_{1,0}(y)
  + \sum\limits_{j\neq i}^k\mu_j^{-\frac{N}{q+1}}
  U_{1,0}(\frac{\mu_iy+\xi_i-\xi_j}{\mu_j})
    \bigg)
  \Psi^h_{1,0}(y)\bigg|dy
  +  o(1)\nonumber\\
  = &\sum\limits_{i=1}^k \ln\Big|\ln\mu_i^{ -\frac{N}{q+1}}\Big|
  \int_{\frac{B_i-\xi_i}{\mu_i}} U^q_{1,0}(y)\Psi^h_{1,0}(y)dy\nonumber\\
 &  + \sum\limits_{i=1}^k\frac{1}{|\ln \mu_i|} \int_{\frac{B_i-\xi_i}{\mu_i}}
  U^q_{1,0}(y)
  \bigg[|\ln  \mu_i|\ln
     \bigg(1+\frac{ \ln \Big[e^{1-\frac{N}{q+1}|\ln \mu_i|}+U_{1,0}(y)\Big]}{\frac{N}{q+1}
     |\ln \mu_i|}\bigg)\bigg]\Psi^h_{1,0}(y)dy
     + o(1).
\end{align}
Moreover, we set
$\Lambda(y)=U^q_{1,0}(y) |\ln\mu_i|\ln
     \bigg(1+\frac{ \ln \Big(e^{1-\frac{N}{q+1}|\ln  \mu_i|}+U_{1,0}(y)\Big)}{\frac{N}{q+1}
     |\ln  \mu_i|}\bigg)\Psi_{1,0}^h(y)$ for
$h=1,\cdots,N$.
Since $\Psi_{1,0}^h(y)$ is a odd function,
we deduce  $\int_{\mathbb{R}^N}\Lambda(y)dy=0$.
Further,  Lemma \ref{zsj}  yields that
\begin{align*}
 &  \int_{\mathbb{R}^N}\Lambda(y)dy-\sum\limits_{i=1}^k\int_{\frac{B_i-\xi_i}{\mu_i}}\Lambda(y)dy
 =\int_{\mathbb{R}^N \backslash {\{\cup_{i=1 }^k \frac{B_i-\xi_i}{\mu_i}}\}}\Lambda(y)dy\\
  \leq & C  |\ln \mu_i |\ln
     \bigg(1+\frac{ \ln \Big(e^{1-\frac{N}{q+1}|\ln \mu_i|}+ U_{1,0}(y) \Big)}{\frac{N}{q+1}|\mu_i|}\bigg)
     \int_{\mathbb{R}^N \backslash {\{\cup_{i=1 }^k \frac{B_i-\xi_i}{\mu_i}}\}}
     U^q_{1,0}(y)  |\Psi_{1,0}^h(y)|dy\\
   \leq & C\Big(\frac{q+1}{N}\ln U_{1,0}+o(1)\Big)\mu_i^{[(N-2)p-2]q}
   =O(\mu^{[(N-2)p-2]q}).
\end{align*}
Hence, for $ \mu$ small enough, we conclude that
\begin{align}\label{ma2}
  &  \int_\Omega \Big (\sum\limits_{i=1}^kU_i\Big )^q\ln\ln \Big(e+\sum\limits_{i=1}^k
        U_i  \Big)\Psi_{jh}dx
        =O(\mu^{[(N-2)p-2]q}),
        \quad
   \mbox{for}\
h=1,\cdots,N.
\end{align}

When  $h=0$, since $(U_i,V_i) $ is the unique positive solution of problem  (\ref{uni}) and also, $(\Psi_i^0,\Phi_i^0)$ solves (\ref{linear-equ}), then

\begin{align*}
 \int_{\mathbb{R}^N} (V_i^p \Phi_i ^0+U_i ^q \Psi_i ^0 )dx
 = & \int_{\mathbb{R}^N} \Big[(-\Delta U_i)   \Phi_i^0+ (-\Delta V_i) \Psi_i^0\Big] dx\\
 = & \int_{\mathbb{R}^N} \Big[ U_i   (-\Delta \Phi_i^0)+  V_i  (-\Delta\Psi_i^0)\Big] dx\\
 = & \int_{\mathbb{R}^N}\Big[ qU_i^q\Psi_i ^0+pV_i^p\Psi_i^0\Big] dx,
\end{align*}
which follows that
\begin{equation*}
 \int_{\mathbb{R}^N} (V_i^p \Phi_i^0+U_i^q \Psi_i^0 )dx=0.
\end{equation*}
Thus, we get $\int_{\mathbb{R}^N} U^q_{1,0}(y)\Psi_{1,0}^0(y)dx=0$,
also,
$
\int_{\mathbb{R}^N \backslash {\{\cup_{i=1 }^k \frac{B_i-\xi_i}{\mu_i}}\}}
U^q_{1,0}(y)\Psi_{1,0}^0(y)dy=0.
$
From (\ref{maa}) and
Lemma \ref{zsj}, there holds
 \begin{align*}
  &   \int_\Omega \Big(\sum\limits_{i=1}^kU_i\Big)^q\ln\ln\Big(e+\sum\limits_{i=1}^k
        U_i \Big)\Psi^0_i dx\nonumber\\
  = & \frac{q+1}{N}
  \sum\limits_{i=1}^k \frac{1}{|\ln\mu_i|}
   \int_{\mathbb{R}^N}
     U^q_{1,0}(y)
     \ln \Big(U_{1,0}(y)\Big)\Psi_{1,0}^0(y)dy +O(\mu^{[(N-2)p-2]q})\nonumber\\
   = &  -  \frac{q+1}{N}
  \sum\limits_{i=1}^k \frac{1}{|\ln\mu_i|}\mathcal{A}_1
  +O(\mu^{[(N-2)p-2]q}),
 \end{align*}
 where the estimate of $\mathcal{A}_1$  is   postponed until the end of this lemma.
Arguing similarly,  if $j\neq i$, one has
\begin{equation*}
\int_\Omega\bigg| \Big (\sum\limits_{i=1}^kU_i \Big)^q\ln\ln \Big(e+\sum\limits_{i=1}^k
        U_i  \Big)\Psi_{jh}\bigg|dx
        =
 O(\mu^{[(N-2)p-2]q}).
\end{equation*}
Consequently,
\begin{eqnarray*}
  Q_3 &  = &
- \int_\Omega \Big[ g_\epsilon(\mathbf{P}U_{\mathbf{d},\boldsymbol\xi})
 -g_0(\mathbf{P}U_{\mathbf{d},\boldsymbol\xi}) \Big]\Psi_{jh}dx \nonumber\\
 &  = &
\left\{ \arraycolsep=1.5pt
   \begin{array}{lll}
  -\frac{q+1}{N} \mathcal{A}_1
  \sum\limits_{i=1}^k \frac\epsilon{|\ln\mu_i|}
  + O\Big(\sum\limits_{i=1}^k \frac\epsilon{|\ln\mu_i|}\Big)\ \   \ & {\rm if}\  h=0, \\[2mm]
  O\Big(\epsilon\mu^{\frac{Np}{q+1}} \ln|\ln\mu|\Big)\ \  \  & {\rm if}\  h=1,\cdots,N. \\[2mm]
\end{array}
\right.
\end{eqnarray*}

Finally, let us state that $\mathcal{A}_1$ is a positive constant.
Let $B(0,\tilde{r})$ be a ball with a fixed  $\tilde{r}>0$,
there holds
\begin{align*}
\mathcal{A}_1 = & - \int_{\mathbb{R}^N}
      U^q_{1,0}(y)
     \ln \Big(U_{1,0}(y)\Big)\Psi_{1,0}^0(y)dy\\
  = &- \int_{B(0,\tilde{r})}
      U^q_{1,0}(y)
     \ln \Big(U_{1,0}(y)\Big)\Psi_{1,0}^0(y)dy
     - \int_{\mathbb{R}^N\setminus B(0,\tilde{r})}
      U^q_{1,0}(y)
     \ln \Big(U_{1,0}(y)\Big)\Psi_{1,0}^0(y)dy,
\end{align*}
In a ball $B(0,\tilde{r})$,
the result of (\ref{max}) means that
$ 0<U_{1,0}(x)\leq 1$,
 then, by the continuity of logarithm function, we know that $\ln( U_{1,0}(x))$ is a negative function.
Thus, there is a constant $C>0$ such that
\begin{align*}
  \int_{B(0,\tilde{r})}
      U^q_{1,0}(y)
     \ln \Big(U_{1,0}(y)\Big)\Psi_{1,0}^0(y)dy
     \leq C
\end{align*}
In $\mathbb{R}^N\setminus B(0,\tilde{r})$,
by (\ref{pssi}) and Lemma \ref{ls},  it holds
\begin{align*}
   \int_{\mathbb{R}^N\setminus B(0,\tilde{r})}
      U^q_{1,0}(y)
     \ln \Big(U_{1,0}(y)\Big)\Psi_{1,0}^0(y)dy
     \rightarrow -\infty.
\end{align*}
Consequently $\mathcal{A}_1>0$.
\end{proof}

\begin{lemma}\label{pf1R}
For $j=1,\cdots,k$, there holds
\begin{eqnarray*}
P_2&  = &
- \int_\Omega \bigg[f_\epsilon\Big(\sum_{i=1}^{k}PV_{i}\Big)
 -f_0\Big(\sum_{i=1}^{k}PV_{i}\Big)\bigg]P\Phi_{jh}dx \nonumber\\
   &  = &
\left\{ \arraycolsep=1.5pt
   \begin{array}{lll}
  -\frac{p+1}{N} \mathcal{\tilde{A}}_1
  \sum\limits_{i=1}^k \frac\epsilon{|\ln\mu_i|}
  + O\Big(\sum\limits_{i=1}^k \frac\epsilon{|\ln\mu_i|}
  +  \epsilon\mu^{\frac{Nq}{q+1}}
   (\ln|\ln\mu|)\Big) \ \   \ & {\rm if}\  h=0, \\[2mm]
  O \Big(\epsilon\mu^{\frac{Nq}{q+1}}
   (\ln|\ln\mu|)\Big)\ \  \  & {\rm if}\  h=1,\cdots,N. \\[2mm]
\end{array}
\right.
\end{eqnarray*}
where $\mathcal{\tilde{A}}_1$  is given  in  Proposition \ref{leftside}
and  its estimate is similar to $\mathcal{A}_1$.
\end{lemma}

\begin{proof}
Observe that
\begin{align}\label{mrfa}
P_2= & - \int_\Omega \bigg[f_\epsilon\Big(\sum_{i=1}^{k}PV_{i}\Big)
 -f_0\Big(\sum_{i=1}^{k}PV_{i}\Big)\bigg]P\Phi_{jh}dx \nonumber\\
= & - \int_\Omega \bigg[f_\epsilon\Big(\sum_{i=1}^{k}PV_{i}\Big)
 -f_0\Big(\sum_{i=1}^{k}PV_{i}\Big)\bigg]\Phi_{jh}dx \nonumber\\
 & - \int_\Omega \bigg[f_\epsilon\Big(\sum_{i=1}^{k}PV_{i}\Big)
 -f_0\Big(\sum_{i=1}^{k}PV_{i}\Big)\bigg](P\Phi_{jh}-\Phi_{jh}) dx.
\end{align}
The first term  can be estimated as in the proof of Lemma \ref{pf1}, we omit the details here.
Now, we are going to deal with the second term in (\ref{mrfa}),
a direct calculation deduces that
\begin{align*}
& \int_\Omega \bigg[f_\epsilon\Big(\sum_{i=1}^{k}PV_{i}\Big)
 -f_0\Big(\sum_{i=1}^{k}PV_{i}\Big)\bigg](P\Phi_{jh}-\Phi_{jh}) dx\nonumber\\
\leq  &
\Big|f_\epsilon\Big(\sum_{i=1}^{k}PV_{i}\Big)
 -f_0\Big(\sum_{i=1}^{k}PV_{i}\Big)\Big|_{L^{\frac{p+1}{p}}(\Omega)}
 \Big|P\Phi_{jh}-\Phi_{jh}\Big|_{L^{p+1}(\Omega)} .
\end{align*}
Now, we use  Lemma \ref{zsj}  to find that
\begin{align}\label{onqan}
 \int_\Omega \bigg|f_\epsilon
  \Big(\sum_{i=1}^{k}PV_{i}\Big)-f_0\Big(\sum_{i=1}^{k}PV_{i}\Big)
  \bigg|^{\frac{p+1}{p}}dx
  \leq  &  \epsilon \int_\Omega \bigg|\Big(\sum_{i=1}^{k}PV_{i}\Big)^p
  \ln\ln\Big(e+\sum_{i=1}^{k}PV_{i}\Big)\bigg|^{\frac{p+1}{p}}dx \nonumber\\
  \leq  & \epsilon \int_\Omega\bigg|\Big(\sum\limits_{i=1}^k V_i\Big)^p
    \ln\ln \Big(e+ \sum\limits_{i=1}^k V_i \Big)\bigg|^{\frac{p+1}{p}}dx \nonumber\\
  \leq  &
  \epsilon \int_{\Omega\setminus \cup_{i=1 }^k B_i} \bigg| \Big(\sum\limits_{i=1}^k V_i\Big)^p
  \ln\ln \Big(e+\sum\limits_{i=1}^k V_i\Big)\bigg|^{\frac{p+1}{p}}dx \nonumber\\
  & + \epsilon \sum_{i=1}^k\int_{B_i} \bigg| \Big(\sum\limits_{i=1}^k V_i\Big)^p
  \ln\ln \Big(e+\sum\limits_{i=1}^k V_i\Big)\bigg|^{\frac{p+1}{p}}dx.
\end{align}
By (\ref{bubble}), one has
\begin{align}\label{ksdu}
   \int_{\Omega\setminus \cup_{i=1}^k B_i} \bigg| \Big(\sum\limits_{i=1}^k V_i\Big)^p
  \ln\ln \Big(e+\sum\limits_{i=1}^k V_i\Big)\bigg|^{\frac{p+1}{p}}dx
  \leq   &  C \sum\limits_{i=1}^k \int_{\Omega\setminus \cup_{i=1 }^k B_i} \bigg| V_i^p
  \ln\ln \Big(e+\sum\limits_{i=1}^k V_i\Big)\bigg|^{\frac{p+1}{p}}dx\nonumber\\
  \leq &  C\sum\limits_{i=1}^k \mu_i^{N}
    \bigg|
  \ln\ln \Big(e+\sum\limits_{i=1}^k \mu_i^{-\frac{N}{p+1}}\Big)\bigg|^{\frac{p+1}{p}}\nonumber\\
  \leq &  C  \Big(\mu^{\frac{Np}{p+1}}
   (\ln|\ln\mu|) \Big)^{\frac{p+1}{p}}.
\end{align}
To prove     the second integral in (\ref{onqan}),
 we first note that
\begin{align}\label{ksu}
& \int_{B_i} \bigg|\Big (\sum\limits_{i=1}^k V_i\Big)^p
  \ln\ln \Big(e+\sum\limits_{i=1}^k V_i\Big)\bigg|^{\frac{p+1}{p}}dx\nonumber\\
 = & \int_{B_i} \bigg| \Big(V_i+\sum_{j\neq i}^kV_j\Big)^p
  \ln\ln \Big(e+V_i+\sum_{j\neq i}^kV_j\Big)\bigg|^{\frac{p+1}{p}}dx\nonumber\\
 = &  \int_{B_i} \bigg|\Big[V_i^p+ O\Big(\sum_{j\neq i}^kV_j\Big)    \Big]\ln\ln \Big(e+V_i+\sum_{j\neq i}^kV_j\Big)\bigg|^{\frac{p+1}{p}}dx\nonumber\\
 \leq  & C \int_{B_i} \bigg|V_i^p\ln\ln \Big(e+V_i+\sum_{j\neq i}^kV_j\Big)\bigg|^{\frac{p+1}{p}}dx
  + C\sum_{j\neq i}^k\int_{B_i}V_j \bigg|\ln\ln \Big(e+V_i+\sum_{j\neq i}^kV_j\Big)\bigg|^{\frac{p+1}{p}}dx.
\end{align}
We  set $x-\xi_i=\mu_i y$, by (\ref{bubble}) and Lemma \ref{zsj}, we have
\begin{align*}
&    \int_{B_i} \bigg|V_i^p\ln\ln \Big(e+V_i+\sum_{j\neq i}^kV_j\Big)\bigg|^{\frac{p+1}{p}}dx\\
= & \int_{B_i}
\bigg|\mu_i^{-\frac{Np}{p+1}}V^p_{1,0} (\frac{x-\xi_i}{\mu_i} )
  \ln\ln \Big(e+\mu_i^{-\frac{N}{p+1}}V_{1,0} (\frac{x-\xi_i}{\mu_i} )
  +\sum_{j\neq i}^k \mu_j^{-\frac{N}{p+1}}V_{1,0} (\frac{x-\xi_j}{\mu_j}) \Big)\bigg|^{\frac{p+1}{p}}dx\\
 = &
 \int_{\frac{B_i-\xi_i}{\mu_i}}
  \bigg| V^p_{1,0} (y)\ln\ln \Big(e+\mu_i^{-\frac{N}{p+1}}V_{1,0} (y)
  +\sum_{j\neq i}^k \mu_j^{-\frac{N}{p+1}}V_{1,0} (\frac{\mu_iy+\xi_i-\xi_j}{\mu_j}) \Big)\bigg|^{\frac{p+1}{p}}dy\\
 = & \int_{\frac{B_i-\xi_i}{\mu_i}}
  \bigg|  V^p_{1,0} (y)\bigg(\ln|\ln \mu_i^{-\frac{N}{p+1}}|
  +   \ln \Big[1+\frac{\ln\Big(e^{1-\frac{N}{p+1}|\ln \mu_i|  }
  + V_{1,0} (y)
  \Big)}{\frac{N}{p+1}|\ln \mu_i|  }
\Big]\bigg)\bigg|^{\frac{p+1}{p}}dy +o(\epsilon)\\
= & \int_{\frac{B_i-\xi_i}{\mu_i}}
 \bigg|  V^p_{1,0} (y)\bigg(
   \ln|\ln \mu_i^{-\frac{N }{p+1}}| +  \frac{1}{ |\ln \mu_i|}\frac{p+1}{N} \ln ( V_{1,0} (y))
 \bigg)\bigg|^{\frac{p+1}{p}}dy +o(\epsilon)\\
\leq & C (\ln|\ln\mu| )^{\frac{p+1}{p}}.
\end{align*}
Regarding the   second term in (\ref{ksu}),
 which can be handled by the same way for estimating
\begin{align*}
& \sum_{j\neq i}^k\int_{B_i}V_j \bigg|\ln\ln \Big(e+V_i+\sum_{j\neq i}^kV_j\Big)\bigg|^{\frac{p+1}{p}}dx
=O\Big((\ln|\ln\mu| )^{\frac{p+1}{p}}\Big).
\end{align*}
Thus,
\begin{equation}\label{onran}
\Big|f_\epsilon\Big(\sum_{i=1}^{k}PV_{i}\Big)-f_0\Big(\sum_{i=1}^{k}PV_{i}\Big)\Big|_{\frac{p+1}{p}}
=O \Big(\epsilon\mu^{\frac{Np}{p+1}}
   (\ln|\ln\mu|)\Big).
\end{equation}
Putting (\ref{ksdu})-(\ref{onran}) and (\ref{gisewr2}) together, we obtain  \begin{align*}
 \int_\Omega \bigg[f_\epsilon\Big(\sum_{i=1}^{k}PV_{i}\Big)
 -f_0\Big(\sum_{i=1}^{k}PV_{i}\Big)\bigg](P\Phi_{jh}-\Phi_{jh}) dx
= O \Big(\epsilon\mu^{\frac{Np}{p+1}}
   (\ln|\ln\mu|)\Big).
\end{align*}

\end{proof}

\begin{lemma}\label{poi}
For $j=1,\cdots,k$, there holds
\begin{eqnarray*}
Q_5 &  = &
- \int_\Omega \Big[ g_0(\mathbf{P}U_{\mathbf{d},\boldsymbol\xi})
 -  \sum_{i=1}^{k}g_0(U_i) \Big]\Psi_{jh}dx
 \nonumber\\
   &  = &
\left\{ \arraycolsep=1.5pt
   \begin{array}{lll}
  \Big(\frac{b_{N,p}}{\gamma_N}\Big)^p \mathcal{A}_2
\mu^{\frac{N(p+1)}{q+1}}
\sum\limits_{i=1}^kd_i^{\frac{N}{q+1}}
\widetilde{H}_{\mathbf{d},\boldsymbol{\xi}}(\xi_i)  -a_{N,p}\mathcal{A}_4
\mu^{(N-2)p-2}
\sum\limits_{j\neq i}^k\frac{d_{i}^{\frac{2N}{q+1}}
d_{j}^{\frac{N(p-1)}{q+1}}}{|\xi_{i}-\xi_{j}|^{(N-2)p-2}}
 \\
\quad + O(\mu^{(N-2)p-1})
  \ \   \  {\rm if}\  h=0, \\[2mm]
  \Big(\frac{b_{N,p}}{\gamma_N}\Big)^p
  \mathcal{A}_3
  \mu^{\frac{N(p+1)}{q+1}+1}
\sum\limits_{i=1}^kd_i^{\frac{N}{q+1}+1}
\partial_{\xi_{ih}}\tilde{\rho}(\xi_i)
+  O(\mu^{(N-2)p-1})
  \ \  \   {\rm if}\  h=1,\cdots,N, \\[2mm]
\end{array}
\right.
\end{eqnarray*}
where $\mathcal{A}_2$,  $\mathcal{A}_3$  and  $\mathcal{A}_4$ are given  in  Proposition \ref{leftside}.
\end{lemma}

\begin{proof}
By (\ref{puo}),  we have
\begin{align}\label{uaeaa}
  Q_5   = &
- \int_\Omega \Big[ g_0(\mathbf{P}U_{\mathbf{d},\boldsymbol\xi})
 -  \sum_{i=1}^{k}g_0(U_i) \Big]\Psi_{jh}dx
 \nonumber\\
 = & \int_\Omega \Big[\sum_{i=1}^{k}g_0(U_i)- g_0(\mathbf{P}U_{\mathbf{d},\boldsymbol\xi})\Big]\Psi_{jh} dx
 \nonumber\\
 = & \int_\Omega \Big[\sum_{i=1}^{k}U_i^q- (\mathbf{P}U_{\mathbf{d},\boldsymbol\xi})^q\Big]\Psi_{jh} dx
 \nonumber\\
 = & \int_\Omega \bigg(\sum_{i=1}^{k}U_i^q- \Big[\sum_{j=1}^kU_j-\mu^{\frac{Np}{q+1}}
\Big(\frac{b_{N,p}}{\gamma_N}\Big)^p\widetilde{H}_{\mathbf{d},\boldsymbol{\xi}}(x)
+o(\mu^{\frac{Np}{q+1}})\Big]^q
 \bigg)\Psi_{jh}dx
 \nonumber\\
 = & \int_\Omega \bigg[\sum_{i=1}^{k}U_i^q- \Big(\sum_{j=1}^kU_j \Big)^q\bigg]\Psi_{jh}dx\nonumber\\
 & +q\int_\Omega\Big(\sum_{j=1}^kU_j\Big)^{q-1}
 \Big[\mu^{\frac{Np}{q+1}}
\Big(\frac{b_{N,p}}{\gamma_N}\Big)^p\widetilde{H}_{\mathbf{d},\boldsymbol{\xi}}(x)
+o(\mu^{\frac{Np}{q+1}})\Big]\Psi_{jh}dx
+o\Big(\mu^{\frac{Np}{q+1}}
\int_\Omega\Psi_{jh}dx\Big).
\end{align}
By the Taylor  expansion, we have
\begin{align}\label{iak5}
 \int_\Omega \bigg[\sum_{i=1}^{k}U_i^q- \Big(\sum_{j=1}^kU_j \Big)^q\bigg]\Psi_{jh}dx
= & \sum_{i=1}^{k}\int_{B_i}
\bigg[U_i^q
-\Big(U_i+\sum_{j\neq i}^{k}U_j\Big)^q
\bigg]\Psi_{jh}dx
+ O(\mu^{(N-2)p-1})
\nonumber\\
=& - q\sum_{i=1}^{k}\int_{B_i} U_i^{q-1} \sum_{j\neq i}^{k} U_j\Psi_{jh}dx
+ O(\mu^{(N-2)p-1}).
\end{align}
If  $h=0$, one has
\begin{align*}
&  \sum_{j\neq i}^{k}
\int_{\Omega}U_i^{q-1}(x)U_j(x) \Psi_j^0(x)dx\nonumber\\
=& \sum_{j\neq i}^{k}
 \int_{\Omega}\Big(\mu_i^{-\frac{N}{q+1}}
 U_{1,0}(\frac{x-\xi_i}{\mu_i})\Big)^{q-1}
 \mu_j^{-\frac{N}{q+1}}U_{1,0}(\frac{x-\xi_j}{\mu_j}) \mu_j^{-\frac{N}{q+1}}\Psi_{1,0}^0(\frac{x-\xi_j}{\mu_j})dx\\
= &  \sum_{j\neq i}^{k} \mu_i^{N-\frac{N(q-1)}{q+1}}\mu_j^{-\frac{2N}{q+1}}
 \int_{\frac{\Omega-\xi_i}{\mu_i}}
 U^{q-1}_{1,0}(y)
U_{1,0}(\frac{\mu_iy+\xi_i-\xi_j}{\mu_j}) \Psi_{1,0}^0(\frac{\mu_iy+\xi_i-\xi_j}{\mu_j})dy\\
= &  \sum_{j\neq i}^{k} \mu_i^{N-\frac{N(q-1)}{q+1}}\mu_j^{-\frac{2N}{q+1}}
 \int_{\frac{\Omega-\xi_i}{\mu_i}}
 U^{q-1}_{1,0}(y)
\Big[U_{1,0}(\frac{\xi_i-\xi_j}{\mu_j})
+\nabla U_{1,0}(\frac{\xi_i-\xi_j}{\mu_j})\frac{\mu_i}{\mu_j}|y|
+o(1) \Big]\\
& \times \Big[\Psi^0_{1,0}(\frac{\xi_i-\xi_j}{\mu_j})
+o(1) \Big]dy\\
= &  \sum_{j\neq i}^{k} \mu_i^{N-\frac{N(q-1)}{q+1}}\mu_j^{-\frac{2N}{q+1}}
 \int_{\frac{\Omega-\xi_i}{\mu_i}}
 U^{q-1}_{1,0}(y)
U_{1,0}(\frac{\xi_i-\xi_j}{\mu_j})
\Psi^0_{1,0}(\frac{\xi_i-\xi_j}{\mu_j})dy\\
&  + O\bigg(\int_{\frac{\Omega-\xi_i}{\mu_i}}U^{q-1}_{1,0}(y)
\Big(\frac{\xi_i-\xi_j}{\mu_j}\Big)^{-[(N-2)p-2]}|y|
\Psi^0_{1,0}(\frac{\xi_i-\xi_j}{\mu_j})dy
\bigg)\\
= &  \sum_{j\neq i}^{k} \mu_i^{N-\frac{N(q-1)}{q+1}}\mu_j^{-\frac{2N}{q+1}}
a_{N,p}
\Big(\frac{\xi_i-\xi_j}{\mu_j}\Big)^{-[(N-2)p-2]}\\
 & \times\int_{\frac{\Omega-\xi_i}{\mu_i}}
 U^{q-1}_{1,0}(y)
\Big[ \frac{\xi_i-\xi_j}{\mu_j}\nabla U_{1,0}(\frac{\xi_i-\xi_j}{\mu_j})
+\frac{N}{q+1}U_{1,0}(\frac{\xi_i-\xi_j}{\mu_j}) \Big]dy
 + O(\mu^{(N-2)p-2})\\
= &  \sum_{j\neq i}^{k} \mu_i^{N-\frac{N(q-1)}{q+1}}\mu_j^{-\frac{2N}{q+1}}
a_{N,p}
\Big(\frac{\xi_i-\xi_j}{\mu_j}\Big)^{-[(N-2)p-2]}\\
 & \times\int_{\frac{\Omega-\xi_i}{\mu_i}}
 U^{q-1}_{1,0}(y)
\bigg[ \frac{\xi_i-\xi_j}{\mu_j}
a_{N,p}[(N-2)p-2]
\Big(\frac{\xi_i-\xi_j}{\mu_j}\Big)^{-[(N-2)p-1]}\\
& +\frac{N}{q+1}
a_{N,p}\Big(\frac{\xi_i-\xi_j}{\mu_j}\Big)^{-[(N-2)p-2]}
\bigg]dy
 + O(\mu^{(N-2)p-2})\\
= & a_{N,p} \sum_{j\neq i}^{k} \mu_i^{N-\frac{N(q-1)}{q+1}}
\mu_j^{-\frac{2N}{q+1}}
\Big(\frac{\xi_i-\xi_j}{\mu_j}\Big)^{-[(N-2)p-2]}
\int_{\frac{\Omega-\xi_i}{\mu_i}}
 U^{q-1}_{1,0}(y)
dy + O(\mu^{(N-2)p-2})\\
= &  a_{N,p}\mathcal{A}_4
\mu^{(N-2)p-2}
\sum\limits_{j\neq i}^k\frac{d_{i}^{\frac{2N}{q+1}}
d_{j}^{\frac{N(p-1)}{q+1}}}{|\xi_{i}-\xi_{j}|^{(N-2)p-2}}
 + O(\mu^{(N-2)p-2}).
\end{align*}
If $h=1,\cdots,N$,
\begin{align}\label{uad}
 &  \sum_{j\neq i}^{k}
\int_{\Omega}U_i^{q-1}(x)U_j(x) \Psi_{jh}(x)dx\nonumber\\
= &  \sum_{j\neq i}^{k}
 \int_{\Omega}\Big(\mu_i^{-\frac{N}{q+1}}
 U_{1,0}(\frac{x-\xi_i}{\mu_i})\Big)^{q-1}
 \mu_j^{-\frac{N}{q+1}}U_{1,0}(\frac{x-\xi_j}{\mu_j}) \mu_j^{-\frac{N}{q+1}}\Psi_{1,0}^h(\frac{x-\xi_j}{\mu_j})dx\nonumber\\
= & \sum_{j\neq i}^{k} \mu_i^{N-\frac{N(q-1)}{q+1}}\mu_j^{-\frac{2N}{q+1}}
 \int_{\frac{\Omega-\xi_i}{\mu_i}}
 U^{q-1}_{1,0}(y)
U_{1,0}(\frac{\mu_iy+\xi_i-\xi_j}{\mu_j}) \Psi_{1,0}^h(\frac{\mu_iy+\xi_i-\xi_j}{\mu_j})dy\nonumber\\
= & \sum_{j\neq i}^{k} \mu_i^{N-\frac{N(q-1)}{q+1}}\mu_j^{-\frac{2N}{q+1}}
 \int_{\frac{\Omega-\xi_i}{\mu_i}}
 U^{q-1}_{1,0}(y)
U_{1,0}(\frac{\mu_iy+\xi_i-\xi_j}{\mu_j}) \partial_{y_h}U_{1,0} (\frac{\mu_iy+\xi_i-\xi_j}{\mu_j})dy.
\end{align}
Moreover, using Lemma \ref{ls}, (\ref{uad}) becomes
\[
\sum_{j\neq i}^{k}
\int_{\Omega}U_i^{q-1}(x)U_j(x) \Psi_{jh}(x)dx
=O(\mu^{(N-2)p-2}).
\]

Now, by (\ref{psi}) and (\ref{bubble}), we obtain  for $j=i$ and $h=0$,  we have
\begin{align*}
 &  q\int_\Omega\Big(\sum_{i=1}^kU_i(x)\Big)^{q-1}
 \Big[\mu^{\frac{Np}{q+1}}
\Big(\frac{b_{N,p}}{\gamma_N}\Big)^p
\widetilde{H}_{\mathbf{d},\boldsymbol{\xi}}(x)
+o(\mu^{\frac{Np}{q+1}})\Big]\Psi_i^0(x) dx\\
= &   q \mu^{\frac{Np}{q+1}}
\Big(\frac{b_{N,p}}{\gamma_N}\Big)^p  \int_\Omega\Big(\sum_{i=1}^kU_i(x)\Big)^{q-1}
\widetilde{H}_{\mathbf{d},\boldsymbol{\xi}}(x)
 \Psi_i^0(x) dx
 +o(\mu^{\frac{Np}{q+1}})\\
= &  q \mu^{\frac{Np}{q+1}}
\Big(\frac{b_{N,p}}{\gamma_N}\Big)^p
\sum_{i=1}^k\int_{B_i} U_i^{q-1}(x)\widetilde{H}_{\mathbf{d},\boldsymbol{\xi}}(x)
 \Psi_i^0(x) dx
 +o(\mu^{\frac{Np}{q+1}})\\
= &  q \mu^{\frac{Np}{q+1}}
\mu_i^{N-\frac{N(q-1)}{q+1}}
\mu_i^{-\frac{N}{q+1}}
\Big(\frac{b_{N,p}}{\gamma_N}\Big)^p
\sum_{i=1}^k\int_{\frac{B_i-\xi_i}{\mu_i}}
 U^{q-1}_{1,0}(y)
 \widetilde{H}_{\mathbf{d},\boldsymbol{\xi}}(\mu_iy+\xi_i)
\Psi_{1,0}^0(y)dy
 +o(\mu^{\frac{Np}{q+1}})\\
= &  q \Big(\frac{b_{N,p}}{\gamma_N}\Big)^p
\mu^{\frac{N(p+1)}{q+1}}
\sum_{i=1}^kd_i^{\frac{N}{q+1}}
\widetilde{H}_{\mathbf{d},\boldsymbol{\xi}}(\xi_i)
\int_{\mathbb{R}^N}
 U^{q-1}_{1,0}(y)
\Psi_{1,0}^0(y)dy
 +O(\mu^{(N-2)p-2})\\
 = &  \Big(\frac{b_{N,p}}{\gamma_N}\Big)^p \mathcal{A}_2
\mu^{\frac{N(p+1)}{q+1}}
\sum_{i=1}^kd_i^{\frac{N}{q+1}}
\widetilde{H}_{\mathbf{d},\boldsymbol{\xi}}(\xi_i)
 +O(\mu^{(N-2)p-2}),
\end{align*}
also for $h=1,\cdots,N$ and $j=i$, by (\ref{xy}) and (\ref{bubble}),  one has
\begin{align*}
& q\int_\Omega\Big(\sum_{i=1}^kU_i\Big)^{q-1}
 \Big[\mu^{\frac{Np}{q+1}}
\Big(\frac{b_{N,p}}{\gamma_N}\Big)^p
\widetilde{H}_{\mathbf{d},\boldsymbol{\xi}}(x)
+o(\mu^{\frac{Np}{q+1}})\Big]\Psi_{ih}dx\\
= &  q \mu^{\frac{Np}{q+1}}
\Big(\frac{b_{N,p}}{\gamma_N}\Big)^p
\sum_{i=1}^k\mu_i\int_{B_i} U_i^{q-1}(x)\widetilde{H}_{\mathbf{d},\boldsymbol{\xi}}(x)
 \Psi_{ih}(x) dx
 +o(\mu^{\frac{Np}{q+1}})\\
 = &   \mu^{\frac{Np}{q+1}}
\Big(\frac{b_{N,p}}{\gamma_N}\Big)^p
\sum_{i=1}^k\mu_i
\int_{B_i}
\widetilde{H}_{\mathbf{d},\boldsymbol{\xi}}(x)
\frac{\partial }{ \partial \xi_{ih}}U_i^q(x)dx \\
  = &   \mu^{\frac{Np}{q+1}}\sum_{i=1}^k \mu_i  \Big( \mu_i^{\frac{N}{q+1}} \frac{\partial }{ \partial \xi_{ih}}
     \int_{\frac{B_i-\xi_i}{ \mu_i }} U^q_i(y)
     \widetilde{H}_{\mathbf{d},\boldsymbol{\xi}}( \mu_i y+\xi_i) dy
     - \mu_i^{\frac{N}{q+1}}
     \int_{\frac{B_i-\xi_i}{ \mu_i }} U^q_i \frac{\partial
     \widetilde{H}_{\mathbf{d},\boldsymbol{\xi}}( \mu_i y+\xi_i)}{ \partial \xi_{ih}}dy\Big)\\
  = & \Big(\frac{b_{N,p}}{\gamma_N}\Big)^p
   \mathcal{A}_3  \sum_{i=1}^k\mu_i^{\frac{N(p+1)}{q+1}+1}\Big( \frac{\partial \widetilde{H}_{\mathbf{d},\boldsymbol{\xi}}(\xi_i)}{ \partial \xi_{ih}} + o( \mu_i )\Big)\\
  = &
  \Big(\frac{b_{N,p}}{\gamma_N}\Big)^p
  \mathcal{A}_3
  \mu^{\frac{N(p+1)}{q+1}+1}
\sum_{i=1}^kd_i^{\frac{N}{q+1}+1}
\partial_{\xi_{ih}}\tilde{\rho}(\xi_i)
+O(\mu^{(N-2)p-2}).
\end{align*}
Performing a similar computation for $j\neq i$
and  $h=0,\cdots,N$, one has
\[
q\int_\Omega\Big(\sum_{i=1}^kU_i\Big)^{q-1}\Big[\mu^{\frac{Np}{q+1}}
\Big(\frac{b_{N,p}}{\gamma_N}\Big)^p\widetilde{H}_{\mathbf{d},\boldsymbol{\xi}}(x)
+o(\mu^{\frac{Np}{q+1}})\Big]\Psi_{jh}dx
=O(\mu^{(N-2)p-2}).
\]
Collecting above the estimates,  the results follows.
\end{proof}

\subsection*{Acknowledgments}
The authors were supported by National Natural Science Foundation of China 11971392.

\end{document}